\documentclass[letterpaper,letterpaperjournal]{IEEEtran}
\usepackage{cite}

\usepackage{graphicx}
\usepackage{amssymb}
\usepackage[cmex 10]{amsmath}
\usepackage{array}
\usepackage{mdwmath}
\usepackage{mdwtab}
\usepackage{balance}
\usepackage{setspace}
\newtheorem{corollary}{Corollary}
\newtheorem{theorem}{Theorem}
\newtheorem{definition}{Definition}
\newtheorem{lemma}{Lemma}
\newtheorem{remark}{Remark}

\def\<{\leqslant}
\def\>{\geqslant}

\begin{document}
%\begin{document}
\title{Generalizing Negative Imaginary Systems Theory to    Include Free Body Dynamics:\\  Control of Highly Resonant  Structures with Free Body Motion }

\author{M.~A.~Mabrok,
    A.~G.~Kallapur, %~\IEEEmembership{Member,~IEEE,}
        I.~R.~Petersen %~\IEEEmembership{Fellow,~IEEE,}
   and~A.~Lanzon %~\IEEEmembership{Senioe Member,~IEEE,} % <-this % stops a space

%\thanks{This work was supported by the Australian Research Council.}
\thanks{M. Mabrok, A. Kallapur and I. R. Petersen are with the School of Engineering and Information Technology, University of New South Wales at
 the Australian Defence Force Academy, Canberra ACT 2600, Australia, email:abdallamath@gmail.com, abhijit.kallapur@gmail.com, i.r.petersen@gmail.com.}% <-this % stops a space
\thanks{A. Lanzon is with the Control Systems Centre, School of Electrical and Electronic Engineering, University of Manchester, Manchester M13 9PL,
United Kingdom, email:Alexander.Lanzon@manchester.ac.uk.}
\thanks{This research was supported by the Australian Research Council and the EPSRC.}

}

% make the title area
\maketitle


\begin{abstract}
Negative imaginary (NI) systems  play an important role in the
robust control of highly resonant flexible structures. In this
paper, a generalized NI system  framework is presented. A new NI
system definition is given, which allows for flexible structure
systems with colocated force actuators and position sensors, and
with free body motion. This definition extends the existing
definitions of NI systems. Also,   necessary and sufficient
conditions are provided for the stability of  positive feedback
control systems where the plant is NI according to the new
definition and the controller is strictly negative imaginary. The
stability conditions in this paper are given purely in terms of
properties of the plant and controller transfer function matrices,
although the proofs rely on state space techniques. Furthermore, the
stability conditions given are independent of the plant and
controller system order. As an application of these results, a case
study  involving the control of a flexible robotic  arm  with a
piezo-electric actuator and sensor is presented.
\end{abstract}

\section{Introduction}
%\doublespacing

  Flexible structure dynamics   arise  in many  areas such
as flexible robot manipulators  \cite{Wilson2002}, ground and
aerospace vehicles \cite{Harigae20}, atomic force microscopes (AFMs)
\cite{Bhikkaji2009,Mahmood2011} and other nano-positioning systems
\cite{Salapaka2002,Hulzen2010,Devasia2007,Diaz2012}. Flexible
structures can be  modeled as infinite dimensional distributed
parameter systems \cite{Ray1978281}. However,  finite dimensional
models  are often
used for the purpose of
 designing controllers  \cite{Curtain2009,Demetriou2003,Ray1978281,Preumont2011}.
  In designing controllers for these flexible systems, it
is important to consider  the effect of highly resonant modes. Such
resonant modes  are known to  adversely affect the stability and
performance of flexible  structure feedback control systems
\cite{Preumont2011,fanson1990,petersen2010}, and are often very
sensitive to changes in environmental variables. For instance,  a
small change in the environment of the system such as changing
temperature, can lead to  significant  changes in the resonant
frequencies of such systems. These changes in resonant frequencies
can lead to large changes in the gain and phase of the system
frequency response at a given frequency, which may lead to
instability or poor performance in the corresponding feedback
system.
In addition, highly resonant modes lead to vibrational
effects which limit the ability of control systems to achieve
desired levels of performance in many applications
 such as precision instrumentation, optical systems,
precision machine tools, wafer steppers, telescopes, and  atomic
force microscopes \cite{Preumont2011}. These issues arising from the
presence of highly resonant modes in flexible structures motivate
the need for tools to guarantee robust stability and performance in
flexible structure control systems.

One  common solution to issues of   robustness, stability, and
performance in the control of highly resonant flexible structures is
to  use force actuators combined with colocated measurements of
velocity, position, or acceleration
\cite{Preumont2011,fanson1990,petersen2010}.
 Colocated control with velocity measurements, known as
negative-velocity feedback, can be  used  to directly increase the
effective damping in the system, thereby facilitating the design of
controllers that can guarantee closed-loop stability in the presence
of parameter variations and unmodeled plant dynamics
\cite{Preumont2011}. Similarly, a class of  colocated controllers
with position measurements, known as  positive-position feedback
controllers, where velocity sensors are replaced with position
sensors, can also be used to increase damping in  flexible systems
as discussed in \cite{fanson1990,balas1979}.
 Also, positive-position
feedback controllers are robust against uncertainties in  resonant
frequencies as well as unmodeled plant dynamics, in a similar way to
negative-velocity feedback controllers
\cite{fanson1990,lanzon2008,petersen2010}.

The properties    of negative-velocity feedback has been studied
using passivity theory and  the theory of  positive real (PR) linear
time invariant (LTI) systems; e.g., see
\cite{anderson-bk1973,brogliato-bk2007}. However, PR theory cannot
be used directly when using position or acceleration measurements
\cite{petersen2010}. This  drawback is important in  applications to
the field of nanotechnology, especially for nano-positioning
systems, where position measurements are widely used; see e.g.,
\cite{Bhikkaji2009,Dong2007,Sebastian2005,Salapaka2002,Hulzen2010,Michellod2006,Devasia2007,Messenger2009,ElRifai2004,Diaz2012}.
Similar issues also arise in application to the area of robotics
where position measurements are also widely used.

  Lanzon and Petersen introduced a notion of  negative imaginary
(NI) systems  in \cite{lanzon2008,petersen2010}  for the robust
control of flexible structures with force  actuators combined with
position or acceleration sensors. %In the single-input single-output
(SISO) case, NI systems are defined by considering the properties of
the imaginary part of the system frequency response $G(j\omega) $
and requiring the condition $j\left( G(j\omega )-G(j\omega )^{\ast
}\right) \geq 0$ for all $\omega\in(0,\infty)$.
 The NI property
arises  in many practical systems. For example, such systems arise
when considering the transfer function from a force actuator to a
corresponding colocated position sensor (for instance, a
piezoelectric sensor) in a lightly damped structure
\cite{fanson1990,petersen2010,Bhikkaji2009,Yong2010,hagen2010}.
Another area  where the underlying  system dynamics are NI, in the area of
nano-positioning systems; see e.g.,
\cite{Bhikkaji2009,Dong2007,Sebastian2005,Salapaka2002,Hulzen2010,Michellod2006,Devasia2007,Messenger2009,ElRifai2004,Diaz2012}.
Also, the positive-position feedback control scheme in
\cite{fanson1990,goh1995}, can be considered using  the  NI
framework. Furthermore, other control methodologies in the
literature such as integral resonant control (IRC)
\cite{Pereira2011} and resonant feedback control
\cite{Halim2001,Mahmood2008}, fit into the NI framework and their
stability robustness properties can be explained by NI systems
theory.

\begin{figure}
  \centering\includegraphics[width=5cm]{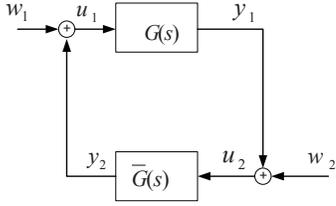}\\
  \caption{A negative-imaginary feedback control system. If the plant
transfer function matrix $G(s)$ is NI and the controller transfer
function matrix $\bar{G}(s)$ is SNI, then the positive-feedback
interconnection is internally stable if and only if the DC gain
condition, $\lambda_{max}(G(0)\bar{G}(0))<1,$ is satisfied.
}\label{conn:NI:SNI}
\end{figure}

The stability robustness of  interconnected NI systems  has been
studied in \cite{lanzon2008,petersen2010}. In these papers, it is
shown that a necessary and sufficient condition for the internal
stability of a positive-feedback control system  (see Fig.
\ref{conn:NI:SNI}) consisting of an NI plant  with transfer function
matrix $G(s)$ and a strictly negative imaginary (SNI) controller
with transfer function matrix $\bar{G}(s)$ is given by the DC gain
condition
\begin{equation}\label{DC:ian:alex:cond}
    \lambda_{max}(G(0)\bar{G}(0))<1,
\end{equation}
 where the notation $\lambda_{max}(\cdot)$ denotes the maximum
eigenvalue of a matrix with only real eigenvalues. This stability
result has been used in a number of  practical applications
\cite{hagen2010,Bhikkaji2009,Mahmood2011,Ahmed2011,Bhikkaji2012,Diaz2012}.
For example in \cite{hagen2010},    this   stability result is
applied to the problem of decentralized control of large vehicle
platoons. In \cite{Bhikkaji2009,Mahmood2011}, the NI stability
result is applied to  nanopositioning  in   an  atomic force
microscope. A positive position feedback control scheme based on the
NI stability result provided in \cite{lanzon2008,petersen2010} is
used to design a novel compensation method for a coupled
fuselage-rotor mode of a rotary wing unmanned aerial vehicle  in
\cite{Ahmed2011}. In \cite{Diaz2012}, an IRC scheme based on the
stability results provided in \cite{lanzon2008,petersen2010} is used
to design  an active vibration control  system for the mitigation of
human induced vibrations in light-weight civil engineering
structures, such as floors and footbridges via proof-mass actuators.
 An identification algorithm which enforces the
NI constraint is proposed in  \cite{Bhikkaji2012} for estimating
model parameters, following which  an Integral resonant controller
is designed for damping vibrations in  flexible structures. In
addition, it is shown in \cite{Schaft2011} that the class of linear
systems having NI transfer function matrices is closely related to
the class of linear Hamiltonian input-output systems. Also, an
extension of the NI systems theory  to infinite-dimensional systems
is presented in \cite{Opmeer2011}.

The NI framework presented in \cite{lanzon2008,petersen2010}
considers systems with poles in the open left half of the complex
plane. This theory has  been extended in \cite{xiong21010jor} to
include NI systems  with  poles in the closed left half of the
complex plane, except at the origin. Also,  further extensions  to
NI systems theory include the study of NI controller synthesis
\cite{song2010,Song2012b}, connections between NI systems analysis
and $\mu$-analysis \cite{Engelken2010a}, and conditions for robust
stability analysis of mixed NI and bounded-real classes of
uncertainties \cite{patra2010}. Furthermore, the concept of lossless
NI transfer functions is introduced  in \cite{Xiong2012},  an
algebraic approach to the realization of a lossless NI behavior is
presented in \cite{Rao2012}, and a spectral characterization of NI
descriptor systems is discussed  in  \cite{Benner2012}. The NI
systems theory can be extended to  nonlinear systems using the
concept  of counter-clockwise input-output dynamics as presented in
\cite{Angeli2006,Angeli2007,Padthe2005}. In \cite{Padthe2005}, a
sufficient conditions under which a semilinear Duhem model is
counter-clockwise is given, where the counter-clockwise input-output
system is restricted to  periodic input signals.  %Stability of
positive feedback interconnection for SISO linear case is provided
in \cite{Padthe2005}.
\begin{figure}
  \centering\includegraphics[width=8cm]{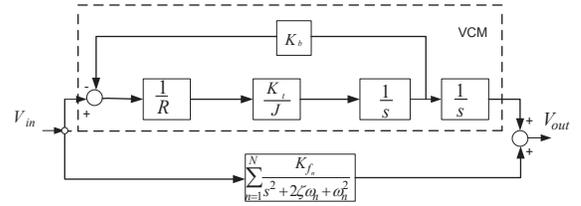}\\
  \caption{Block diagram of a disk-drive reader head system which includes a flexible structure driven by a
voice coil motor (VCM). The parameters $R, K_t , J$, and $K_b$ are
the  coil resistance, torque constant, moment of inertia, and back
electromotive force gain, respectively. Also, the parameters $ K_f ,
\zeta$, and $\omega_n$ are the flexible structure gain, damping
ratio, and natural frequency respectively. }\label{vcm:fig}
\end{figure}

Despite  generalizations of  the   NI systems  framework  presented
in \cite{xiong21010jor}, an important  class of systems, that  cannot
be captured by the existing NI systems  framework, corresponds to
flexible systems with free body motion. These  systems arise in
areas such as rotating flexible  spacecraft  \cite{Hugher1986},
rotary cranes \cite{Thomas1996}, robotics and  flexible link
manipulators \cite{Pereira2011,Mahmood2008,Choi1999}, and dual-stage
hard disk drives \cite{Yunfeng2003,Sang2001,Goh2001,Devasia2012}.
Flexible structures  with  free body motion lead to  dynamical
models including  poles at the origin, which is not covered in
earlier work on NI systems theory. In particular, the stability
condition \eqref{DC:ian:alex:cond} is not well defined in the case
of flexible structures with  free body motion which results in poles
at the origin, since in this case, the plant DC gain  $G(0)$ will be
infinite. However, control systems involving flexible structures
with free body motion arising  in these important application areas
still suffer from the stability and performance issues mentioned
above.
Thus we are motivated to extend the NI robust stability
theory developed in \cite{lanzon2008,petersen2010,xiong21010jor} so
that it can be applied to control systems involving highly resonant
flexible structures with free body motion.

Fig. \ref{vcm:fig} shows  a block diagram   of  a system which
includes a flexible structure with free body motion  that arises in
a problem of disk-drive control; see \cite{orpacharapan2004}. Here,
a voice coil motor (VCM) is used to actuate the arm of the reader
head. The free body motion of the reader head leads to a transfer
function from the input $V_{in}$ to the VCM output $V_{out}$ which
has poles at the origin. Furthermore, the  overall system  satisfies the  NI frequency response property  and
includes poles at the origin. However, the NI stability results
presented in \cite{lanzon2008,petersen2010,xiong21010jor} do not
allow for poles at the origin and cannot be applied to control
systems such as this   disk-drive control  system.

In this paper, we present a new generalized   definition of   NI
systems which allows for flexible structures with colocated force
actuators and position sensors and  with free body motion. This
 definition extends     the previous definitions of NI systems presented
in \cite{lanzon2008,petersen2010,xiong21010jor} to allow for up to
two poles at the origin. We also derive new generalized   stability
conditions for positive-feedback control systems involving  an NI
plant  and an SNI controller.
 %These results extend   previous  stability results presented in
%\cite{lanzon2008,petersen2010,xiong21010jor}, to allow for this new
%definition of NI systems. As such, the new stability results
%presented in this paper provides  definitive stability results for
%the positive feedback interconnection between an NI system and an
%SNI system.

As in \cite{lanzon2008,petersen2010,xiong21010jor}, the stability
conditions presented in this paper,  are  given purely  in terms of
properties of the plant and controller transfer function matrices,
although the proofs rely on  state space  techniques. Furthermore,
the stability conditions given are independent of the plant and
controller system order and can be stated without   using the fact
that the plant and the controller transfer function matrices are
rational. However, the proofs given in this paper only apply to the
rational case.

%The proofs of the  NI stability results given in this paper and  in the
%papers \cite{petersen2010,xiong21010jor,song2010} use the fact that
%NI systems can be transformed into PR systems and vice versa under
%certain  technical assumptions. However, this equivalence is not
%complete as mentioned in \cite{petersen2010,xiong21010jor,song2010}.
%For instance, such a transformation applied to an  SNI system always
%leads to a non-strict PR system and thus SNI systems cannot be
%transformed into strictly positive real systems. Hence, the
%passivity theorem \cite{anderson-bk1973,brogliato-bk2007} cannot be
%used to prove  the stability of the positive feedback
%interconnection of an NI and an SNI system and a more sophisticated
%methodology is required to prove the stability results of this paper
%as well as the previous results in
%\cite{petersen2010,xiong21010jor,song2010}.

Preliminary conference versions of the    stability results
presented in this paper   were presented in
\cite{Mabrok2012,Mabrok2011b}. However, in this paper,  much more
general versions of these stability results are presented in Theorems 
\ref{min:result} - \ref{min:result:clo2.1} and Corollaries
\ref{min:result:clo3}, \ref{min:result:clo2},
which allow for the existence of free body motion in some  but not
all input-output channels. This is important since multivariable
control systems involving flexible structures with free body motion
usually include free body motion in some but not all input-output
channels. %This includes the  case study considered in Section
%\ref{sec:example} which involves the control of a flexible robotic
%arm with a piezo-electric actuator and sensor along with a torque
%actuator and hub angle sensor.
Also, this paper includes a case  study involving the control of a flexible robotic arm, which  has not been considered  in the previous conference
versions of the paper.%The
%results presented in \cite{Mabrok2012} are equivalent to Corollary
%\ref{min:result:clo2} %and Corollary \ref{min:result:clo4}
%in this
%paper, which are special cases of the main results reported here for
%the first time.

This paper  is further organized as follows: Section
\ref{sec:preliminaries} recalls  the  existing  definition for NI
systems and outlines the notation  that will be used in the rest of
the paper. Section \ref{sec:main-results} introduces the     new
generalized definition for NI systems, which allows for  systems
with free body dynamics. Also in this section, we present the  main
stability results in Theorems \ref{min:result} - \ref{min:result:clo2.1} and Corollaries
\ref{min:result:clo3}-\ref{min:result:clo2}. Section
\ref{sec:example} presents  a case study, which involves  a flexible
robotic arm, as an application of the NI theory presented in this
paper. The paper   is concluded with a summary and remarks on future
work in Section \ref{sec:conclusion}. All proofs of the presented
theorems, lemmas and corollaries are given in the Appendix.

\section{Preliminaries and Notation }
\label{sec:preliminaries}

In this section, we recall  the  existing   definitions of NI and
SNI systems as given in \cite{xiong21010jor} for systems with poles
in the closed  left half of the complex plane, except at the origin.
We will use this existing definition of SNI  (first introduced in
\cite{lanzon2008}) systems but in the next section we will present
our new definition of generalized NI systems. We also define
notation used to describe positive feedback interconnections and
internal stability, which will be used to present  the main results
in this paper.

Consider the following LTI system,
\begin{align}
\label{eq:xdotn}
&\dot{x}(t) = A x(t)+B u(t), \\
\label{eq:yn} &y(t) = C x(t)+D u(t),
\end{align}%
where $A \in \mathbb{R}^{n \times n},B \in \mathbb{R}^{n \times m},C
\in \mathbb{R}^{m \times n},$ $D \in \mathbb{R}^{m \times m},$ and
with the  square transfer function matrix $G(s)=C(sI-A)^{-1} B+D$.
The transfer function matrix $G(s)$ is said to be strictly proper if
$G(\infty)=D=0$. We will use the notation $
\begin{bmatrix}
\begin{array}{c|c}
A & B \\ \hline C & D
\end{array}
\end{bmatrix}$ to denote the state space realization
\eqref{eq:xdotn}, \eqref{eq:yn}.

%\begin{definition}\label{Old:Df:NI}
 The existing definition of NI systems states that a square transfer function matrix $G(s)$ is NI if the following conditions are satisfied \cite{xiong21010jor}:
\begin{enumerate}
\item[1.] $G(s )$ has no pole at the origin and in $Re[s]>0$.
\item[2.] The corresponding frequency response $G(j\omega )$ is such that
\begin{equation*}
    j\left( G(j\omega )-G(j\omega )^{\ast }\right) \geq 0,
\end{equation*} for all $\omega >0$ where $j\omega$ is not a pole of $G(s
)$.
\item[3.] If $j\omega_{0}$ with $\omega_0>0$ is a pole of $G(s)$, it is at most a simple pole and the residue matrix $K_{0}= \lim_{ j\rightarrow j\omega_{0}}(s-j\omega_{0})sG(s)$ is positive semidefinite Hermitian.
\end{enumerate}
%\end{definition}

\begin{definition}\cite{xiong21010jor}
A square transfer function matrix $G(s)$ is SNI if  the
following conditions are satisfied:
\begin{enumerate}
\item ${G}(s)$ has no pole in $Re[s]\geq0$.
\item For all $\omega >0$, $j\left( {G}(j\omega )-{G}(j\omega )^{\ast }\right) > 0$.
\end{enumerate}
\end{definition}

Now, consider a positive feedback interconnection between an NI
system  with transfer function matrix  $G(s)$ and an SNI system with
transfer function matrix $\bar{G}(s)$ as shown in Fig.
\ref{conn:NI:SNI}. Also, suppose that the transfer function matrix
$G(s)$ has a minimal state space realization  $
\begin{bmatrix}
\begin{array}{c|c}
A & B \\ \hline C & D
\end{array}
\end{bmatrix},$ and $\bar{G}(s)$ has a minimal state space realization  $
\begin{bmatrix}
\begin{array}{c|c}
\bar{A} & \bar{B} \\ \hline \bar{C} & \bar{D}
\end{array}
\end{bmatrix}.$ Furthermore, it is assumed that the matrix $I-D\bar{D}$   is nonsingular. Then the closed system has a system matrix given by

\begin{small}
\begin{align}
\breve{A} =\begin{bmatrix} A+B\bar{D}(I-D\bar{D})^{-1}C &
B\bar{C}+B\bar{D}(I-D\bar{D})^{-1}D\bar{C} \\
\bar{B}(I-D\bar{D})^{-1}C & \bar{A}+\bar{B}(I-D\bar{D})^{-1}D\bar{C}%
\end{bmatrix}.%
\label{mat:A:CL}%\\
%\breve{B} =&%
%\begin{bmatrix}
%B\bar{D} \\
%\bar{B}%
%\end{bmatrix}%
%(I-D\bar{D})^{-1}, \nonumber\\
%\breve{C} =&(I-D\bar{D})^{-1}%
%\begin{bmatrix}
%C & D\bar{C}%
%\end{bmatrix},\nonumber
%\\ \nonumber
%\breve{D} =&(I-D\bar{D})^{-1}.\nonumber
\end{align}
\end{small}%
Moreover, the positive feedback interconnection between $G(s)$ and
$\bar{G}(s)$ as shown in Fig. \ref{conn:NI:SNI} and denoted
$[G(s),\bar{G}(s)]$ is said to be  internally stable if  the
closed-loop system  matrix $\breve{A}$ in \eqref{mat:A:CL} is
Hurwitz; e.g., see \cite{Glover1996}.

\section{Main results}
\label{sec:main-results}

The main contribution  of this paper   is
 a generalization of the    framework for NI systems
presented in \cite{xiong21010jor}. We  introduce a new definition of
NI systems that will allow for systems with  free body dynamics.
%This is achieved by relaxing the condition requiring  the absence of
%poles at the origin in the NI definition presented in
%\cite{xiong21010jor}. Secondly,
This generalized  definition will be
used in  a new set of stability conditions that will allow for   NI
systems with free body motion to be included into the framework of
NI systems theory. Henceforth, when a system is said to be NI, we
will mean NI as defined below, not NI as defined in earlier papers.
%and recalled in the text description above.

\begin{definition}\label{Def:NI}
A square transfer function matrix $G(s)$ is NI if the following
conditions are satisfied:
\begin{enumerate}
\item $G(s)$ has no pole in $Re[s]>0$.
\item For all $\omega >0$ such that $j\omega$ is not a pole of $G(s
)$,
$
    j\left( G(j\omega )-G(j\omega )^{\ast }\right) \geq 0.
$
\item If $s=j\omega _{0}$ with $\omega _{0}>0$ is a pole of $G(s)$, then it is a simple pole and the residue matrix $K=\underset{%
s\longrightarrow j\omega _{0}}{\lim }(s-j\omega _{0})jG(s)$ is
Hermitian and  positive semidefinite.
 \item If $s=0$ is a pole of $G(s)$, then
$\underset{s\longrightarrow 0}{\lim }s^{k}G(s)=0$ for all $k\geq3$
and $\underset{s\longrightarrow 0}{\lim }s^{2}G(s)$ is Hermitian and
positive semidefinite.
\end{enumerate}
\end{definition}
Here, $G(j\omega )$ is the  frequency response corresponding to the
transfer function $G(s )$.
%\begin{equation*}
%    j\left( G(j\omega )-G(j\omega )^{\ast }\right) \geq 0,
%\end{equation*} for all $\omega >0$ where $j\omega$ is not a pole of $G(s
%)$.
%
%\begin{remark}
%Note that the number of poles at the origin that it is allowed here
%is maximum two poles. This is due to the fact that we more than two
%poles at the origin will have more than $180$ phase shift.
%
%\end{remark}
Unlike the NI definition  presented in \cite{xiong21010jor},
Definition \ref{Def:NI} allows for  poles at the origin. In this
case, we cannot use the existing stability results presented in
\cite{petersen2010,xiong21010jor,song2010},  because the stability
condition in \eqref{DC:ian:alex:cond} is not  defined.
 The inclusion of poles at the origin extends the NI systems theory
to include  flexible systems with free body dynamics. 
In order to derive a  set of stability conditions that allow for NI
systems with free body motion, we define the following constant
matrices for a given $m\times m$ NI transfer function matrix $G(s):$
\begin{align}
G_2&=\underset{s\longrightarrow 0}{\lim }s^2G(s),\nonumber \\
G_1&=\underset{s\longrightarrow 0}{\lim }s \left(G(s)-\frac{G_2}{s^2}\right),\nonumber \\
G_0&=\underset{s\longrightarrow
0}{\lim}\left(G(s)-\frac{G_2}{s^2}-\frac{G_1}{s}\right)
\label{f_G0}.
\end{align}

These matrices  are the first three  coefficients in the Laurent
series expansion of the transfer function $G(s)$. These matrices
carry information about properties of the free body motion  of   the
system under consideration and  will be used in stability conditions
for the positive feedback interconnection of NI and SNI systems.
Note that the DC gain condition \eqref{DC:ian:alex:cond} cannot be
defined for an NI system with transfer function matrix $G(s)$ unless
$G_2=G_1=0$, which reduces to the case where the dynamical system
has no free body motion.
 From Condition 4) in   Definition \ref{Def:NI},   the matrix
$G_2$ is required  to be Hermitian and positive semidefinite.
 Hence, it follows (e.g., see \cite{Koeber2006}) that if $G_2\neq0$, it  can be decomposed
 in the form
 \begin{align}
 G_{2}=JJ^{T},\label{JJ-m}
 \end{align}
 where $J $ is a full column rank matrix.

We now  present   conditions for the stability of a positive
feedback control system involving an NI plant with  free body
motion. These conditions are stated using the quantities defined in
\eqref{f_G0}. First, we define the $2m\times 2m$
Hankel matrix $\Gamma$ as
\begin{align}
 \Gamma=\begin{bmatrix} G_1 & G_2
\\G_2 &0
\end{bmatrix}.\label{Hankel_m}
\end{align}
Suppose that $\Gamma\neq0$. Using the  singular value decomposition
(SVD), we can decompose the Hankel matrix $\Gamma$ as
\begin{align}
 \Gamma&=\begin{bmatrix}H_1&H_2
\end{bmatrix}\begin{bmatrix}S&0
\\ 0&0
\end{bmatrix}\begin{bmatrix} V_1^T
\\ V_2^T
\end{bmatrix}\notag\\&=H_1SV_1^T=UV_1^T=\begin{bmatrix} U_1
\\U_2
\end{bmatrix}V_1^T,\label{SVD-H-1}
\end{align}
where $\begin{bmatrix}H_1&H_2
\end{bmatrix},\begin{bmatrix} V_1^T
\\ V_2^T
\end{bmatrix}$ are unitary matrices, $S>0, U=H_1S\in \mathbb{R}^{2m \times \tilde{n}}$, $U_1\in
\mathbb{R}^{m \times \tilde{n}}$, $U_2\in \mathbb{R}^{m \times
\tilde{n}}$ and the matrices $U$ and $V_1$ each have orthogonal columns.
Furthermore, we can decompose the $\tilde{n}\times \tilde{n}$ matrix
$U_{1}^T U_2$ using the  SVD as
\begin{align}
U_{1}^T U_2=\hat{U} \hat{S}\hat{V}^T=\hat{U}\begin{bmatrix}S_1&0
\\ 0&0
\end{bmatrix}\begin{bmatrix} \hat{V_1}^T
\\ \hat{V_2}^T
\end{bmatrix}\label{SVD-uu},
\end{align}
where  $\hat{U}\in \mathbb{R}^{\tilde{n} \times \tilde{n}}$  and
$\hat{V}\in \mathbb{R}^{\tilde{n} \times \tilde{n}}$ are orthogonal
matrices, $\hat{V_2}\in \mathbb{R}^{\tilde{n} \times \check{n}}$ and
$S_1>0$.

We now introduce some  notation which will be used throughout the
paper. Given matrices $X\in \mathbb{R}^{m \times m}$ and $Y \in
\mathbb{R}^{m\times \check{n}}$ such that $\det(Y^{T}XY)\neq0$, then
the matrix valued function $\mathcal{P}(X,Y)$ is defined by
\begin{align}
\mathcal{P}(X,Y)\triangleq X-XY\left( Y^{T}XY\right) ^{-1}Y^{T}%
X.\label{newP-N-f-m}
 \end{align}
Using this notation, we define the matrix
 \begin{align}N_{f}=\mathcal{P}(\bar{G}(0),F)
,\label{N-f-m}
 \end{align}
 where the $m\times \check{n}$ matrix $F$  is given by
\begin{align}\label{F-SVD1}
F=U_{1}\hat{V_2},
\end{align} and we will assume that $\det(F^{T}\bar{G}(0)F)\neq0$.

We will use the following condition in the theorem which follows:
\begin{equation}\label{Eq:Con1:th1}
    F^{T}\bar{G}(0)F<0.
\end{equation}
Also, for the case in which $N_f$ is positive semidefinite, we will
use the condition
\begin{equation}\label{Eq:Con2:th1}
    I-N_{f}^{\frac{1}{2}}G_{0}N_{f}^{\frac{1}{2}}-N_{f}^{\frac{1}{2}}G_{1}J(J^{T}J)^{-2}J^{T}G_{1}^{T}N_{f}^{\frac{1}{2}}>0.
\end{equation}
Moreover, for the case in which $N_f$ is negative  semidefinite, we
will use the condition
\begin{equation}\label{Eq:Con3:th1}
    \det(I+\tilde{N}_{f}G_{0}\tilde{N}_{f}+\tilde{N}_{f}G_{1}J(J^{T}J)^{-2}J^{T}G_{1}^{T}\tilde{N}_{f})\neq0.
\end{equation}
Here, $\tilde{N}_{f}=(-N_f)^{\frac{1}{2}}$ and  matrices
$G_1,G_0,J,N_f$ and $F$ are defined in \eqref{f_G0},
\eqref{JJ-m}, \eqref{N-f-m}, and \eqref{F-SVD1} respectively.  Also,
$(\cdot)^{\frac{1}{2}}$ denotes the square root of a positive
semidefinite matrix.

%========================
%========================
%========================
%========================
%========================
%Theorem 11111111111111
%========================
%========================
%========================
%========================
%========================

The following theorem is our first  main stability result for the
case in which $G_2\neq0$.  That is, the system has double poles at the origin. %This corresponds to the case of plants
%which have free body motion without friction being present.

\begin{theorem}\label{min:result}
Suppose that the square transfer function matrix $G(s)$ is strictly
proper and NI with $G_2\neq0$, and the transfer function matrix
$\bar{G}(s)$ is SNI. Also, suppose that the matrix
$F^{T}\bar{G}(0)F$ is  non-singular. If $N_f$ is positive
semidefinite, then the closed-loop positive-feedback interconnection
between $G(s)$ and $\bar{G}(s)$ as shown in  Fig.
 \ref{conn:NI:SNI} is internally
stable   if and only if conditions \eqref{Eq:Con1:th1} and
\eqref{Eq:Con2:th1} are satisfied. Furthermore, if $N_f$ is negative
semidefinite, then the closed-loop positive-feedback interconnection
between $G(s)$ and $\bar{G}(s)$ is internally stable   if and only
if conditions \eqref{Eq:Con1:th1} and \eqref{Eq:Con3:th1} are
satisfied.
%\begin{enumerate}
%\item $I-\tilde{N}_{f}G_{0}\tilde{N}_{f}-\tilde{N}_{f}G_{1}J(J^{T}J)^{-2}J^{T}G_{1}^{T}\tilde{N}_{f}>0,$
%\item $F^{T}\bar{G}(0)F<0,$
%%\item $\det(A+B\bar{G}(0)C)\neq0.$
%\end{enumerate}
\end{theorem}
The proof of this and subsequent theorems and corollaries  are  presented in Appendix B.%\ref{appB}.

%Theorem \ref{min:result} gives a    stability result for  the
%positive-feedback interconnection of an NI  system including  free
%body motion and an SNI system for the case in which $G_2\neq0$.
 We
now  present a corollary to this theorem which considers the special
case in which none of the free body modes of the plant have
frictional force  present; i.e., $G_1=0$. In order to present this
corollary, we define the matrix $N_2$ as follows:
\begin{align}
 N_{2}=\mathcal{P}(\bar{G}(0),J),%\bar{G}(0)-\bar{G}(0)J\left( J^{T}\bar{G}(0)J\right) ^{-1}J^{T}%
%\bar{G}(0),
\label{N-2-g10}
\end{align}
where we assume that the matrix $J^{T}\bar{G}(0)J$ is non-singular.

%\begin{assumption}
%\label{assm:siso2:crol3} The matrix the matrix $J^{T}\bar{G}(0)J$ is
%assumed to be non-singular and the matrix $N_2$ in \eqref{N-2-g10}
%is assumed to be either positive semidefinite or negative
%semidefinite.
%\end{assumption}

%\begin{remark}
%\label{rem:1:crol2}
%Under this assumption, we define a complex matrix $\tilde{N}_{2}$  so that
%\begin{equation}\label{mode:Nf:crol2}
% \tilde{N}_{2}^2=N_2.
%\end{equation}
% In the case of which $N_2\geq0$, then $\tilde{N}_{2}=N_2^{\frac{1}{2}}$. In the case of which $N_2\leq0$, then $\tilde{N}_{2}=i(-N_2)^{\frac{1}{2}}$.
%\end{remark}
We will use the following condition in the next corollary, which
 corresponds to condition
\eqref{Eq:Con1:th1} in Theorem \ref{min:result}:
\begin{equation}\label{Eq:Con1:crol3}
    J^{T}\bar{G}(0)J<0.
\end{equation}
Also, for the case in which $N_2$ is positive semidefinite, we will
use the following  condition which corresponds to condition
\eqref{Eq:Con2:th1} in Theorem \ref{min:result}:
\begin{equation}\label{Eq:Con2:crol3}
    I-N_{2}^{\frac{1}{2}}G_{0}N_{2}^{\frac{1}{2}}>0.
\end{equation}
Moreover, for the case in which $N_2$ is negative  semidefinite, we
will use the following  condition which corresponds to condition
\eqref{Eq:Con3:th1} in Theorem \ref{min:result}:
\begin{equation}\label{Eq:Con3:crol3}
    \det(I+\tilde{N}_{2}G_{0}\tilde{N}_{2})\neq0,
\end{equation}
where $\tilde{N}_{2}=(-N_2)^{\frac{1}{2}}$.

%========================
%========================
%========================
%========================
%========================
%Corollary 11111111111111
%========================
%========================
%========================
%========================
%========================

\begin{corollary}\label{min:result:clo3}
Suppose that the transfer function matrix  $\bar{G}(s)$ is SNI and
 the strictly proper transfer function matrix  $G(s)$
is  NI with $G_1=0$ and $G_2\neq0$. Also, suppose that the matrix
$J^{T}\bar{G}(0)J$ is  non-singular.  If $N_2$ is positive
semidefinite, then the closed-loop positive-feedback interconnection
between $G(s)$ and $\bar{G}(s)$  is internally stable   if and only
if conditions \eqref{Eq:Con1:crol3} and \eqref{Eq:Con2:crol3} are
satisfied. Furthermore, if $N_2$ is negative semidefinite, then the
closed-loop positive-feedback interconnection between $G(s)$ and
$\bar{G}(s)$ is internally stable   if and only if conditions
\eqref{Eq:Con1:crol3} and \eqref{Eq:Con3:crol3} are satisfied.
\end{corollary}

%========================
%========================
%========================
%========================
%========================
%Corollary 22222222222222
%========================
%========================
%========================
%========================
%========================
The following theorem imposes some extra conditions on the matrix
$G_2$ which enables us to relax the sign definiteness condition on
the matrix $N_2$. This then leads to a simplified stability
condition.
\begin{theorem}\label{min:result:clo3.1}
Suppose that the transfer function matrix  $\bar{G}(s)$ is SNI and
 the strictly proper transfer function matrix  $G(s)$
is  NI with $G_1=0$ and $G_2\neq0$. Also,  suppose that
$\mathcal{N}(G_2)\subseteq\mathcal{N}(G_0^T)$, where
$\mathcal{N}(\cdot)$ denotes the null space of a matrix. Then the
closed-loop positive-feedback interconnection between $G(s)$ and
$\bar{G}(s)$ is internally stable if and only if condition
\eqref{Eq:Con1:crol3}  is satisfied.  %\hfill $\blacksquare$
\end{theorem}

%========================
%========================
%========================
%========================
%========================
%Corollary 3333333333333
%========================
%========================
%========================
%========================
%========================
%An important special case of the above theorems  is the following
%corollary.
%\begin{corollary}\label{min:result:clo4}
%Suppose that the transfer function matrix  $\bar{G}(s)$ is SNI and
% the strictly proper transfer function matrix  $G(s)$
%is  NI with $G_1=0$ and $G_2>0$.
% Then, the closed-loop positive-feedback interconnection
%between $G(s)$ and $\bar{G}(s)$ is internally stable if and only if
%$\bar{G}(0)<0.$
% %\hfill $\blacksquare$
%\end{corollary}

%We also present special cases of this theorem including the case when $G_2=0$
%in the form of Corollaries \ref{min:result:clo1}-\ref{min:result:clo5}, which are given below.
  %which  are simplifications of the main results in Theorem
%%\ref{min:result}. These can be  directly  applied in some cases such
%as applications where frictional  force is considered.Also, in the
%case when  the free body motion exists in all input-output
%directions.

%In the SISO case, as  in the example shown in Fig. \ref{vcm:fig},
%the effect of free body motion may  lead to   a double integrator in
%the plant model. However, if frictional force is present, then the
%effect of the free body motion leads to only  a single  integrator.
%In the general MIMO case, this corresponds to the situation in which
%$G_2=\underset{s\longrightarrow 0}{\lim}s^2G(s)=0$.
In Theorem \ref{min:result:clo1}, Theorem \ref{min:result:clo2.1}
and Corollary \ref{min:result:clo2}, we consider   cases which
correspond to free body motion with frictional force present. As in
Theorem \ref{min:result}, these   cases allow for fact that the free
body motion may  not be present in all input-output channels.

In order to present Theorem \ref{min:result:clo1} and  Theorem
\ref{min:result:clo2.1}, %and Corollary  \ref{min:result:clo2},
suppose that $G_1\neq0$ and $G_2=0$. This corresponds to  the  case when the
system has a single pole at the origin. Then we consider the
following SVD decomposition of the matrix $G_1$ defined in
\eqref{f_G0}:

\begin{align}
G_1=\begin{bmatrix}\tilde{F}_1&\tilde{F}_2
\end{bmatrix}\begin{bmatrix}S_2&0
\\ 0&0
\end{bmatrix}\begin{bmatrix} V_1^T
\\ V_2^T
\end{bmatrix}=F_1V_1^T,\label{C-D-1-f_G1}
\end{align}
where $S_2>0$, and the matrices  $F_1=\tilde{F}_1S_2$ and $V_1$ each have
orthogonal columns. Also, we define the matrix $N_1$ as follows:

\begin{align}
 N_{1}=\mathcal{P}(\bar{G}(0),F_1),%\bar{G}(0)-\bar{G}(0)F_1\left( F_{1}^{T}\bar{G}(0)F_1\right) ^{-1}F_{1}^{T}%
%\bar{G}(0),
 \label{N-1-g20}
\end{align}
where the matrix $F_{1}^{T}\bar{G}(0)F_1$ is assumed to be non-singular.

%\begin{assumption}
%\label{assm:siso2:crol1} We assume  that $G_1\neq0$. Also, the
%matrix the matrix $F_1^{T}\bar{G}(0)F_1$ is assumed to be
%non-singular and the matrix $N_1$ in \eqref{N-1-g20} is assumed to
%be either positive semidefinite or negative semidefinite.
%\end{assumption}

We will use the following condition in Theorem \ref{min:result:clo1}
and Corollary \ref{min:result:clo2}  which
 corresponds to condition
\eqref{Eq:Con1:th1} in Theorem \ref{min:result}:
\begin{equation}\label{Eq:Con1:crol1}
    F_1^{T}\bar{G}(0)F_1<0.
\end{equation}
For the case in which $N_1$ is positive semidefinite, we also will use
the following  condition which corresponds to condition \eqref{Eq:Con2:th1} in Theorem \ref{min:result}:
\begin{equation}\label{Eq:Con2:crol1}
    I-N_{1}^{\frac{1}{2}}G_{0}N_{1}^{\frac{1}{2}}>0.
\end{equation}
Moreover, for the case in which $N_1$ is negative  semidefinite, we
will use the following  condition which corresponds to condition \eqref{Eq:Con3:th1} in Theorem \ref{min:result}:
%\begin{equation}\label{Eq:Con2:crol1}
\begin{equation}\label{Eq:Con3:crol1}
    \det(I+\tilde{N}_{1}G_{0}\tilde{N}_{1})\neq0,
\end{equation}
where  $\tilde{N}_{1}=(-N_1)^{\frac{1}{2}}$.

%========================
%========================
%========================
%========================
%========================
%Theorem 2222222222222222
%========================
%========================
%========================
%========================
%========================
\begin{theorem}\label{min:result:clo1}
Suppose that the transfer function matrix $\bar{G}(s)$ % with a
is SNI and the  strictly proper transfer function matrix  $G(s)$ is
NI with $G_2=0$ and $G_1\neq0$. Also, suppose that  the matrix
$F_1^{T}\bar{G}(0)F_1$ non-singular.  If $N_1$ is positive
semidefinite, then the closed-loop positive-feedback interconnection
between $G(s)$ and $\bar{G}(s)$  is internally stable   if and only
if conditions \eqref{Eq:Con1:crol1} and \eqref{Eq:Con2:crol1} are
satisfied. Furthermore, if $N_1$ is negative semidefinite, then the
closed-loop positive-feedback interconnection between $G(s)$ and
$\bar{G}(s)$  is internally stable   if and only if conditions
\eqref{Eq:Con1:crol1} and \eqref{Eq:Con3:crol1} are satisfied.
\end{theorem}
%\begin{flushright}
%$\blacksquare$
%\end{flushright}

%========================
%========================
%========================
%========================
%========================
%Corollary 222222.....111111
%========================
%========================
%========================
%========================
%========================
The following theorem imposes some extra conditions on the matrix
$G_1$ which enables us to relax the sign definiteness condition on
the matrix $N_1$. This then leads to a simplified stability
condition.
\begin{theorem}\label{min:result:clo2.1}
Suppose that the transfer function matrix  $\bar{G}(s)$ is SNI and
 the strictly proper transfer function matrix  $G(s)$
is  NI with $G_2=0$ and $G_1\neq0$. Also,  suppose that
$\mathcal{N}(G_1^T)\subseteq\mathcal{N}(G_0^T)$.  Then the
closed-loop positive-feedback interconnection between $G(s)$ and
$\bar{G}(s)$  is internally stable   if and only if condition
\eqref{Eq:Con1:crol1}  is satisfied.
\end{theorem}

%========================
%========================
%========================
%========================
%========================
%Corollary 222222.....22222222222222
%========================
%========================
%========================
%========================
%========================
The following corollary presents an important   special case of
 Theorem  \ref{min:result:clo3.1} and \ref{min:result:clo2.1}.
\begin{corollary}\label{min:result:clo2}
Suppose that the transfer function matrix  $\bar{G}(s)$ is SNI and
the strictly proper transfer function matrix  $G(s)$ is  NI with
either  $G_2=0$ and $G_1$ invertible or $G_1=0$ and $G_2>0$. Then,
the closed-loop positive-feedback interconnection between $G(s)$ and
$\bar{G}(s)$ is internally stable if and only if $\bar{G}(0)<0.$
\end{corollary}

%5555555555555555555555
%555555555555555555555

\begin{remark}
\label{rem:2}  The case where $G_2=0$ and $G_1=0$ corresponds  to the existing stability results presented in
\cite{petersen2010,xiong21010jor,song2010}. In this case, the
stability condition  reduces to $\lambda_{max}(\bar{G}(0)G(0))<1.$
This condition  can be obtained  from \eqref{Eq:Con2:crol1}  using
the fact $N_1=\bar{G}(0)$ in this case. Also, we require the
assumption $\bar{G}(0)>0$. Hence,
\begin{align*}
&I -N_1^{\frac{1}{2}}G_0N_1^{\frac{1}{2}}>0,\notag\\
\Leftrightarrow &N_1^{-1} -G_0>0,\notag\\
\Leftrightarrow &\lambda_{max}(\bar{G}(0)G_0)<1.
\end{align*}
Note that using a similar argument to the proof of Theorem
\ref{min:result:clo1}, we can obtain a similar result under the
assumption that $\bar{G}(0)<0.$
\end{remark}

\section{Case Study: Control of Flexible robotic arm }\label{sec:example}

In this section, we present an application of the stability results presented in this paper  to the control of
a flexible robotic arm  system. The robotic arm is pinned to a motor at one end.
For the purposes of modeling the flexible robotic arm, we use an equivalent slewing beam model as depicted in Fig. \ref{fig:beam}; see  \cite{pota1995}.
\begin{figure}[hpb]
\centering{\includegraphics[width=7cm]{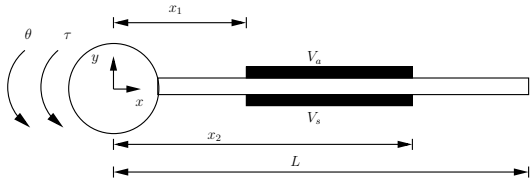}} \caption{Schematic diagram  of
the  slewing beam equivalent to the  robotic arm.} \label{fig:beam}
\end{figure}
 The motor allows the
robotic arm to traverse in the vertical plane. Two piezoelectric
patches are attached to the arm on  either side. Here, one
piezoelectric patch acts as an actuator while the other is a sensor.
The robotic arm system has  two inputs and two outputs: the inputs
are  the voltage $V_a$ applied to the piezoelectric actuator and the
torque $\tau$ applied by the motor, whereas the outputs are the
voltage $V_s$ produced  by the piezoelectric sensor and the motor
hub angle $\theta$.  The fact that this system involves colocated
``force'' actuators and ``position" sensors   indicates that the
system will be NI; e.g., see \cite{petersen2010}.

%In the current setup we neglect the tip position.
%
\subsection{Mathematical model for the robotic arm}
The beam in Fig. \ref{fig:beam} is modeled using the Bernoulli-Euler equations of motion for a beam with actuating and sensing piezoelectric elements as in \cite{pota1995}:
\begin{small}
\begin{equation}
    \frac{\partial^2}{\partial x^2}
        \left[
            E I
            \frac{\partial^2}{\partial y^2} y(x,t)
            -
            C_a V_a(x,t)
        \right]
        +
        \rho A
        \frac{\partial^2}{\partial t^2} y(x,t)
    = 0.
\label{eq:bernoulli}
\end{equation}
\end{small}
Here, $E$ is  Young's modulus and $I$ is the second
moment of inertia of the beam,  $\rho$ is the density of the beam,
and $A$ is the area of the composite beam. If the thickness of the
piezoelectric films are comparable to the thickness of the beam,
then the products $EI$ and $\rho A$ would be different in the
laminated and non-laminated areas of the beam. However, since
piezoelectric films used in practical applications are often thin
compared to the thickness of the beam, these differences will be
neglected.  Assuming that  the products $EI$ and $\rho A$ are
uniform over the length of the beam simplifies the modeling
procedure.

Now we consider various boundary conditions in modeling the beam. These are given as
\begin{align}
\label{eq:bc-time1}
    &y(0,t) = 0, \\
    &EI
        \frac{\partial^2}{\partial x^2} y(0,t)
        -
        I_h \frac{\partial^3}{\partial t^2 \partial x} y(0,t)
        +
        \tau(t)
        = 0, \\
    &EI
        \frac{\partial^2}{\partial x^2} y(L,t)
        +
        I_t \frac{\partial^3}{\partial t^2 \partial x} y(L,t)
        = 0, \\
\label{eq:bc-time2}
    &EI
        \frac{\partial^3}{\partial y^3} y(L,t)
        -
        M_t \frac{\partial^3}{\partial t^2 \partial x} y(L,t)
        = 0.
\end{align}
Here, $M_t$ and $I_t$ are the mass and inertia of the tip, which
will be neglected in this paper.Also, \eqref{eq:bc-time1} represents
the inability of the motor joint to undergo transverse motion. As in
\cite{pota1995}, the time domain beam equation \eqref{eq:bernoulli}
with boundary conditions \eqref{eq:bc-time1}-\eqref{eq:bc-time2} can
be transformed into an equivalent Laplace domain representation as
\begin{equation}
\label{eq:bernoulli-laplace}
    Y''''(x,s) - \beta^4 Y(x,s) = \frac{C_a V_a''(x,s)}{EI}
\end{equation}
with boundary conditions
\begin{align}
\label{eq:bc-laplace1}
    &Y(0,s) = 0, \\
    &EIY''(0,s) - I_h s^2 Y'(0,s) + \tau(s) = 0, \\
    &EIY''(L,s) + I_t s^2 Y'(L,s) = 0, \\
\label{eq:bc-laplace2}
    &EIY'''(L,s) - M_t s^2 Y(L,s) = 0,
\end{align}
where the primes indicate spatial derivatives and
\begin{equation}
\label{eq:beta}
    \beta^4(s) = -\frac{\rho A s^2}{EI}.
\end{equation}
Note that \eqref{eq:bernoulli-laplace} is the Laplace domain
equivalent of the Bernoulli-Euler beam equation with $V_a''(\cdot)$
as a forcing input. Together,
\eqref{eq:bernoulli-laplace}-\eqref{eq:bc-laplace2} represent a set
of linear ordinary differential equations with mixed boundary
conditions: two at $x=0$ and two at $x=L$. A state space
representation for the system can be formed  from  equations
\eqref{eq:bernoulli-laplace}-\eqref{eq:bc-laplace2} as in
\cite{pota1995}:
\begin{small}
\begin{align}
\label{eq:beam-ss}
    \begin{bmatrix}
        Y'(x,s) \\
        Y''(x,s) \\
        Y'''(x,s) \\
        Y''''(x,s)
    \end{bmatrix}
    =&
    \begin{bmatrix}
        0 & 1 & 0 & 0 \\
        0 & 0 & 1 & 0 \\
        0 & 0 & 0 & 1 \\
        \beta^4 & 0 & 0 & 0 \\
    \end{bmatrix}
    \begin{bmatrix}
        Y(x,s) \\
        Y'(x,s) \\
        Y''(x,s) \\
        Y'''(x,s)
    \end{bmatrix}
   \notag\\& +
    \begin{bmatrix}
        0 \\
        0 \\
        0 \\
        1
    \end{bmatrix}
    \frac{C_a V_a(s)}{EI}
    \sum_{i=1}^2 \delta(x - x_i)(-1)^{i+1}
\end{align}
\end{small}
where $\delta(\cdot)$ represents the Dirac delta function. The
equation  \eqref{eq:beam-ss} can be written in the general form
\begin{equation}
\label{eq:ss-general}
    Z'(x,s) = \bar{A} \; Z(x,s) + \bar{B} \; U(x,s),
\end{equation}
the solution to which is given by
\begin{small}
\begin{align}
\label{eq:ss-solution}
    &Z(x,s)\notag\\
    &=
    e^{\bar{A}x} Z(0,s)
    +
    \left[
        \bar{A} e^{\bar{A} (x-x_1)} \bar{B}
        -
        \bar{A} e^{\bar{A} (x-x_2)} \bar{B}
    \right]
    \frac{C_a V_a(s)}{EI}.
\end{align}
\end{small} Once the boundary conditions $Z(0,s)$ and $Z(L,s)$ are
known, \eqref{eq:ss-solution} will depend upon three conditions for
$x$, namely, $0 \le x \le x_1, \; x_1 \le x \le x_2,\; \text{and}\;
x_2 \le x \le L$. For further details see \cite{pota1995}.

\subsection{Infinite Dimensional Transfer function Model }\label{sec:B:ex}
Here, we present the input-output relationship between the two inputs $V_a$ and $\tau$, and the corresponding collocated outputs $V_s$ and $\theta$ in the form of the  transfer function matrix,
\begin{equation}
\label{eq:tf}
    \begin{bmatrix}
        \theta(s) \\
        V_s(s)
    \end{bmatrix}
    =
    G(s)
    \begin{bmatrix}
        \tau(s) \\
        V_a(s)
    \end{bmatrix},
\end{equation}
where  $G(s)
    =
    \begin{bmatrix}
        G_{\tau,\theta}(s) & G_{V_a,\theta}(s) \\
        G_{\tau,V_s}(s) & G_{V_a, V_s}(s) \\
    \end{bmatrix} $ and each of the elements of this transfer function matrix  is an infinite dimensional transfer function defined in terms of  transcendental functions of $\beta$. Indeed, each of the four transfer functions in \eqref{eq:tf} can be written as a ratio of numerator and denominator functions   computed as
\begin{small}
\begin{align}
\label{eq:g-tau-theta}
    &G_{\tau,\theta}(s)
        = \frac{N_{\tau,\theta}(s)}{D(s)}
        = \frac{Y'(0,s)}{T} \Big|_{V_a(s)=0}, \\
    &G_{V_a,\theta}(s)
        = \frac{N_{V_a,\theta}(s)}{D(s)}
        = \frac{Y'(0,s)}{V_a} \Big|_{\tau(s) = 0}, \\
    &G_{\tau,V_s}(s)
        = \frac{N_{\tau,V_s}(s)}{D(s)}
        = \frac{C_s \left( Y'(x_2,s)-Y'(x_1,s) \right)}{\tau(s)} \Big|_{V_a(s) = 0}, \\
\label{eq:g-va-vs}
    &G_{V_a,V_s}(s)
        = \frac{N_{V_a,V_s}(s)}{D(s)}
        = \frac{C_s \left( Y'(x_2,s)-Y'(x_1,s) \right)}{V_a(s)} \Big|_{T(s) = 0}.
\end{align}

Here, \begin{align} \label{eq:Dss}
    D(s)
    =&
    4 \beta E I (\rho A(\cos(\beta l)\sinh(\beta l)-\cosh(\beta l)\sin(\beta l)))\notag\\&-4\beta^4 E I I_h (1+\cos(\beta l))\cosh(\beta l)),
\end{align}\end{small} where
$I_h$ is the hub inertia. Also, the functions
 $N_{\tau,\theta}(s),N_{V_a,\theta},N_{\tau,V_s}(s)$, $N_{V_a,V_s}(s)$ are given by  very complicated expressions which can be found in  equations  (26)-(28) in  \cite{pota1995}.

We now compute the transfer functions in
\eqref{eq:g-tau-theta}-\eqref{eq:g-va-vs} for the case where the
piezoelectric actuators and sensors span the entire length of the
beam. This corresponds to the  substitutions: $x_1 = 0$ and
$x_2=L$. The resulting transfer functions  have been
verified  in \cite{Alberts1995} for an experimented robotic arm
system.

Despite the fact that  we have not defined the NI property for
infinite dimensional transfer functions, we will provide some
calculations which indicate that the  infinite dimensional transfer
function matrix $G(s)$ defined in \eqref{eq:tf}-\eqref{eq:g-va-vs}
satisfies  the NI conditions given in Definition \ref{Def:NI}. Since
the infinite dimensional transfer function matrix $G(s)$ is actually
a transcendental function of $\beta(s)$, $G(s)=\tilde{G}(\beta(s))$,
then  Condition 2) in Definition \ref{Def:NI} is equivalent to the
condition
\begin{equation}
\label{eq:ni}
    j(\tilde{G}(\beta(j\omega)) - \tilde{G}(\beta(j\omega))^*) \ge 0
\end{equation} for all $\omega\geq0$ where $\beta(s)$ is given by \eqref{eq:beta}. Indeed, it is straightforward to verify from the formulas for the transfer function matrix  \eqref{eq:tf}-\eqref{eq:g-va-vs} that  $j(\tilde{G}(\beta(j\omega)) - \tilde{G}(\beta(j\omega))^*) = 0$ for all $\omega\geq0 $. Also, the function  $D(s)$ given  in \eqref{eq:Dss} has an  infinite numbers of roots. However, we can check   Condition 3) in Definition 2  for a finite number of these roots  on the imaginary axis. To do so, we have calculated the first eleven $j\omega$-axis roots of $D(s)$ numerically. At each of these roots $s=j\omega_0$, the corresponding  residue matrix $K(j\omega _{0})=\underset{%
s\longrightarrow j\omega _{0}}{\lim }(s-j\omega _{0})jG(s)$ is calculated using  L'Hopital's rule as follows:
\begin{align}
\label{eq:lopetal_role}
    K(j\omega_0)&=\underset{%
s\longrightarrow j\omega _{0}}{\lim }(s-j\omega _{0})jG(s)\notag\\&=\underset{%
s\longrightarrow j\omega _{0}}{\lim }(s-j\omega _{0})j\begin{bmatrix}
        \frac{N_{\tau,\theta}(s)}{D(s)} & \frac{N_{V_a,\theta}(s)}{D(s)} \\
        \frac{N_{\tau,V_s}(s)}{D(s)} & \frac{N_{V_a,V_s}(s)}{D(s)} \\
    \end{bmatrix}\notag\\
&=j\begin{bmatrix}
        \frac{N_{\tau,\theta}(j\omega _{0})}{D'(j\omega _{0})} & \frac{N_{V_a,\theta}(j\omega _{0})}{D'(j\omega _{0})} \\
        \frac{N_{\tau,V_s}(j\omega _{0})}{D'(j\omega _{0})} & \frac{N_{V_a,V_s}(j\omega _{0})}{D'(j\omega _{0})} \\
    \end{bmatrix},
\end{align}
 where $D'(j\omega )$ denotes  the first derivative of $D(j\omega)$ with respect to  $\omega $.

In this case study, the parameter values for the robotic arm  are
taken   from \cite{Alberts1995}. These
parameter values are shown in the Table \ref{parpametr:values}.

\begin{table}[ht]
\caption{Robotic arm parameter values.} % title of Table
\centering  % used for centering table
\begin{small}\begin{tabular}{|l|c|c|}
  \hline\hline
  Parameter & Value & Unit  \\ \hline
  Hub inertia, $I_h$  & 0.0348  & $N-m-s^2$ \\ \hline
  Beam length,  $l$ & 2 & $m$ \\ \hline
  Volumetric mass  density, $\rho$ & 2712.6 & $Kg/m^2$ \\ \hline
  Cross sectional area, $A$  & 483.87$\times 10^{-6}$ & $m^2$ \\ \hline
  Young's Modulus,  $E$  & 69.0$\times 10^{9}$ & $N/m^2$ \\ \hline
  Area moment of inertia, $I$ & 1.63$\times 10^{-9}$  & $m^4$ \\ \hline
  Coupling Coefficient $k_{31}$ & -0.340 &-\\ \hline
  Capacitance, $C$  & 68.35 & $\mu F/m^2$ \\ \hline
  Thickness $t_s$ & 3.05$\times 10^{-4}$ &$ m$ \\ \hline
  \hline
\end{tabular}\end{small}.
\label{parpametr:values} % is used to refer this table in the text
\end{table}

Table \ref{table1} shows the calculated roots of $D(s)$ and the minimum eigenvalue of the corresponding residue matrix given in \eqref{eq:lopetal_role}.
Also, the matrix $G_2=\underset{s\longrightarrow 0}{\lim }s^{2}G(s)$ is found to be $
    \begin{bmatrix}
        0.14 & 0 \\
        0 & 0
    \end{bmatrix}$ which is positive semidefinite.
\begin{table}[ht]
\caption{The minimum eigenvalues of the residue matrix  corresponding to the first ten resonant modes  of the infinite dimensional plant transfer function matrix.} % title of Table
\centering  % used for centering table
\begin{small}\begin{tabular}{|c|c|c|}
  \hline\hline
  n & Root $s=j\omega$ (rad/s) & {Minimum eigenvalue }  \\ \hline
  0 & 0           & 0 \\ \hline
  1 & 3.395326441 & 0.1434 \\ \hline
  2 & 9.501801884 & 0.2553 \\ \hline
  3 & 17.08210071 & 0.1320 \\ \hline
  4 & 29.32863976 & 0.0361 \\ \hline
  5 & 47.01240951 & 0.0142 \\ \hline
  6 & 96.84550724 & 0.0049 \\ \hline
  7 & 128.7332003 & 0.0034\\ \hline
  8 & 165.2195349 & 0.0025 \\ \hline
  9 & 206.2898971 & 0.0019 \\ \hline
  10 & 251.9420283 & 0.0015 \\
  \hline
\end{tabular}\end{small}.
\label{table1} % is used to refer this table in the text
\end{table}
These  results  show that the  infinite dimensional transfer
function matrix $G(s)$  satisfies  the   conditions of Definition
\ref{Def:NI}, at least for the first ten resonant modes. %The
%An infinite dimensional transfer function matrix $G(s)$ model for
%the robotic arm  is  given in \cite{pota1995}:
%\begin{small}
%$%\begin{align}\label{dddd}
%G(s) = \begin{bmatrix}
%        \frac{N_{\tau,\theta}(s)}{D(s)} & \frac{N_{V_a,\theta}(s)}{D(s)} \\
%       \frac{N_{\tau,V_s}(s)}{D(s)} & \frac{N_{V_a,V_s}(s)}{D(s)} \\
%    \end{bmatrix},
%$%\end{align}
%\end{small} where $N_{\tau,\theta}(s),N_{V_a,\theta},N_{\tau,V_s}(s)$,
%$N_{V_a,V_s}(s)$ and $D(s)$ are given by   complicated expressions
%which can be found in  equations  (26)-(28) in \cite{pota1995}.

%
%
Also in the Fig. \ref{Resud11}, we plot the  log of the  minimum eigenvalue of the matrix $D'(j\omega)^2
K(j\omega) + \gamma D( j\omega )^2$ as a function of  frequency $\omega$, where $\gamma$ is a positive constant.  This plot also indicates that the residue matrix defined in \eqref{eq:lopetal_role} will be positive semidefinite at the system poles within the frequency range of interest.

  \begin{figure}
    \centering{\includegraphics[width=8.6cm]{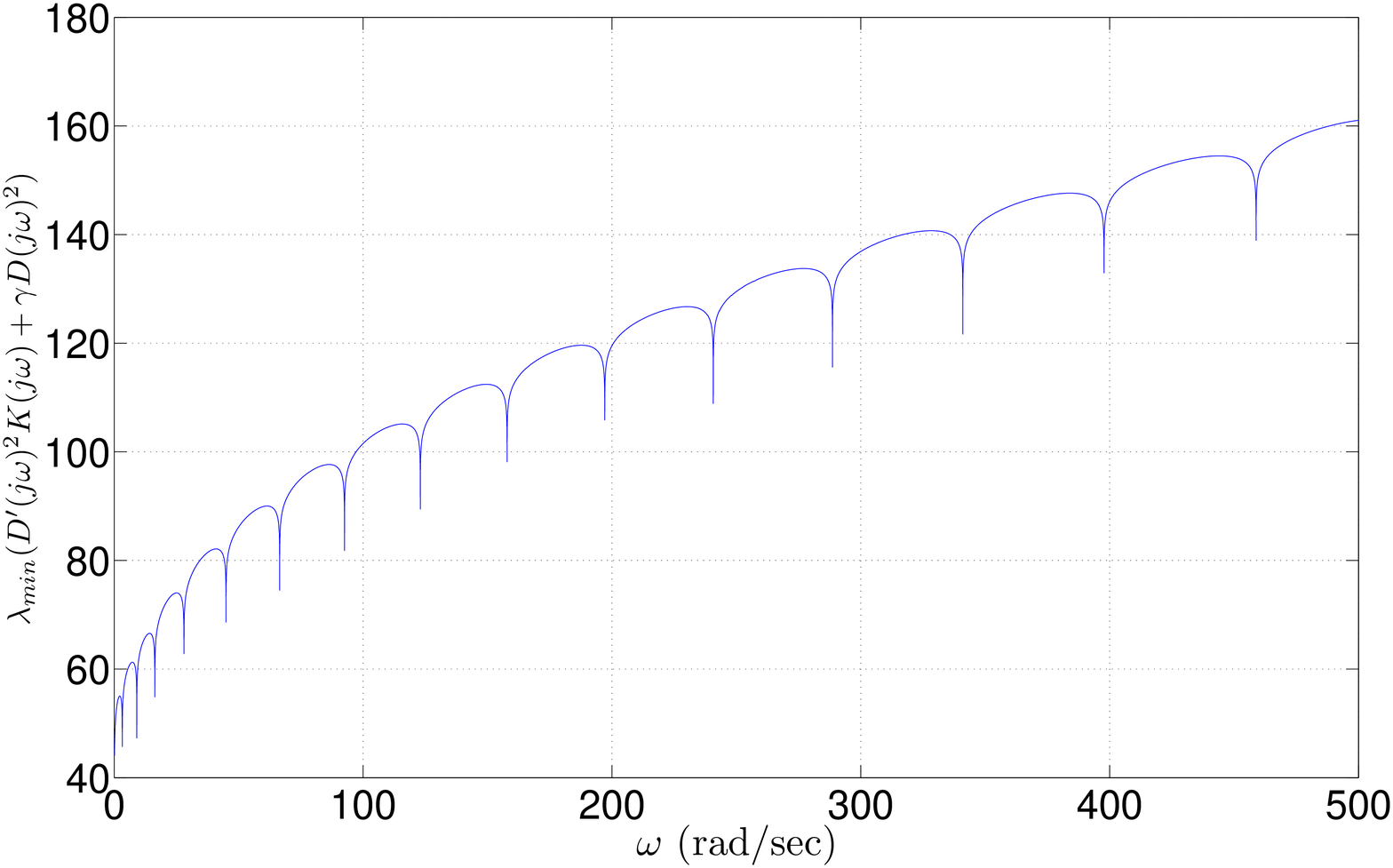}}
\caption{Plot of the log of the  minimum eigenvalue of the matrix $ D'(j\omega)^2  K(j\omega) + \gamma D( j\omega )^2$ versus  frequency $\omega$.} \label{Resud11}
\end{figure}

\subsection{Approximate Finite-dimensional Transfer Function Matrix}
\label{ssec:finite-dim}
The transfer functions in \eqref{eq:g-tau-theta}-\eqref{eq:g-va-vs} are irrational functions of $s$.
We now approximate  these transfer functions by rational functions in $s$ in order to design a suitable controller for the robotic arm system and to simulate its performance.
%\begin{equation}
%\label{eq:G}
%   G(s)
%   =
%   \begin{bmatrix}
%       G_{\tau,\theta}(s) & G_{V_a,\theta}(s) \\
%       G_{\tau,V_s}(s) & G_{V_a, V_s}(s) \\
%   \end{bmatrix}.
%\end{equation}
Various methods such as the Maclaurin series expansion presented in
\cite{Schmitz:1985}, the Rayleigh-Ritz method \cite{meirovitch1975},
and the assumed modes method \cite{meirovitch1975} are available in
literature for the finite dimensional approximation  of such an
infinite dimensional model. Here,  we adopt a partial fraction
approach to obtain a  finite dimensional approximation  of $G(s)$.
This method is similar to the assumed modes technique described in
\cite{meirovitch1975}. The finite dimensional model can be written
as
\begin{small}
\begin{align}
\label{eq:G}
    G_f(s)
    &=
    \begin{bmatrix}
        G_{f\tau,\theta}(s) & G_{fV_a,\theta}(s) \\
        G_{f\tau,V_s}(s) & G_{fV_a, V_s}(s) \\
    \end{bmatrix}\notag\\&= \begin{bmatrix}
       \frac{N_{\tau,\theta}(s)}{D(s)} & \frac{N_{V_a,\theta}(s)}{D(s)} \\
     \frac{N_{\tau,V_s}(s)}{D(s)} & \frac{N_{V_a,V_s}(s)}{D(s)} \\
         \end{bmatrix}
    \notag\\&=
    \sum_{i=0}^{n}\frac{1}{k}\begin{bmatrix}
        \frac{a_i}{s^2 + p_i^2} & \frac{b_i}{s^2 + p_i^2} \\
        \frac{c_i}{s^2 + p_i^2} & \frac{d_i}{s^2 + p_i^2} \\
    \end{bmatrix}.
\end{align}
\end{small}
Here $D(s)$ in the infinite dimensional model $G(s)$ is approximated  by
\begin{equation}
  %  \label{eq:D:f}
    D_f(s)=k\prod_{i=0}^n(s^2 + p_j^2),
\end{equation}
where, $jp_0...jp_n$ are the first n $j\omega$-axis roots of $D(s)$. Also,  the  coefficient matrices $C_i=\begin{bmatrix}
        a_i & b_i \\
        c_i& d_i \\
    \end{bmatrix},$ are computed using a  partial fraction expansion method. That is,
   \begin{small} \begin{equation}C_i=   \frac{1}{k\prod_{j=0,j\neq i}^n(-p_i^2 + p_j^2)}\begin{bmatrix}\label{rez:approxmated}
        N_{\tau,\theta}(j p_i) & N_{V_a,\theta}(j p_i)\\
        N_{\tau,V_s}(jp_i)& N_{V_a,V_s}(jp_i) \\
    \end{bmatrix}.
\end{equation}\end{small}
 The constant $k$ is chosen so that
\begin{equation}
    \label{eq:cons:k}
    D(j\omega_0)=k\prod_{i=0}^n(-\omega_{0}^2 + p_i^2),
\end{equation}
where $\omega_0$ is such that $j\omega_0$ is not a root of $D(s)$.
%\begin{equation}
%C_i=   \frac{1}{k\prod_{j=0,j\neq i}^n(s^2 + p_i^2)}\begin{bmatrix}\label{rez:approxmated}
%        \frac{N_{f\tau,\theta}(p_i)}{k\prod_{j=0,j\neq i}^n(s^2 + p_i^2)} & \frac{N_{fV_a,\theta}(p_i)}{k\prod_{j=0,j\neq i}^n(s^2 + p_i^2)} \\
%        \frac{Nf_{\tau,V_s}(p_i)}{k\prod_{j=0,j\neq i}^n(s^2 + p_i^2)} & \frac{Nf_{V_a,V_s}(p_i)}{k\prod_{j=0,j\neq i}^n(s^2 + p_i^2)} \\
%    \end{bmatrix}.
%\end{equation}
 We consider the first  resonant  mode; i.e., $n=1$ for the controller design. %for truncating $G(s)$ as,
%\begin{equation}
%\label{eq:partial-fraction}
%   G_f(s)
%   =
%   \frac{C_0}{s^2}
%   +
%   \frac{C_1}{s^2 + p_1^2}
%   +
%   \frac{C_2}{s^2 + p_2^2}
%   +
%   \frac{C_3}{s^2 + p_3^2}
%   +
%   \frac{C_4}{s^2 + p_4^2}.
%\end{equation}
%Here, the coefficient matrix $C_0$ corresponds to to the double pole at the origin whereas the coefficient matrices $C_1, C_2, C_3$, and $C_4$ correspond to double poles at $p_1, p_2, p_3$, and $p_4$ respectively defining the other four modes of the finite dimensional approximation for $G(s)$.
The corresponding  coefficient matrices were computed %using \eqref{rez:approxmated}
and were found to be
\begin{small}
$%\begin{align}
%\label{eq:coeffs-pole}
    C_0
    =
    \begin{bmatrix}
        0.14 & 0 \\
        0 & 0
    \end{bmatrix};
    C_1
    =
    \begin{bmatrix}
        3.0907 & 3.5573\times10^{-4} \\
        3.5573\times10^{-4} & 2.3500
    \end{bmatrix};%\notag\\
    %&C_2
%    =
%    \begin{bmatrix}
%        12.1185  & -0.2165\times10^{-4} \\
%        -0.2165\times10^{-4} & 44.1415
%    \end{bmatrix};\notag\\
%    &C_3
%    =
%    \begin{bmatrix}
%        11.0959 & 11.724\times10^{-4} \\
%        11.724\times10^{-4}  & 911.9085
%    \end{bmatrix}; \notag\\
%    &C_4
%    =
%    \begin{bmatrix}
%        3.0791 & -1.4638\times10^{-4} \\
%        -1.4638\times10^{-4} & 17438.04
%    \end{bmatrix},
$ %\end{align}
\end{small} and %the constant $k$ was calculated using
%\eqref{eq:cons:k} to be
 $k=6.6667 \times 10^{-8}$. Also, the poles
were computed to be $p_0=0,p_1 = 3.4$.%; p_2 = 9.5; p_3 = 17.08; p_4 =
%29.33$.%, which are the first four resonant modes given in Table
%\ref{table1}.

%\subsection{NI property test}
%In this section, we test for the NI properties of the infinite dimensional transfer function matrix $G(s)$ as well as the truncated finite dimensional transfer function matrix $G_f(s)$ for the robotic arm. Since the infinite dimensional transfer function matrix $G(\cdot)$ is actually a transcendental function of $\beta(s)$, it satisfies the NI property if \cite{petersen2010}
%\begin{equation}
%\label{eq:ni}
%   j(G(\beta(j\omega)) - G(\beta(j\omega))^*) \ge 0,
%\end{equation}
%where $\beta(s)$ is given by \eqref{eq:beta}. Substituting the component transfer functions from \eqref{eq:G} into \eqref{eq:ni} we get $j(G(\beta(j\omega)) - G(\beta(j\omega))^*) = 0$, which shows that the infinite dimensional transfer function describing the robotic arm is indeed NI.

The finite dimensional model $G_f(s)$ in
\eqref{eq:G} is NI, since  $j(G_f(j\omega) - G_f(j\omega)^*) = 0$
for all $\omega\geq0$, where $j\omega$ is not a pole for $G_f(s)$. This follows because in this example, $G_f(j\omega)$ is real and symmetric for all $\omega$ such that $j\omega$ is  not a pole of  $G_f(s)$.
Also, the coefficient matrices $C_0,C_1$ are positive semidefinite
which implies that Condition 3) in Definition \ref{Def:NI} is
satisfied. Moreover, $G_2=\underset{s\longrightarrow 0}{\lim
}s^{2}G(s)=C_0\geq0,$ which implies that   Condition 4) in
Definition \ref{Def:NI} is satisfied.
%Furthermore, since the finite dimensional transfer function $G_f(s)$ has poles at the origin corresponding to the free body motion for the flexible robotic arm, we check for the properties in (5)-(7) to test for the NI property. Computing the matrices $G_2, G_1$, and $G_0$ in (5)-(7), we get
%\begin{equation}
%\label{eq:Gs}
%   G_2
%   =
%   \begin{bmatrix}
%       0.14 & 0 \\
%       0 & 0
%   \end{bmatrix}
%   \> 0;
%   \qquad
%   G_1
%   =
%   \begin{bmatrix}
%       0 & 0 \\
%       0 & 0
%   \end{bmatrix}
%   = 0;
%   \qquad
%   G_0
%   =
%   \begin{bmatrix}
%       0.19 & 0 \\
%       0 & 0.073
%   \end{bmatrix}
%   > 0
%\end{equation}
%which satisfies the NI property for the finite dimensional system transfer function matrix $G_f(s)$ for the flexible robotic arm. In fact, the matrices in \eqref{eq:Gs} correspond to the case described in Corollary 3. We will use the conditions in Corollary 3 to check for the stability of the closed loop system in Section \ref{ssec:controller}.

\subsection{Controller design}
\label{ssec:controller}

According to Theorem \ref{min:result} if a plant is  NI, any SNI
controller which satisfies the conditions of  Theorem
\ref{min:result}  will stabilize the system. The fact that the robotic arm plant involves colocated ``force" actuators and ``position" sensors indicates that this plant should be NI.
%Furthermore,
%the analysis given in Subsection \ref{sec:B:ex} above indicate that
%any finite dimensional approximation to the manipulator system
%should be NI.
 In particular, the finite dimensional approximation
to the robotic arm model  derived in Subsection
\ref{ssec:finite-dim} was shown to be NI. We will now use a finite dimensional model of the form  \eqref{eq:G} to design a
controller for the system. First, we compute the matrices $G_2,
G_1$, and $G_0$ in \eqref{f_G0}, for the finite
dimension approximate system where n=1 in \eqref{eq:G} to obtain
\begin{small}
\begin{align}
\label{eq:Gs}
    &G_2
    =
    \begin{bmatrix}
        0.14 & 0 \\
        0 & 0
    \end{bmatrix}
    \> 0;
    \qquad
    G_1
    =
    \begin{bmatrix}
        0 & 0 \\
        0 & 0
    \end{bmatrix};
    \notag\\
    &
    G_0
    =
    \begin{bmatrix}
        0.41253083 & 0.0000319 \\
        0.0000319 & 0.15672805
    \end{bmatrix}.
\end{align}
\end{small} This implies that we can use Corollary
\ref{min:result:clo3} to guarantee the stability of the  positive
feedback interconnection between the plant and an SNI controller.

In this case study,  an integral resonant controller (IRC) is chosen to stabilize the system; e.g., see \cite{petersen2010}. An IRC is a first order controller which takes the form
\begin{equation}
\label{eq:irc}
    \bar{G}(s) = (sI + \Gamma \Phi)^{-1} \Gamma - \Delta.
\end{equation}
This controller is SNI if   $\Gamma>0,\Phi>0$ and $\Delta$ is a
symmetric matrix \cite{petersen2010}. Now,  we  chose the controller
matrices $\Gamma>0,\Phi>0$ and $\Delta$ such that the  conditions
of Corollary \ref{min:result:clo3} are satisfied. We choose the
controller matrices as follows:
\begin{small}
\begin{align}
\label{eq:irc-mtx}
  \Gamma
    =
    \begin{bmatrix}
        35 & 15 \\
        15 & 20
    \end{bmatrix};
    \Phi
    =
    \begin{bmatrix}
        0.745 & 0.521 \\
        0.521 & 1.021
    \end{bmatrix};
    \Delta
    =
    \begin{bmatrix}
        4.2900 & 0 \\
        0 & 2.22
    \end{bmatrix}.
\end{align}
\end{small}
This leads to a controller   DC gain matrix of  $\bar{G}(0)=\begin{bmatrix} -2.2029 &  -1.0650\\
   -1.0650  &  -0.6971\end{bmatrix}$. To check the stability conditions  in Corollary \ref{min:result:clo3}, we first compute  the matrix $J$ in \eqref{JJ-m} using $G_2$ in \eqref{eq:Gs}. This yields  $J
        =
        \begin{bmatrix}
            0.3751 \\
            0
        \end{bmatrix}.$ Also, the matrix $N_2$ in \eqref{N-2-g10} is calculated as   $N_2=\begin{bmatrix} 0 & 0\\
  0  & -0.182252 \end{bmatrix}$, which is negative semidefinite.
Then we conclude   \begin{small}$\det(I + \tilde{N}_2 G_0 \tilde{N}_2)=\det\begin{bmatrix} 1.000000025 & 0.000000003\\
  0.000000003  &  5.390603\end{bmatrix} \neq0,$ where $\tilde{N}_2=(-N_2)^{1/2}.$ Also,  $J^T \bar{G}(0) J =-0.309908135 < 0.$ \end{small}
  Thus, the conditions of Corollary \ref{min:result:clo3} are satisfied.

  \begin{figure}
    \centering{\includegraphics[width=5.6cm]{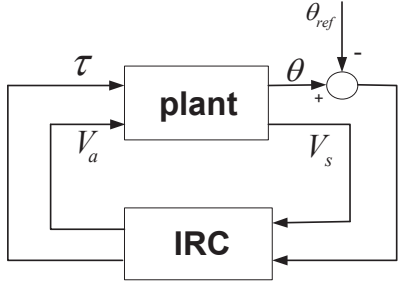}}
\caption{Block diagram  corresponding to a step change in the robotic arm  reference position.} \label{blok:step}
\end{figure}
 To verify the performance  of the closed loop system, we simulate the  response of this system corresponding   to a step change in the reference  position  of the robotic arm; see Fig. \ref{blok:step}.
This step response is shown in Fig. \ref{step11}. Also,  the
corresponding response of the piezo sensor output $V_s$ is shown in
Fig. \ref{step22}. Here, the step responses were calculated using
finite dimensional plant models defined in
\eqref{eq:G}
%-\eqref{eq:coeffs-pole}
for different numbers of modes,
n=2,3...7.

To this end, we have used the proposed  controller which is designed
for the finite dimensional model with n=1 when applied to the plant
with finite dimensional model  where n=2,3...7 in order to check the
performance and robustness of the proposed controller. In fact, the
performance of the closed loop system is
  found to
  improve by increasing the number of modes; see Fig
  \ref{step11} and Fig \ref{step22}.

 \begin{figure}
    \centering{\includegraphics[width=8.6cm]{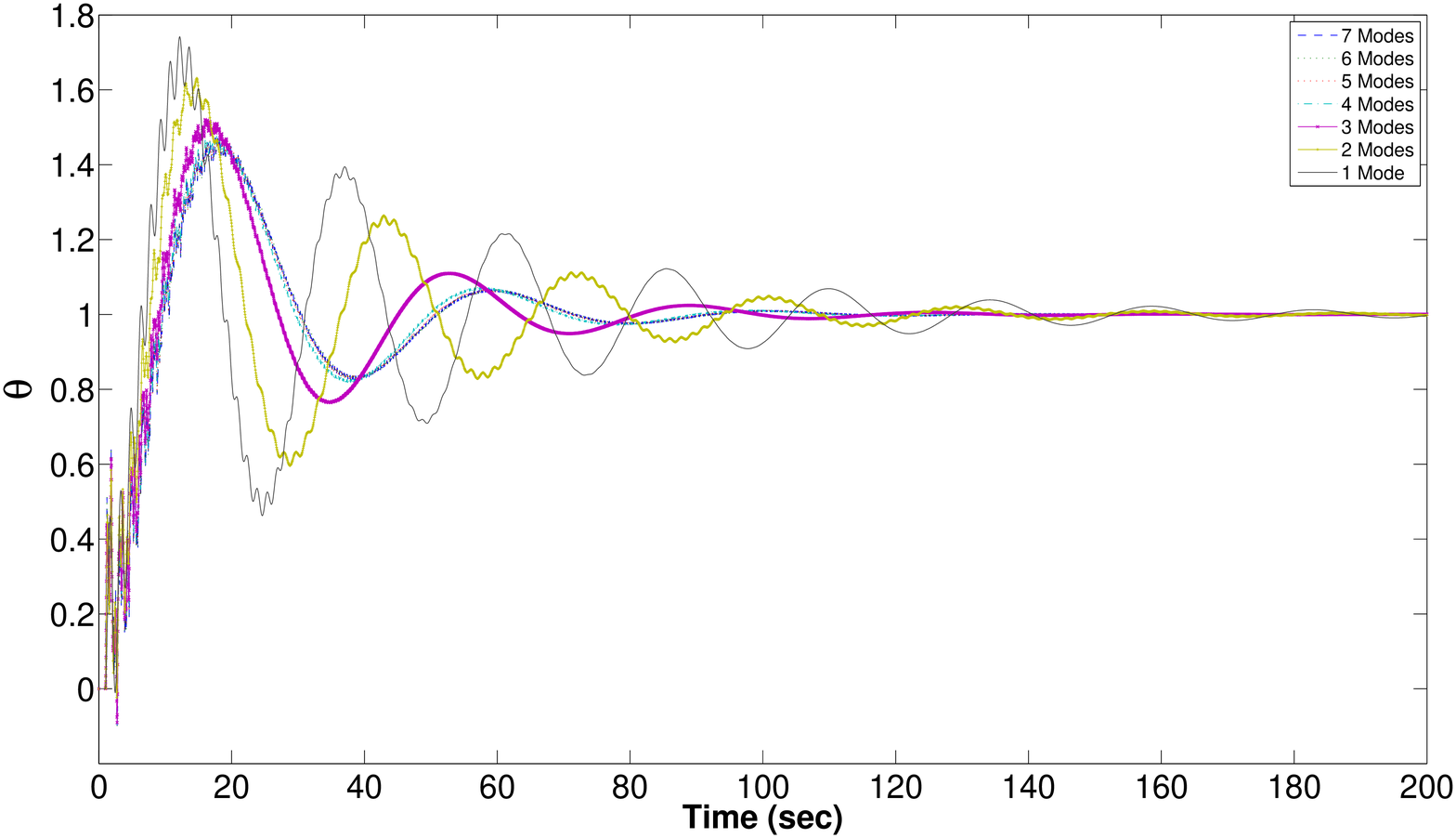}}
\caption{Position response  of the robotic arm  system corresponding to a unit step change in reference position. Here, different numbers of modes are used in the plant model. } \label{step11}
\end{figure}

 \begin{figure}
    \centering{\includegraphics[width=8.6cm]{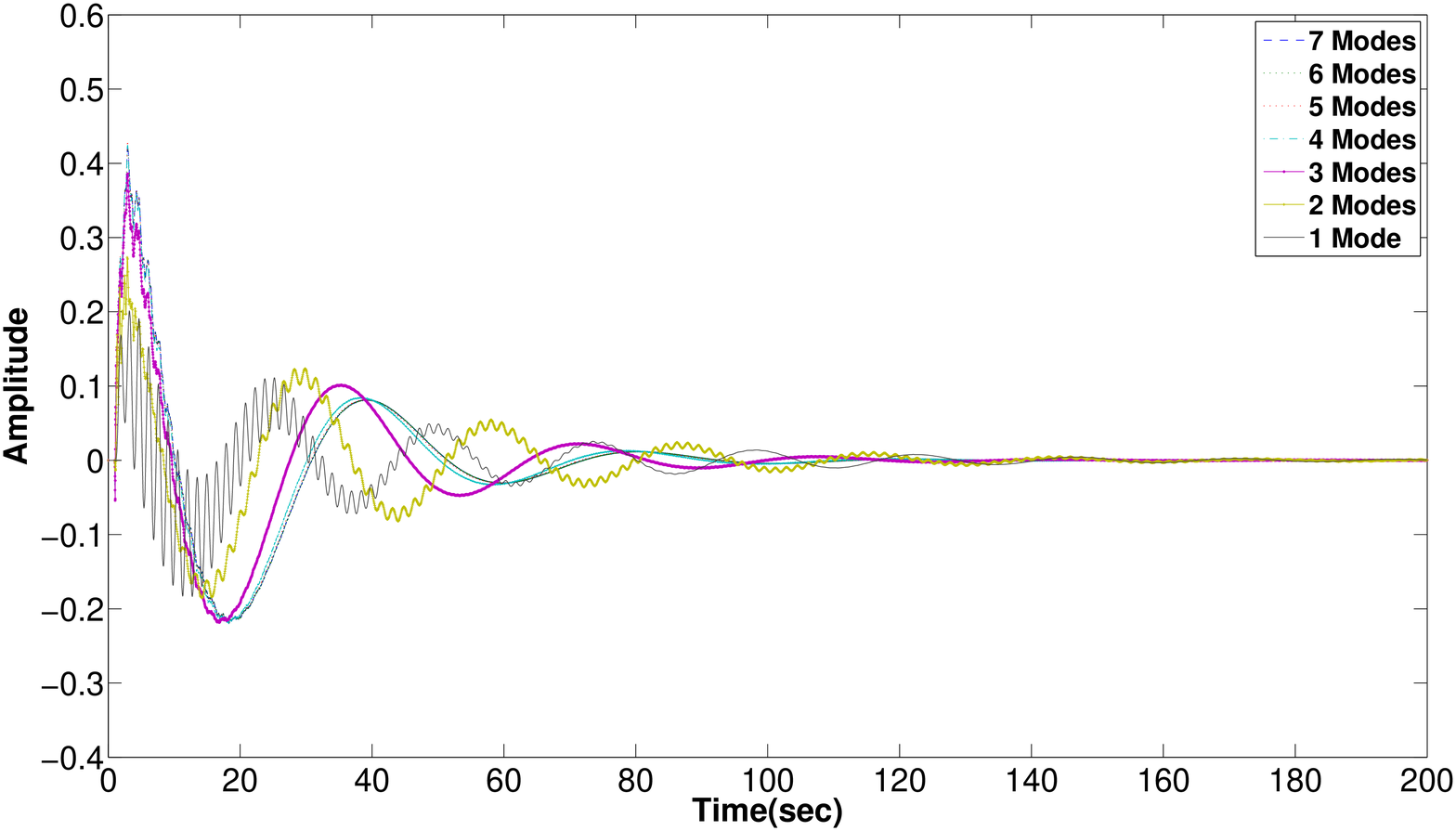}}
\caption{ Piezoelectric response sensor output $V_s$ response corresponding to  a unit step change in reference position. Here, different numbers of modes are used in the plant model.} \label{step22}
\end{figure}

Note that the controller parameters in \eqref{eq:irc-mtx} were
chosen by process of trial and error to obtain good closed loop
performance the case of the nominal plant model, $n=1$. An
alternative approach, which would be useful in the case of a more
complicated SNI controller structure, would be to use an
optimization procedure to obtain the controller parameters; e.g.,
see \cite{Moheimani2006}.

\section{Conclusion}
\label{sec:conclusion} In this paper,  new  stability results for the
positive-feedback interconnection of  negative imaginary
systems have been derived. A new NI definition is presented,
 which   allows for systems  having  free
body dynamics  to be considered as  NI systems. This work can be
used in controller design to allow for a broader class of NI systems
than considered in previous work. %Also, the stability result for an
%NI system with poles at the origin connected with an SNI system with
%positive feedback can be used for
%  controller design including    integral action.
The application  of the main results in this paper  has  been
illustrated via a case  study involving the control of  a flexible
robotic arm.

\bibliographystyle{IEEEtran}
\bibliography{stability_bib}
\balance

\section{Appendix A}

In this appendix, we present state space    results (some of which are of independent interest) using a particular
state space representation of the plant transfer function matrix
$G(s)$. The first  stability result  is
Theorem \ref{Pre_main_re}. We  will subsequently  use Theorem
\ref{Pre_main_re} to prove
Theorem  \ref{min:result}. We also present a number of corollaries which will be used to prove the remaining  results of the paper.  %Theorem  %\ref{min:result:clo1} and  Corollaries \ref{min:result:clo3} -\ref{min:result:clo2}.
%\subsubsection{Appendix A}\label{app:a}

We first consider an  NI square  transfer function matrix $G(s)$
with a
 minimal state space realization of the form
\begin{align}
\label{eq:xdot}
&\dot{x}(t) = A x(t)+B u(t), \nonumber \\
%\label{eq:y}
&y(t) = C x(t),
\end{align}%
where, $A \in \mathbb{R}^{n \times n},B \in \mathbb{R}^{n \times
m},C \in \mathbb{R}^{m \times n},$ and
%Consider the following class of systems $\digamma(s)=\{G(s)\}$ such
%that $G(s)=\begin{bmatrix}
%\begin{array}{c|c}
%A & B \\ \hline C & D
%\end{array}
%\end{bmatrix}$ can be represented in  a block diagonal form as
%following
\begin{align}
A =&
\begin{bmatrix}
A_{1} & 0 & 0 \\
0 & A_{2} & 0 \\
0 & 0 & A_{3}%
\end{bmatrix}%
;\:\: B =
\begin{bmatrix}
B_{1} \\
B_{2} \\
B_{3}
\end{bmatrix};\:\
C=
\begin{bmatrix}
C_{1} & C_{2} & C_{3}%
\end{bmatrix},\label{J_C_F}
\end{align}
 $A_1 \in \mathbb{R}^{n_1 \times n_1},A_2 \in \mathbb{R}^{n_2
\times n_2},A_3 \in \mathbb{R}^{2k \times 2k}, B_1 \in
\mathbb{R}^{n_1 \times m},B_2 \in \mathbb{R}^{n_2 \times m},
B_{3}=\begin{bmatrix}
B_{3a} \\
B_{3b}
\end{bmatrix},B_{3a}
\in \mathbb{R}^{k \times m},B_{3b} \in \mathbb{R}^{k \times m} ,C_1
\in \mathbb{R}^{m \times n_1},C_2 \in \mathbb{R}^{m \times n_2},
C_{3}=
\begin{bmatrix}
C_{3a} & C_{3b}
\end{bmatrix}, C_{3a} \in \mathbb{R}^{m \times k}, C_{3b} \in \mathbb{R}^{m
\times k},$ $A_1$ is nonsingular, $A_{2}=0$, and
\begin{align}\label{matA3}
A_{3}=\begin{bmatrix}
0 & I_{k \times k} \\
0 & 0%
\end{bmatrix}.
\end{align}

We also consider an SNI  transfer function matrix $\bar{G}(s)$ with a
minimal state space realization
\begin{align}
\label{SNI:eq:xdot}
&\dot{x}(t) = \bar{A} x(t)+\bar{B} u(t), \nonumber \\
%\label{SNI:eq:y}
&y(t) = \bar{C} x(t)+\bar{D} u(t),
\end{align}%
where $\bar{A} \in \mathbb{R}^{\bar{n} \times \bar{n}},\bar{B} \in
\mathbb{R}^{\bar{n} \times m},\bar{C} \in \mathbb{R}^{m \times
\bar{n}},$ and $\bar{D} \in \mathbb{R}^{m \times m}.$

\begin{remark}
We allow any of the matrices in these  models to have zero dimensions.
In sequel,  any matrix  with zero dimension is regarded as being of
full rank.

\end{remark}

The corresponding transfer function matrix for the state space
realization  \eqref{eq:xdot}-\eqref{J_C_F} is given as follows:
\begin{align}
G(s)=&C_{1}(sI-A_{1})^{-1}B_{1}+C_{2}(sI-A_{2})^{-1}B_{2}\notag\\&+C_{3}(sI-A_{3})^{-1}B_{3}\notag\\
=&C_{1}(sI-A_{1})^{-1}B_{1}+\frac{C_{2}B_{2}+C_{3}B_{3}}{s}+\frac{C_{3a}B_{3b}}{s^2}.\label{parelle-g}
\end{align}

The following theorem  provides  a necessary and sufficient
condition for the stability of the
 positive-feedback interconnection between the NI
transfer function matrix $G(s),$ with state space realization
(\ref{eq:xdot})-(\ref{J_C_F}), and the  SNI transfer function matrix
$\bar{G}(s),$ with state space realization
(\ref{SNI:eq:xdot}). In order to present this
theorem, we define the following matrix
\begin{align}\label{N-g-lemma-1}
N=\mathcal{P}\left(\bar{G}(0),\begin{bmatrix}
C_{2} & C_{3a}%
\end{bmatrix}\right). %\bar{G}%
%(0)-\bar{G}(0)%
%\begin{bmatrix}
%C_{2} & C_{3a}%
%\end{bmatrix}%
%\Xi^{-1}\begin{bmatrix}
%C_{2}^T \\
%C_{3a}^T%
%\end{bmatrix}%
%\bar{G}(0),
\end{align}
 Also, the matrix
 \begin{align}\label{segma-N}
 \Xi=%
\begin{bmatrix}
C_{2}^{T} \\
C_{3a}^{T}%
\end{bmatrix}%
\bar{G}(0)%
\begin{bmatrix}
C_{2} & C_{3a}%
\end{bmatrix}%
\end{align}
 is assumed to be non-singular.
%================================================================================
%=================================================================================
%
%\begin{assumption}
%\label{assm:siso3} We assume  that $k\neq0$. Also, the matrix $\Xi$
%in \eqref{segma-N} is assumed to be non-singular and the matrix  $N$
%in \eqref{N-g-lemma-1} is assumed to be either positive semidefinite
%or negative semidefinite.
%\end{assumption}
In addition,
we will use the following condition in the theorem which follows:
\begin{equation}\label{Eq:Con1:th2}
   \begin{bmatrix}
C_{2}^{T} \\
C_{3a}^{T}%
\end{bmatrix}%
\bar{G}(0)%
\begin{bmatrix}
C_{2} & C_{3a}%
\end{bmatrix}<0.
\end{equation} %where $\Xi$ defined as \eqref{segma-N}.
Also, for the case in which $N$ is positive semidefinite, we will
use the condition
\begin{equation}\label{Eq:Con2:th2}
    I+N^{\frac{1}{2}}C_{1}A_{1}^{-1}B_{1}N^{\frac{1}{2}}-N^{\frac{1}{2}}C_{3b}P_{2}^{-1}C_{3b}^{T}N^{\frac{1}{2}}>0,
\end{equation} where  $P_{2}=C_{3a}^{T}B_{3b}^{T}(B_{3b}B_{3b}^{T})^{-1},$ which
will be shown to be  symmetric and positive definite in Lemma
\ref{NI-lemma-jordon-f}. Moreover, for the case in which $N$ is
negative  semidefinite, we will use the condition
\begin{equation}\label{Eq:Con3:th2}
    \det(I-\tilde{N}C_{1}A_{1}^{-1}B_{1}\tilde{N}+\tilde{N}C_{3b}P_{2}^{-1}C_{3b}^{T}\tilde{N})\neq0,
\end{equation} where  $\tilde{N}=(-N)^{\frac{1}{2}}$.

\begin{theorem}\label{Pre_main_re}
Suppose   that $k\neq0$ and the matrix $\Xi$ in \eqref{segma-N} is
non-singular. Also, suppose that  the transfer function matrix
$G(s),$ with the minimal state space realization
(\ref{eq:xdot}), is NI and  the transfer function
matrix $\bar{G}(s),$ with the minimal state space realization
(\ref{SNI:eq:xdot}), is SNI.  If $N$ is positive
semidefinite, then the closed-loop positive-feedback interconnection
between $G(s)$ and $\bar{G}(s)$  is internally stable   if and only
if conditions \eqref{Eq:Con1:th2} and \eqref{Eq:Con2:th2} are
satisfied. Also, if $N$ is negative semidefinite, then the
closed-loop positive-feedback interconnection between $G(s)$ and
$\bar{G}(s)$  is internally stable   if and only if conditions
\eqref{Eq:Con1:th2} and \eqref{Eq:Con3:th2} are satisfied.
\end{theorem}

The proof of this theorem is given at the end of this appendix.

\begin{corollary}\label{Pre_main_re-G1-G2-0}
Suppose that the matrix $\Xi$ in \eqref{segma-N} is non-singular and
the matrix  $N$ in \eqref{N-g-lemma-1} satisfies $N\begin{bmatrix}
C_{1} & C_{3b}
\end{bmatrix}=0$.
Also suppose that the transfer function matrix $G(s),$ with the
minimal state space realization (\ref{eq:xdot}), is NI
and   the transfer function matrix $\bar{G}(s),$ with the minimal
state space realization (\ref{SNI:eq:xdot}), is
SNI. Then the closed-loop positive-feedback interconnection  between
$G(s)$ and $\bar{G}(s)$ is internally stable   if and only if
condition \eqref{Eq:Con1:th2} is satisfied.
\end{corollary}

The proof of this corollary is given at the end of this appendix.

%of Corollary \ref{Pre_main_re-G1-G2-0}
%\emph{Proof:} The proof of this  corollary will proceeds in an
%almost identical fashion to the proof of Theorem \ref{Pre_main_re}.
%Indeed, we first state the following claim:
%
%\emph{Claim 2}: Assume that the matrix  $N$ in \eqref{N-g-lemma-1}
%satisfies $N\begin{bmatrix} C_{1} & C_{3b}
%\end{bmatrix}=0$, then  $T>0$ if and only if \eqref{Eq:Con1:th2} is satisfied.
%
%This claim is equivalent to Claim 1 in Theorem \ref{Pre_main_re}
%when we relax the conditions on the matrix $N$. The proof of this
%claim is similar to the proof of Claim 1 in the proof of Theorem
%\ref{Pre_main_re} since \eqref{Eq:shur:cond2-th2} is automatically
%satisfied in the case when $N\begin{bmatrix} C_{1} & C_{3b}
%\end{bmatrix}=0$.
%
%Also,  the determinant condition in \eqref{det_eq3} will be
%automatically satisfied using the fact $N\begin{bmatrix} C_{1} &
%C_{3b}
%\end{bmatrix}=0$ in  \eqref{det-G1-G2-0}. The proof of the corollary then follows as in the proof of Theorem \ref{Pre_main_re}. \hfill $\blacksquare$

The following corollary considers the  case when $n_2=0$ and
$k\neq0$;
 i.e., the matrix $A$  in \eqref{J_C_F} has  the block diagonal
 form $A=\begin{bmatrix}
A_{1} & 0 \\
 0 & A_{3}%
\end{bmatrix}.$  In the   case when $n_2=0$, the matrix $N$ in \eqref{N-g-lemma-1} will
be given by
\begin{align}\label{N-g-croll-2}
N=\mathcal{P}(\bar{G}(0),C_{3a}),%\bar{G}(0)-\bar{G}(0)C_{3a}\left( C_{3a}^{T}\bar{G}%
%(0)C_{3a}\right) ^{-1}C_{3a}^{T}\bar{G}(0)
 \end{align}
where we assume that  $ C_{3a}^{T}\bar{G}(0)C_{3a}$ is non-singular.

%\begin{assumption}
%\label{assm:siso3:crol7} The matrix $C_{3a}^{T}\bar{G}(0)C_{3a}$ is
%assumed to be non-singular. Also, the matrix $N$ in
%\eqref{N-g-croll-2} is assumed to be either positive semidefinite or
%negative semidefinite.
%\end{assumption}

We will use the following conditions in the next corollary  which
 correspond to conditions
\eqref{Eq:Con1:th2}-\eqref{Eq:Con3:th2} in Theorem
\ref{Pre_main_re}. The first condition to be considered is
\begin{equation}\label{Eq:Con1:crol7}
    C_{3a}^{T}\bar{G}(0)C_{3a}<0.
\end{equation}
Also, for the case in which $N$ is positive semidefinite, we will
use the condition
\begin{equation}\label{Eq:Con2:crol7}
    I+N^{\frac{1}{2}}C_1A_{1}^{-1}B_{1}N^{\frac{1}{2}}-N^{\frac{1}{2}%
}C_{3b}P_{2}^{-1}C_{3b}^{T}N^{\frac{1}{2}}>0,
\end{equation}  where  $P_{2}=C_{3a}^{T}B_{3b}^{T}(B_{3b}B_{3b}^{T})^{-1}$. Moreover, for the case in which $N$ is
negative  semidefinite, we will use the condition
\begin{equation}\label{Eq:Con3:crol7}
    \det(I-\tilde{N}C_1A_{1}^{-1}B_{1}\tilde{N}+\tilde{N}C_{3b}P_{2}^{-1}C_{3b}^{T}\tilde{N})\neq0,
\end{equation} where  $\tilde{N}=(-N)^{\frac{1}{2}}$.

%==========================
%==========================
%Corollary 33333333....222222222222==
%==========================
%==========================

\begin{corollary}\label{min:result:clo2-jour}
Suppose that the matrix $C_{3a}^{T}\bar{G}(0)C_{3a}$ is
non-singular, $k\neq0$, and $n_2=0$. Also, suppose that   the
transfer function matrix $G(s),$ with the minimal state space
realization (\ref{eq:xdot}) is NI and the transfer
function matrix $\bar{G}(s),$ with the minimal state space
realization in (\ref{SNI:eq:xdot}), is SNI. If $N$
in \eqref{N-g-croll-2} is positive semidefinite, then the
closed-loop positive-feedback interconnection  between $G(s)$ and
$\bar{G}(s)$ is internally stable   if and only if conditions
\eqref{Eq:Con1:crol7} and \eqref{Eq:Con2:crol7} are satisfied. Also,
if $N$ in \eqref{N-g-croll-2} is negative semidefinite, then the
closed-loop positive-feedback interconnection between $G(s)$ and
$\bar{G}(s)$ is internally stable   if and only if conditions
\eqref{Eq:Con1:crol7} and \eqref{Eq:Con3:crol7} are satisfied.

\end{corollary}

The proof of this corollary is given at the end of this appendix.

%of Corollary \ref{min:result:clo2-jour}:
%\emph{Proof: } Note that for the system
%\eqref{eq:xdot}-\eqref{J_C_F} corresponding to the case of this
%corollary, the  conditions \eqref{Eq:Con1:th2}-\eqref{Eq:Con3:th2}
%in Theorem \ref{Pre_main_re}
% reduce to conditions \eqref{Eq:Con1:crol7}-\eqref{Eq:Con3:crol7}. Then the
%proof of the corollary proceeds in an identical fashion to the proof
%of Theorem \ref{Pre_main_re} for the special case being considered,
%where the matrix $P$ defined in \eqref{Eq:P:NI:th2} becomes a matrix
%of the form $P=\begin{bmatrix}
%P_{1} & 0 &0 \\
%0 & 0 &0 \\
%0 & 0& P_2
%\end{bmatrix}$. \hfill $\blacksquare$

%==========================
%==========================
%==========================
%==========================
%==========================

The next corollary considers the case  when $n\neq0$ and  $k=0$;
 i.e., the $A$ matrix in the minimal state realization of $G(s)$  \eqref{eq:xdot}-\eqref{J_C_F} has the   block diagonal
 form $A=\begin{bmatrix}
A_{1} & 0  \\
 0 & A_{2}%
\end{bmatrix}$. In this  case, when $n\neq0$ and $k=0$, the matrix $N$ in \eqref{N-g-lemma-1} will
be given by
\begin{align}\label{N-g-croll-1}
N=\mathcal{P}(\bar{G}(0),C_{2})%\bar{G}(0)-\bar{G}(0)C_{2}(C_{2}^{T}\bar{G}(0)C_{2})^{-1}C_{2}^{T}\bar{G}(0),
 \end{align}
where the matrix $C_{2}^{T}\bar{G}(0)C_{2}$ is assumed to be
non-singular.

%\begin{assumption}
%\label{assm:siso3:crol6} The matrix $C_{2}^{T}\bar{G}(0)C_{2}$ is
%assumed to be non-singular. Also, the matrix $N$ in
%\eqref{N-g-croll-1} is assumed to be either positive semidefinite or
%negative semidefinite.
%\end{assumption}

We will use the following conditions in the next corollary  which
 corresponds to conditions
\eqref{Eq:Con1:th2}-\eqref{Eq:Con3:th2} in Theorem
\ref{Pre_main_re}. The first condition to be considered is
\begin{equation}\label{Eq:Con1:crol6}
    C_{2}^{T}\bar{G}(0)C_{2} <0.
\end{equation}
Also, for the case in which $N$ in \eqref{N-g-croll-1} is positive
semidefinite, we will use the condition
\begin{equation}\label{Eq:Con2:crol6}
    I+N^{\frac{1}{2}}C_1A_{1}^{-1}B_{1}N^{\frac{1}{2}}>0.
\end{equation}
 Moreover, for the case in which $N$ in
\eqref{N-g-croll-1} is negative  semidefinite, we will use the
condition
\begin{equation}\label{Eq:Con3:crol6}
    \det(I-\tilde{N}C_1A_{1}^{-1}B_{1}\tilde{N})\neq0,
\end{equation}where  $\tilde{N}=(-N)^{\frac{1}{2}}$.

\begin{corollary}\label{min:result:clo1-jour}
Suppose that $C_{2}^{T}\bar{G}(0)C_{2}$ is non-singular, $n_2\neq0$,
and $k=0$. Also, suppose that  the transfer function matrix $G(s),$
with the minimal state space realization
(\ref{eq:xdot})  is NI and  the transfer function
matrix $\bar{G}(s),$ with the minimal state space realization in
(\ref{SNI:eq:xdot}), is SNI. If $N$ in
\eqref{N-g-croll-1} is positive semidefinite, then the closed-loop
positive-feedback interconnection  between $G(s)$ and $\bar{G}(s)$
is internally stable   if and only if conditions
\eqref{Eq:Con1:crol6} and \eqref{Eq:Con2:crol6} are satisfied. Also,
if $N$ in \eqref{N-g-croll-1} is negative semidefinite, then the
closed-loop positive-feedback interconnection between $G(s)$ and
$\bar{G}(s)$ is internally stable   if and only if conditions
\eqref{Eq:Con1:crol6} and \eqref{Eq:Con3:crol6} are satisfied.
\end{corollary}

The proof of this corollary is given at the end of this appendix.

In order to prove  Theorem \ref{Pre_main_re} and
Corollaries \ref{Pre_main_re-G1-G2-0}-\ref{min:result:clo1-jour}, we
will use  the following lemmas. First, Lemma \ref{geees}  gives  expressions for the quantities
$G_0,G_1$ and $G_2$ in \eqref{f_G0} in terms of the
state space realization \eqref{eq:xdot}-\eqref{J_C_F}.
\begin{lemma}\label{geees}
Suppose that $G(s)$ has a minimal state space realization   \eqref{eq:xdot}-\eqref{J_C_F}. Then
the quantities $G_0,G_1$ and $G_2$ defined in \eqref{f_G0} are given as follows:
\begin{align}
G_2&%=\underset{s\longrightarrow 0}{\lim }s^2G(s)
=C_{3a}B_{3b}, \label{Jour-f_G2}, \\
G_1&%=\underset{s\longrightarrow 0}{\lim }s(G(s)-\frac{G_2}{s^2})
=C_{2}B_{2}+C_{3}B_{3},\label{Jour-f_G1}
 \\
G_0&%=\underset{s\longrightarrow 0}{\lim}(G(s)-\frac{G_2}{s^2}-\frac{G_1}{s})
=-C_1A_{1}^{-1}B_{1}\label{Jour-f_G0}.
\end{align}
\end{lemma}
\begin{proof}
This lemma follows immediately from  \eqref{parelle-g}.
\end{proof}

Now,   Lemmas \ref{full-rank}, \ref{NI-lemma-jordon-f}, \ref{b2-b3b-c2-c3a} give some useful  properties of the minimal
state space realization   \eqref{eq:xdot}-\eqref{J_C_F}.
\begin{lemma}\label{full-rank}
Suppose that the transfer function matrix $G(s)$  has a minimal
state space realization   \eqref{eq:xdot}-\eqref{J_C_F}. Then, the
matrix %$C_2,C_{3a}$,
$\begin{bmatrix}
C_{2} & C_{3a} \end{bmatrix}$ is of full column rank, %$B_{2}$, $B_{3b}$
and the matrix $\begin{bmatrix} B_{2} \\ B_{3b}
\end{bmatrix}$  is of full row
rank. Also  $m\geq k+n_2$ and the subsystem with  realization
$\begin{bmatrix}
\begin{array}{c|c}
A_1 & B_1 \\ \hline C_1 & 0
\end{array}
\end{bmatrix}$ is minimal.
\end{lemma}

\begin{proof}
 Since  the state space realization   $\begin{bmatrix}
\begin{array}{c|c}
A & B \\ \hline C & 0
\end{array}
\end{bmatrix}$  is  minimal, the pair $(A,C)$ is
observable and the pair $(A,B)$ is controllable. Also, the corresponding
observability matrix is given by
\begin{align*}
O(A,C)=\begin{bmatrix}C\\CA\\CA^2\\ \vdots \\ CA^{n-1}
\end{bmatrix}=\begin{bmatrix}C_1 &C_2& C_{3a} & C_{3b}\\C_1A_1 &0& 0 & C_{3a}\\C_1A_{1}^{2} &0&0 & 0\\ \vdots &\vdots &\vdots &\vdots\\ C_1A_{1}^{n-1} &0&0 & 0
\end{bmatrix}.%\label{Obser1}
\end{align*}
Since the pair $(A,C)$ is observable, it follows that  the observability matrix $O(A,C)$
%in \eqref{Obser1}
is of full rank. This implies that the pair
$(A_1,C_1)$ is observable. Also, since the observability matrix
$O(A,C)$ is of full  rank, it follows that  $C_2$,
$C_{3a}$ and $\begin{bmatrix}
C_{2} & C_{3a}
\end{bmatrix}$ are of full  rank. Furthermore, it follows that  $m\geq k+n_2$. Similarly, since the pair $(A,B)$ is controllable, it follows  that the  corresponding
controllability matrix
\begin{align}
\mathcal{C}(A,B)=&\begin{bmatrix}B&AB&A^2B& \cdots & A^{n-1}B
\end{bmatrix}\notag\\=&\begin{bmatrix}B_1 & A_1B_1& A_{1}^{2}B_1& \cdots & A_{1}^{n-1}B_1\\
B_2&0&0&\cdots &0\\ B_{3a}&B_{3b}&0&\cdots &0 \\ B_{3b}&0&0&\cdots &0
\end{bmatrix},\notag%\label{Contr1}
\end{align} is of full rank. Hence,  the pair $(A_1,B_1)$ is controllable and the
matrices $B_2$, $B_{3b}$ and $\begin{bmatrix}
B_{2} \\ B_{3b}
\end{bmatrix}$ are of full rank. Also, since the
pair $(A_1,C_1)$ is observable and the pair $(A_1,B_1)$ is
controllable, it follows that $\begin{bmatrix}
\begin{array}{c|c}
A_1 & B_1 \\ \hline C_1 & 0
\end{array}
\end{bmatrix}$ is a minimal realization.
\end{proof}

\begin{lemma}\label{NI-lemma-jordon-f}
Suppose that the transfer function matrix $G(s),$  with the minimal
state space  realization \eqref{eq:xdot}-\eqref{J_C_F}, is NI. Then,
there exist symmetric  matrices  $P_1>0,P_{2}>0,$ and matrices
$L_1,$ $W$ such that
\begin{align}
&P_{1}A_{1}+A_{1}^{T}P_{1} =-L_{1}^{T}L_{1}  \label{Sim_LL1}, \\
&P_{1}B_{1}-A_{1}^{T}C_{1}^{T} =-L_{1}^{T}W \label{Sim_LW},\\
&P_{2}B_{3b} =C_{3a}^{^{T}} \label{Sim_LW2},\\
&C_{1}B_{1}+B_{1}^{T}C_{1}^{T}+C_{2}B_{2}+B_{2}^{T}C_{2}^{T}+C_{3}B_{3}+B_{3}^{T}C_{3}^{T}
\notag\\&=W^{T}W \label{Sim_WW1}.
\end{align}
Furthermore,

\begin{align}
P_{2}=C_{3a}^{T}B_{3b}^{T}(B_{3b}B_{3b}^{T})^{-1} , \label{p2-lemma4-n}
\end{align}
and
\begin{align}
G_1+G_1^T&=C_{2}B_{2}+B_{2}^{T}C_{2}^{T}+C_{3}B_{3}+B_{3}^{T}C_{3}^{T}\notag\\&=\left(
W^{T}+C_{1}P_{1}^{-1}L_{1}^{T}\right) \left(
W+L_{1}P_{1}^{-1}C_{1}^{T}\right)\geq0.\label{NI_WW4}
\end{align}
\end{lemma}

\begin{proof}
Consider the  transfer function matrix $G(s)$ with the minimal state
space realization  \eqref{eq:xdot}-\eqref{J_C_F}. Also, define the  transfer function matrix $R(s)=sG(s)$. Using
\eqref{parelle-g}, it follows that %It follows from
%\eqref{J_C_F} that
%\begin{align}
%G(s)&=C_{1}(sI-A_{1})^{-1}B_{1}+C_{2}(sI-A_{2})^{-1}B_{2}+C_{3}(sI-A_{3})^{-1}B_{3}\notag\\
%&=C_{1}(sI-A_{1})^{-1}B_{1}+\frac{C_{2}B_{2}+C_{3}B_{3}}{s}+\frac{C_{3a}B_{3b}}{s^2}.\label{parelle-g}
%\end{align}
\begin{align}
R(s)=&sC_{1}(sI-A_{1})^{-1}B_{1}+\frac{C_{3a}B_{3b}}{s}+C_{2}B_{2}+C_{3}B_{3}\notag\\
=&C_{1}A_{1}(sI-A_{1})^{-1}B_{1}+C_{1}B_{1}+\frac{C_{3a}B_{3b}}{s}\notag\\
\qquad &+C_{2}B_{2}+C_{3}B_{3}.\label{R-parelle-g}
\end{align}
This implies that $R(s)$ has a  state space realization
$\begin{bmatrix}
\begin{array}{c|c}
A_r & B_r \\ \hline C_r & D_r
\end{array}
\end{bmatrix}
$
where $A_r=\begin{bmatrix}
A_1 & 0 \\
0 & 0%
\end{bmatrix}$, $ B_r=\begin{bmatrix}
B_1  \\
B_{3b}%
\end{bmatrix}$, $C_r=\begin{bmatrix}
C_1 A_1& C_{3a}\end{bmatrix}$ and
$D_r=C_{1}B_{1}+C_{2}B_{2}+C_{3}B_{3}$. Using the same argument as  in
the proof of Lemma 2, it follows that   the rank of
the matrix formed from the  first and last  columns in    $O(A_r,C_r)$ is equal to the rank
of the matrix formed  from the  first and third  columns in (56), where, $A_1$ is
invertible. This implies that the matrix $O(A_r,C_r)$ is of full
rank; i.e., the  pair $(A_r,C_r)$ is observable. Similarly,  the pair
$(A_r,B_r)$ is controllable. This implies that the state space
realization $\begin{bmatrix}
\begin{array}{c|c}
A_r & B_r \\ \hline C_r & D_r
\end{array}
\end{bmatrix}$ is  minimal.

We now show that $R(s)$ is positive real; e.g., see page 47 in
\cite{brogliato-bk2007} for a definition of positive real transfer
function matrices. Since $G(s)$ is NI, it follows that $j \left(
G(j\omega )-G(j\omega )^{\ast }\right) \geq 0, $ for all $\omega
>0$ such that $j\omega $  is not a pole of $G(s)$. Then
 given any such $\omega > 0$,
$ R(j\omega )+R(j\omega )^{\ast } =j\omega \left( G(j\omega
)-G(j\omega )^{\ast }\right) \geq 0, $ and $\overline{\left(
R(j\omega )+R(j\omega )^{\ast }\right) }\geq 0 $.  This implies
that $ R(-j\omega )+R(-j\omega )^{\ast } \geq 0$ for all  $\omega
>0$, since $ \overline{R(j\omega )}=R(-j\omega  )$. Hence, $ R(j\omega )+R(j\omega )^{\ast } \geq 0$
for all $\omega <0$ such that $j\omega$ is not a pole of $G(s)$.
Therefore , $ R(j\omega )+R(j\omega )^{\ast }\geq 0$ for all $\omega
\in (-\infty ,\infty )$ such that $j\omega $ is not a pole of
$G(s)$.

Now, consider the case where $j\omega_{0}$ is a pole of $G(s)$ and $\omega _{0}=0.$ In the case where $C_{3a}B_{3b}=0$, the transfer function matrix $R(s)=C_{1}A_{1}(sI-A_{1})^{-1}B_{1}+C_{1}B_{1}+C_{2}B_{2}+C_{3}B_{3}$ will have no pole at the origin. This implies that $R(0)$ is finite.
Since $ R(j\omega )+R(j\omega )^{\ast } \geq 0$ for all $%
\omega >0$ such that $j \omega$ is not a pole of $G(s) $ and
$R(j\omega )$ is continuous at $\omega=0$, this implies that $R(0)+R(0)^{\ast
}\geq 0$. In the case where $C_{3a}B_{3b}\neq0$, the transfer function matrix $R(s)$ is as given  in \eqref{R-parelle-g}. Since $G(s)$ is NI, then
$\underset{s\longrightarrow 0}{\lim }s^2G(s)\geq0$ which implies
that $\underset{s\longrightarrow 0}{\lim }sR(s)\geq0$.

 If $j\omega _{0}$
is a pole of $G(s)$ and $\omega _{0}>0$, then $G(s)$ can be factored as $\frac{1}{s^{2}+\omega _{0}^{2}%
}F(s)$, which according to the definition for NI systems implies
that the residue matrix $ K_{0}=\frac{1}{2\omega _{0}}F(j\omega
_{0})$ is positive semidefinite Hermitian. Hence, $F(j\omega
_{0})=F(j\omega _{0})^{\ast }\geq 0$. Now, the residue matrix of
$R(s)$ at $j\omega _{0}$  with \ $\omega _{0}>0$ is given by,
\begin{eqnarray*}
\underset{s\longrightarrow j\omega _{0}}{\lim }(s-j\omega _{0})R(s) &=&%
\underset{s\longrightarrow j\omega _{0}}{\lim }(s-j\omega _{0})s
G(s), \\
&=&\underset{s\longrightarrow j\omega _{0}}{\lim }(s-j\omega _{0})s\frac{1}{%
s^{2}+\omega _{0}^{2}}F(s),\\
 &=&\frac{1}{2}F(j\omega _{0})
\end{eqnarray*}%
which is positive semidefinite Hermitian. %Moreover  $G(s)$ and
%$F(s)$ has no infinite pole since $G(s)$ and $F(s)$ is strictly
%proper.
Hence, we can conclude that  $R(s)$ is positive real; see page 47 in  \cite{brogliato-bk2007}.
Using the KYP lemma (e.g., see Lemma 3.1 in \cite{brogliato-bk2007}), it now follows
 that there exist  matrices  $P_r>0, L$ and $W$ such that
\begin{align}
P_rA_r+A_r^{T}P_r =&-L^{T}L,  %\label{PR-NI_LMI1}
\nonumber \\
P_rB_r-C_r^{T} =&-L^{T}W, %\label{PR-NI_LMI2}
\nonumber \\
D_r+D_r^{T} =&W^{T}W. \label{PR-NI_LMI3}
\end{align}
If we write $P_r=\begin{bmatrix}
P_{1} & P_{12} \\
P_{12}^T & P_{2}%
\end{bmatrix}%
$ and $L=\begin{bmatrix}
L_{1} & L_{2}%
\end{bmatrix}$, it follows from \eqref{PR-NI_LMI3} that
\begin{align}
&\begin{bmatrix}
P_{1} & P_{12} \\
P_{12}^{T} & P_{2}%
\end{bmatrix}%
\begin{bmatrix}
A_{1} & 0 \\
0 & 0%
\end{bmatrix}%
+%
\begin{bmatrix}
A_{1}^{T} & 0 \\
0 & 0%
\end{bmatrix}%
\begin{bmatrix}
P_{1} & P_{12} \\
P_{12}^{T} & P_{2}%
\end{bmatrix}
\notag\\&=-%
\begin{bmatrix}
L_{1}^{T} \\
L_{2}^{T}%
\end{bmatrix}%
\begin{bmatrix}
L_{1} & L_{2}%
\end{bmatrix},
\notag\\
\Leftrightarrow&\begin{bmatrix}
P_{1}A_{1}+A_{1}^{T}P_{1} & A_{1}^{T}P_{12}\\
P_{12}^{T}A_{1} &0%
\end{bmatrix}%
=-%
\begin{bmatrix}
L_{1}^{T}L_{1} & L_{1}^{T}L_{2} \\
L_{2}^{T}L_{1} & L_{2}^{T}L_{2}%
\end{bmatrix}.
\label{PR-LL} %&\Leftrightarrow\begin{bmatrix}
%P_{1}A_{1}+A_{1}^{T}P_{1} & A_{1}^{T}P_{12}\\
%P_{12}^{T}A_{1} &0%
%\end{bmatrix}%
%=-%
%\begin{bmatrix}
%L_{1}^{T}L_{1} & 0 \\
%0 & 0%
%\end{bmatrix}.
%\label{PR-LL}
\end{align}
Hence $L_2=0$ and since $A_1$ is a nonsingular matrix, it also
follows   that   $P_{12}=0$. Also, \eqref{PR-LL} implies that
\eqref{Sim_LL1} is satisfied.
%\begin{align*}
%P_{1}A_{1}+A_{1}^{T}P_{1}=L_{1}^{T}L_{1}.\label{L1L1-PR}
%\end{align*}
From \eqref{PR-NI_LMI3}, it follows that
\begin{align*}
\begin{bmatrix}
P_{1} & 0 \\
0 & P_{2}%
\end{bmatrix}%
\begin{bmatrix}
B_{1} \\
B_{3b}%
\end{bmatrix}%
-%
\begin{bmatrix}
A_{1}^T C_{1}^{T} \\
C_{3a}^{T}%
\end{bmatrix}
&=-%
\begin{bmatrix}
L_{1}^{T} \\
0%
\end{bmatrix}W,
\end{align*}
which  implies \eqref{Sim_LW} and  \eqref{Sim_LW2}.  Lemma \ref{full-rank} implies that $B_{3b}$ is of full rank and hence,  \eqref{Sim_LW2} implies that \eqref{p2-lemma4-n} is also satisfied.
%\begin{align}
%P_{2}B_{3b}=C_{3a}^{T},\notag%\label{L1W2-PR}
%\end{align}
%and
%\begin{align}
% P_{1}B_{1}-C_{1}^{T}A_{1}^T =L_{1}^{T}W.\notag%\label{L1W1-PR}
%\end{align}
From \eqref{PR-NI_LMI3}, it follows that  \eqref{Sim_WW1} holds.
%\begin{align}
%C_{1}B_{1}+B_{1}^{T}C_{1}^{T}+C_{2}B_{2}+B_{2}^{T}C_{2}^{T}+C_{3}B_{3}+B_{3}^{T}C_{3}^{T}
%=W^{T}W. \notag%\label{WW1-PR}.
%\end{align}
% Hence,   \eqref{PR-NI_LMI1}, \eqref{PR-NI_LMI2} and
% \eqref{PR-NI_LMI3} imply \eqref{Sim_LL1}-\eqref{Sim_WW1}.
 Also, using \eqref{Sim_LW}, we can write $B_1$ as,
\begin{equation*}%\label{NI_B1}
B_{1} =P_{1}^{-1}(A_{1}^{T}C_{1}^{T}-L_{1}^{T}W).
\end{equation*}
 Substituting   this and \eqref{Sim_LL1} %and \eqref{NI_B1}
into \eqref{Sim_WW1}, it
 follows that
\begin{align}
&C_{2}B_{2}+B_{2}^{T}C_{2}^{T}+C_{3}B_{3}+B_{3}^{T}C_{3}^{T}\notag\\&=\left(
W^{T}+C_{1}P_{1}^{-1}L_{1}^{T}\right) \left(
W+L_{1}P_{1}^{-1}C_{1}^{T}\right)\geq0.\notag%\label{NI_WW4n}
\end{align} Using  \eqref{Jour-f_G0} in Lemma \ref{geees}, this  implies  \eqref{NI_WW4}. This completes the proof.
\end{proof}

\begin{lemma}\label{b2-b3b-c2-c3a}
Suppose that the transfer function matrix $G(s)$  with the minimal
state space  realization \eqref{eq:xdot}-\eqref{J_C_F} is NI. Then,
there exists an invertible matrix $R_d$ such that $
\begin{bmatrix}
B_{2} \\
B_{3b}%
\end{bmatrix}%
=R_d%
\begin{bmatrix}
C_{2}^{T} \\
C_{3a}^{T}%
\end{bmatrix}.%
$ Also, if $x \in \mathcal{N}\left(
\begin{bmatrix}
C_{2}^{T} \\
C_{3a}^{T}%
\end{bmatrix}\right)$, then $x \in \mathcal{N}(B_{3a}+P_2^{-1}
C_{3b}^T)$, where  the matrix $P_2$ is defined as in Lemma
\ref{NI-lemma-jordon-f}. Here, $\mathcal{N}(\cdot)$ denotes  as the
null space of a matrix.

\end{lemma}
\begin{proof}
Suppose that $x \in \mathcal{N}\left(
\begin{bmatrix}
C_{2}^{T} \\
C_{3a}^{T}%
\end{bmatrix}\right)$. It follows that
$
\begin{bmatrix}
C_{2}^{T} \\
C_{3a}^{T}%
\end{bmatrix}%
x=0$. Hence  using \eqref{Sim_LW2} in Lemma \ref{NI-lemma-jordon-f}, it follows that there exists a matrix $P_2>0$ such that
%$ x^{T}C_{2}=0, x^{T}C_{3a}=0\text{ and}
$P_{2}B_{3b}x=0$. Therefore, $x^{T}C_{2}=0\text{,
}x^{T}C_{3a}=0\text{ and }B_{3b}x=0.$ Hence, using \eqref{Jour-f_G0}
it follows that
\begin{align*}
& x^{T}(G_1+G_1^{T})x=x^{T}G_1x+x^{T}G_1^{T}x. \end{align*} Using
the fact that  $x^{T}G_1x$ is a scalar, this implies
\begin{align*}
& x^{T}(G_1+G_1^{T})x=2x^{T}G_1x\\
\Rightarrow &2x^{T}(C_{2}B_{2}+C_{3a}B_{3a}+C_{3b}B_{3b})x=0,\\
\Rightarrow &(G_1+G_1^{T})x=0,
\end{align*}
since, $G_1+G_1^{T}\geq 0$  using \eqref{NI_WW4} in Lemma
\ref{NI-lemma-jordon-f}. Hence,
%\Rightarrow&
%(C_{2}B_{2}+C_{3a}B_{3a}+C_{3b}B_{3b}+B_{2}^{T}C_{2}^{T}+B_{3a}^{T}C_{3a}^{T}+B_{3b}^{T}C_{3b}^{T})x=0,
\begin{align*}
&(C_{2}B_{2}+C_{3a}B_{3a}+B_{3b}^{T}C_{3b}^{T})x=0, \\
\Rightarrow
&(C_{2}B_{2}+C_{3a}B_{3a}+C_{3a}P_{2}^{-1}C_{3b}^{T})x=0,
\end{align*}
 using \eqref{Sim_LW2} in Lemma \ref{NI-lemma-jordon-f}. Therefore,
 \begin{align*}
 &
\begin{bmatrix}
C_{2} & C_{3a}%
\end{bmatrix}%
\begin{bmatrix}
B_{2} \\
B_{3a}+P_{2}^{-1}C_{3b}^{T}%
\end{bmatrix}%
x=0, \\
\Rightarrow&
\begin{bmatrix}
B_{2} \\
B_{3a}+P_{2}^{-1}C_{3b}^{T}%
\end{bmatrix}%
x=0,\end{align*}  since $
\begin{bmatrix}
C_{2} & C_{3a}%
\end{bmatrix}$
 is full rank using Lemma \ref{full-rank}. Therefore, \begin{align*}
 B_{2}x=0 \text{ and } (B_{3a}+P_{2}^{-1}C_{3b}^{T})x=0.
\end{align*}
  This implies that  $x \in \mathcal{N}(B_{3a}+P_2^{-1} C_{3b}^T)$. Thus, we have established the second part  of the lemma. Also since $B_{2}x=0 $ and $B_{3b}x=0, $   it follows that
$\begin{bmatrix}
B_{2} \\
B_{3b}%
\end{bmatrix}%
x=0.
$
This implies that if  $x \in \mathcal{N}\left(
\begin{bmatrix}
C_{2}^{T} \\
C_{3a}^{T}%
\end{bmatrix}\right)$, then
$x \in \mathcal{N}\left(
\begin{bmatrix}
B_{2} \\
B_{3b}%
\end{bmatrix}%
\right).$

Similarly, suppose that  $ x \in \mathcal{N}\left(
\begin{bmatrix}
B_{2} \\
B_{3b}%
\end{bmatrix}\right)$, and hence  $
\begin{bmatrix}
B_{2} \\
B_{3b}%
\end{bmatrix}%
x=0$. Therefore, $B_{2}x=0, B_{3b}x=0\text{ and
}C_{3a}^Tx=0.$
Hence using \eqref{Jour-f_G1} in Lemma \ref{geees}, it follows that
\begin{align*}
& x^{T}(G_1+G_1^{T})x =2x^{T}G_1x, \\ \Rightarrow
&2x^{T}(C_{2}B_{2}+C_{3a}B_{3a}+C_{3b}B_{3b})x=0,\\
\Rightarrow &x^{T}(G_1+G_1^{T})=0, \end{align*} since,
$G_1+G_1^{T}\geq 0$ using \eqref{NI_WW4} in Lemma
\ref{NI-lemma-jordon-f}. Therefore,
\begin{align*}
&x^{T}(C_{2}B_{2}+C_{3b}B_{3b}+B_{3a}^{T}C_{3a}^{T})=0, \\
\Rightarrow
&x^{T}(C_{2}B_{2}+C_{3b}B_{3b}+B_{3a}^{T}P_2B_{3b})=0,\end{align*}
using \eqref{Sim_LW2} in Lemma \ref{NI-lemma-jordon-f}. Therefore,
\begin{align*}  & x^{T}\begin{bmatrix}
C_{2} & C_{3b}+B_{3a}^{T}P_2%
\end{bmatrix}%
\begin{bmatrix}
B_{2} \\
B_{3b}%
\end{bmatrix}%
=0, \\
\Rightarrow&
x^{T}\begin{bmatrix}
C_{2} & C_{3b}+B_{3a}^{T}P_2%
\end{bmatrix}%
=0,\end{align*} since $
\begin{bmatrix}
B_{2} \\
B_{3b}%
\end{bmatrix}%
$ is full rank using Lemma \ref{full-rank}. Therefore,
\begin{align*}
& x^{T}C_{2}=0, \\
\Rightarrow& %
x^{T}\begin{bmatrix}
C_{2} & C_{3a}%
\end{bmatrix}%
=0.
\end{align*}%
 %Similarly we can show that if $ x \in \mathcal{N}\left(
%\begin{bmatrix}
%B_{2} \\
%B_{3b}%
%\end{bmatrix}\right)$,
Therefore, $ x \in\mathcal{N}\left(
\begin{bmatrix}
C_{2}^{T} \\
C_{3a}^{T}%
\end{bmatrix}\right).%
$ Hence,
\begin{align*} \mathcal{N}\left(
\begin{bmatrix}
B_{2} \\
B_{3b}%
\end{bmatrix}\right)%
=\mathcal{N}\left(
\begin{bmatrix}
C_{2}^{T} \\
C_{3a}^{T}%
\end{bmatrix}\right).%
\end{align*}
 Hence, there exists an invertible matrix $R_d$ such that
\begin{align*}
\begin{bmatrix}
B_{2} \\
B_{3b}%
\end{bmatrix}%
=R_d%
\begin{bmatrix}
C_{2}^{T} \\
C_{3a}^{T}%
\end{bmatrix}.%
\end{align*} This completes the proof.
\end{proof}

The following  lemma, which follows directly  from the proof of a result presented in  \cite{lanzon2008,xiong21010jor}, gives  useful properties of the minimal realization \eqref{SNI:eq:xdot} of the SNI transfer function matrix $\bar{G}(s)$.
%will be used in proving  Theorem \ref{Pre_main_re}.
%stated as following
\begin{lemma}(See the proof of Lemma 6 in\cite{xiong21010jor}\label{SNI-lemma})
 Suppose that the transfer function matrix $\bar{G}(s),$ with
 minimal state space  realization
\eqref{SNI:eq:xdot}, is SNI.
% Let $
%\begin{bmatrix}
%\begin{array}{c|c}
%\bar{A} & \bar{B }\\ \hline \bar{C} & \bar{D}
%\end{array}
%\end{bmatrix}$ be a minimal realization of the transfer function matrix  $%
%\bar{G}(s)$ for the system in (\ref{eq:xdot})-(\ref{eq:y}).
Then,   $\bar{D}=\bar{D}^T$, $\det(\bar{A})\neq0$ and there exists
a matrix  $\bar{P}=\bar{P}^{T}> 0$ such that %$\bar{L},$ and $\bar{W}$
% such that the following LMI is
%satisfied:
%\begin{align}\label{LMI:PR}
%\begin{bmatrix}
%\bar{P}\bar{A}+\bar{A}^{T}\bar{P} & \bar{P}\bar{B}-\bar{A}^{T}\bar{C}^{T} \\
%\bar{B}^{T}\bar{P}-\bar{C}\bar{A} & -(\bar{C}\bar{B}+\bar{B}^{T}\bar{C}^{T})%
%\end{bmatrix}%
%=%
%\begin{bmatrix}
%-\bar{L}^{T}\bar{L} & -\bar{L}^{T}\bar{W} \\
%-\bar{W}^{T}\bar{L} & -\bar{W}^{T}\bar{W}%
%\end{bmatrix}
%\leq 0.
%\end{align}
%Furthermore, the LMI \eqref{LMI:PR} is equivalent to the following
\begin{equation}\label{LMI:I-La}
\bar{A}\bar{P}^{-1}+\bar{P}^{-1}\bar{A}^{T}\leq0 \ \text{and}\ \bar{B}=-\bar{A}\bar{P}^{-1}\bar{C}^{T}.
\end{equation}
\end{lemma}

The following lemma is a simple matrix theory result.
\begin{lemma}\label{math4} (See e.g.,  \cite{lanzon2008})
Given $A \in \mathbb{C}^{n\times n}$  with $j(A-A^{*})\geq0$ and $B \in \mathbb{C}^{n\times n}$  with $j(B-B^{*})>0$, then
\begin{align*}
\det(I-AB)\neq0.
\end{align*}
\end{lemma}

%We also consider an SNI  transfer function matrix $\bar{G}(s)$ with a
%minimal state space realization
%\begin{align}
%\label{SNI:eq:xdot}
%&\dot{x}(t) = \bar{A} x(t)+\bar{B} u(t), \\
%\label{SNI:eq:y} &y(t) = \bar{C} x(t)+\bar{D} u(t),
%\end{align}%
%where $\bar{A} \in \mathbb{R}^{\bar{n} \times \bar{n}},\bar{B} \in
%\mathbb{R}^{\bar{n} \times m},\bar{C} \in \mathbb{R}^{m \times
%\bar{n}},$ and $\bar{D} \in \mathbb{R}^{m \times m}.$

%================================================
%===== Proof of Lemma \eqref{Pre_main_re}========
%================================================

Now, we are in a position to present the  proof of Theorem \ref{Pre_main_re}.

\emph{Proof of Theorem \ref{Pre_main_re}:}
 The internal stability  of the   positive-feedback
interconnection between $G(s)$ and $\bar{G}(s)$ will be guaranteed  by
considering  the  closed loop %transfer function matrix,
%\begin{equation*}
%\Sigma(s)=G(s)(I-\bar{G}(s)G(s))^{-1}
%\end{equation*}%
%with a corresponding
system  matrix defined in \eqref{mat:A:CL} which is given by
\begin{align*}%\label{mat:A}
\breve{A} &=\begin{bmatrix}
A+B \bar{D}C & B \bar{C} \\
\bar{B}C & \bar{A}%
\end{bmatrix}.%
\end{align*}
Here, $A,B$ and $C$ are  defined as in  \eqref{eq:xdot}-\eqref{J_C_F} and
$\bar{A},\bar{B},\bar{C} , \bar{D}$ are  defined as  in
(\ref{SNI:eq:xdot}). To establish internal stability, we  show that the
matrix $\breve{A} $
%in \eqref{mat:A}
is Hurwitz; i.e., all the
eigenvalues  of $\breve{A} $ lie in the open left-half of the
complex plane.

Consider   $T=%
\begin{bmatrix}
P-C^{T} \bar{D}C & -C^{T} \bar{C} \\
-\bar{C}^{T}C & \bar{P}%
\end{bmatrix}
$ to  be a candidate Lyapunov matrix, where
\begin{equation}\label{Eq:P:NI:th2}
P=\begin{bmatrix}
P_{1} & 0 & 0&0 \\
0 & 0 & 0&0 \\
0 & 0 & 0&0 \\
0 & 0 &0& P_{2}%
\end{bmatrix}\geq0,
\end{equation}
 $P_1>0$, $P_2>0$ are defined as in Lemma \ref{NI-lemma-jordon-f} and $\bar{P}>0$ is
defined as in Lemma \ref{SNI-lemma}.

%================================
%========= T =================

\emph{Claim 1}: In the case when the  matrix $N$ in
\eqref{N-g-lemma-1} is  negative semidefinite, then  $T>0$ if and
only if \eqref{Eq:Con1:th2} is satisfied. Also, in the case when the
matrix $N$ in \eqref{N-g-lemma-1} is  positive semidefinite, then
$T>0$ if and only if \eqref{Eq:Con1:th2} and \eqref{Eq:Con2:th2} are
satisfied.

To establish this claim, we first note that since $\bar{G}(s)$ is
SNI, it follows from Lemma \ref{SNI-lemma} that $\bar{P}$ satisfies
 \eqref{LMI:I-La}. This implies that the condition  $T>0$  is equivalent to
 \begin{small}
\begin{align}
& P-C^{T}\bar{D}C-C^{T}\bar{C}\bar{P}^{-1}\bar{C}^{T}C>0, \notag\\
\Leftrightarrow & P-C^{T}(\bar{D}+\bar{C}\bar{P}^{-1}\bar{C}^{T})C>0, \notag\\
\Leftrightarrow & P-C^{T}(\bar{D}-\bar{C}\bar{A}^{-1}\bar{B})C>0 \text{ via \eqref{LMI:I-La} in Lemma  \ref{SNI-lemma},}\notag\\
\Leftrightarrow & P-C^{T}\bar{G}(0)C>0, \notag\\
\Leftrightarrow &
\begin{bmatrix}
P_{1} & 0 & 0&0 \\
0 & 0 & 0&0 \\
0 & 0 & 0&0 \\
0 & 0 &0& P_{2}%
\end{bmatrix}%
-%
\begin{bmatrix}
C_{1}^{T} \\
C_{2}^{T} \\
C_{3a}^{T}\\
C_{3b}^{T}
\end{bmatrix}%
\bar{G}(0)%
\begin{bmatrix}
C_{1} & C_{2} & C_{3a}&C_{3b}%
\end{bmatrix}
>0. \label{T-matrix-lmi}
\end{align}
\end{small}
Furthermore,  using the Schur complement of the LMI in
\eqref{T-matrix-lmi}, it is straightforward to verify   that the
condition $ T>0$ is equivalent to the conditions

\begin{align}
&\begin{bmatrix}
-C_{2}^{T}\bar{G}(0)C_{2} & -C_{2}^{T}\bar{G}(0)C_{3a} \\
-C_{3a}^{T}\bar{G}(0)C_{2} & -C_{3a}^{T}\bar{G}(0)C_{3a}%
\end{bmatrix}>0 \label{eq:shur:2:f}\\ \text{ and }& \notag\\%
&\begin{bmatrix}
P_{1}-C_{1}^{T}\bar{G}(0)C_{1} & -C_{1}^{T}\bar{G}(0)C_{3b} \\
-C_{3b}^{T}\bar{G}(0)C_{1} & P_{3b}-C_{3b}^{T}\bar{G}(0)C_{3b}%
\end{bmatrix}
\notag\\&-
\begin{bmatrix}
-C_{1}^{T}\bar{G}(0)C_{2} & -C_{1}^{T}\bar{G}(0)C_{3a} \\
-C_{3b}^{T}\bar{G}(0)C_{2} & -C_{3b}^{T}\bar{G}(0)C_{3a}%
\end{bmatrix}\notag\\
&\times
\begin{bmatrix}
-C_{2}^{T}\bar{G}(0)C_{2} & -C_{2}^{T}\bar{G}(0)C_{3a} \\
-C_{3a}^{T}\bar{G}(0)C_{2} & -C_{3a}^{T}\bar{G}(0)C_{3a}%
\end{bmatrix}^{-1}\notag\\
&\times \begin{bmatrix}
-C_{2}^{T}\bar{G}(0)C_{1} & -C_{2}^{T}\bar{G}(0)C_{3b} \\
-C_{3a}^{T}\bar{G}(0)C_{1} & -C_{3a}^{T}\bar{G}(0)C_{3b}%
\end{bmatrix}\notag\\
&>0.\label{Eq:shur:pr:2}%\\
%\Leftrightarrow &H_{22}>0\text{ and }%
%\begin{bmatrix}
%H_{11}-H_{12}H_{22}^{-1}H_{21} & H_{13}-H_{12}H_{22}^{-1}H_{23} \\
%H_{31}-H_{32}H_{22}^{-1}H_{21} & H_{3}-H_{32}H_{22}^{-1}H_{23}%
%\end{bmatrix}%
%>0.
\end{align}
%
%\begin{align*}
%&H_{22}>0\text{ and }%
%\begin{bmatrix}
%H_{11} & H_{13} \\
%H_{31} & H_{33}%
%\end{bmatrix}%
%-
%\begin{bmatrix}
%H_{12} \\
%H_{32}%
%\end{bmatrix}%
%H_{22}^{-1}%
%\begin{bmatrix}
%H_{21} & H_{23}%
%\end{bmatrix}%
%>0,\\
%\Leftrightarrow &H_{22}>0\text{ and }%
%\begin{bmatrix}
%H_{11}-H_{12}H_{22}^{-1}H_{21} & H_{13}-H_{12}H_{22}^{-1}H_{23} \\
%H_{31}-H_{32}H_{22}^{-1}H_{21} & H_{3}-H_{32}H_{22}^{-1}H_{23}%
%\end{bmatrix}%
%>0.
%\end{align*}

%This implies that $T>0$ if and only if
Moreover, \eqref{eq:shur:2:f} is equivalent to %$\begin{bmatrix}
%-C_{2}^{T}\bar{G}(0)C_{2} & -C_{2}^{T}\bar{G}(0)C_{3a} \\
%-C_{3a}^{T}\bar{G}(0)C_{2} & -C_{3a}^{T}\bar{G}(0)C_{3a}%
%\end{bmatrix}>0$
%\begin{description}
\begin{align}
%&\begin{bmatrix}
%-C_{2}^{T}\bar{G}(0)C_{2} & -C_{2}^{T}\bar{G}(0)C_{3a} \\
%-C_{3a}^{T}\bar{G}(0)C_{2} & -C_{3a}^{T}\bar{G}(0)C_{3a}%
%\end{bmatrix}%
%>0, \notag\\
%&\Leftrightarrow -%
-\begin{bmatrix}
C_{2}^{T} \\
C_{3a}^{T}%
\end{bmatrix}%
\bar{G}(0)%
\begin{bmatrix}
C_{2} & C_{3a}%
\end{bmatrix}=-\Xi%
>0,  \label{shour0}
\end{align}
where $\Xi$ is defined in \eqref{segma-N}. This is equivalent to
condition  \eqref{Eq:Con1:th2}.

Also, the condition \eqref{Eq:shur:pr:2} %$\begin{bmatrix}
%H_{11}-H_{12}H_{22}^{-1}H_{21} & H_{13}-H_{12}H_{22}^{-1}H_{23} \\
%H_{31}-H_{32}H_{22}^{-1}H_{21} & H_{33}-H_{32}H_{22}^{-1}H_{23}%
%\end{bmatrix}>0$
is equivalent to
\begin{align}
&
\begin{bmatrix}
P_{1} & 0 \\
0 & P_{2}%
\end{bmatrix}%
-%
\begin{bmatrix}
C_{1}^{T} \\
C_{3b}^{T}%
\end{bmatrix}%
\bar{G}(0)%
\begin{bmatrix}
C_{1} & C_{3b}%
\end{bmatrix}%
\notag\\&\;\;\;+%
\begin{bmatrix}
C_{1}^{T} \\
C_{3b}^{T}%
\end{bmatrix}%
\bar{G}(0)%
\begin{bmatrix}
C_{2} & C_{3a}%
\end{bmatrix}%
\Xi^{-1}\notag\\
&\;\;\;\times\begin{bmatrix}
C_{2}^T \\
C_{3a}^T%
\end{bmatrix}\bar{G}(0)\begin{bmatrix}
C_{1} & C_{3b}\end{bmatrix}
>0, \notag\\
%&\Leftrightarrow
%\begin{bmatrix}
%P_{1} & 0 \\
%0 & P_{2}%
%\end{bmatrix}%
%-%
%\begin{bmatrix}
%C_{1}^{T} \\
%C_{3b}^{T}%
%\end{bmatrix}\notag\\
%&\;\;\;\times
%\left( \bar{G}(0)-\bar{G}(0)%
%\begin{bmatrix}
%C_{2} & C_{3a}%
%\end{bmatrix}%
%\Xi^{-1}\begin{bmatrix}
%C_{2}^T \\
%C_{3a}^T%
%\end{bmatrix}
%\bar{G}(0)\right)\notag\\
%&\;\;\;\times
%\begin{bmatrix}
%C_{1} & C_{3b}%
%\end{bmatrix}%
%>0, \notag\\
&\Leftrightarrow
\begin{bmatrix}
P_{1} & 0 \\
0 & P_{2}%
\end{bmatrix}%
-%
\begin{bmatrix}
C_{1}^{T} \\
C_{3b}^{T}%
\end{bmatrix}%
N
\begin{bmatrix}
C_{1} & C_{3b}%
\end{bmatrix}%
>0 \text{ using \eqref{N-g-lemma-1}. } \label{Eq:shur:cond2-th2}
\end{align}
This condition is always satisfied in the case where $N$ is negative
semidefinite. Hence using \eqref{shour0},  we can conclude that
$T>0$ if and only if \eqref{Eq:Con1:th2} is satisfied in the case
when $N$ is negative semidefinite.

Now in the case when $N$ is positive  semidefinite, the condition
\eqref{Eq:shur:cond2-th2} can be rewritten as follows
\begin{equation*}%\label{Eq:shur:cond2-pos-th2}
     P_{f}-C_{f}^{T}NC_{f}>0,
\end{equation*} where
$P_f=\begin{bmatrix}
P_{1} & 0 \\
0 & P_{2}%
\end{bmatrix}>0$ and $C_f=\begin{bmatrix}
C_{1} & C_{3b}%
\end{bmatrix}$.
However, using the Schur complement, this %\eqref{Eq:shur:cond2-pos-th2}
is equivalent to the condition
\begin{align}
%&\Leftrightarrow P_{f}-C_{f}^{T}NC_{f}>0, \notag\\
&
\begin{bmatrix}
I & N^{\frac{1}{2}}C_{f} \\
C_{f}^{T}N^{\frac{1}{2}} & P_{f}%
\end{bmatrix}%
>0, \notag\\
&\Leftrightarrow  I-N^{\frac{1}{2}}C_{f}P_{f}^{-1}C_{f}^{T}N^{\frac{1%
}{2}}>0,\notag\\%\label{shour112}
%\end{align}
%%\end{description}
%
%It follows from \eqref{shour112} that
%\begin{align}
% \notag\\
%& \Leftrightarrow I-N^{\frac{1}{2}}%
%\begin{bmatrix}
%C_{1} & C_{3b}%
%\end{bmatrix}%
%\begin{bmatrix}
%P_{1}^{-1} & 0 \\
%0 & P_{2}^{-1}%
%\end{bmatrix}%
%\begin{bmatrix}
%C_{1}^{T} \\
%C_{3b}^{T}%
%\end{bmatrix}%
%N^{\frac{1}{2}}>0, \notag\\
%& \Leftrightarrow I-N^{\frac{1}{2}}\left(
%C_{1}P_{1}^{-1}C_{1}^{T}+C_{3b}P_{2}^{-1}C_{3b}^{T}\right) N^{\frac{1}{2%
%}}>0, \notag\\
& \Leftrightarrow I-N^{\frac{1}{2}}C_{1}P_{1}^{-1}C_{1}^{T}N^{%
\frac{1}{2}}-N^{\frac{1}{2}}C_{3b}P_{2}^{-1}C_{3b}^{T}N^{\frac{1}{2}%
} >0.\label{shour2}
\end{align}
%where $P_f=\begin{bmatrix}
%P_{1} & 0 \\
%0 & P_{2}%
%\end{bmatrix}>0$ and $C_f=\begin{bmatrix}
%C_{2} & C_{3a}%
%\end{bmatrix}$.
%\end{description}
%\begin{align}
% \Leftrightarrow &H_{2}>0 \notag\\
%\Leftrightarrow &%
%\begin{bmatrix}
%-C_{2}^{T}\bar{G}(0)C_{2} & -C_{2}^{T}\bar{G}(0)C_{3a} \\
%-C_{3a}^{T}\bar{G}(0)C_{2} & -C_{3a}^{T}\bar{G}(0)C_{3a}%
%\end{bmatrix}%
%>0 \notag\\
%\Leftrightarrow &-%
%\begin{bmatrix}
%C_{2}^{T} \\
%C_{3a}^{T}%
%\end{bmatrix}%
%\bar{G}(0)%
%\begin{bmatrix}
%C_{2} & C_{3a}%
%\end{bmatrix}%
%>0.
%\label{shour0}
%\end{align}
Now using \eqref{NI_WW4} in Lemma
\ref{NI-lemma-jordon-f}, we can define a matrix $M$ as
\begin{align*}
 M=W+L_{1}P_{1}^{-1}C_{1}^{T}%\label{NI_W4}
\end{align*} so that

\begin{align}\label{eq:mm1}
 M^TM&=C_{2}B_{2}+B_{2}^{T}C_{2}^{T}+C_{3}B_{3}+B_{3}^{T}C_{3}^{T}\notag\\
    &=C_{2}B_{2}+B_{2}^{T}C_{2}^{T}+C_{3a}B_{3a}+B_{3a}^{T}C_{3a}^{T}\notag\\&\;\;\;+C_{3b}B_{3b}+B_{3b}^{T}C_{3b}^{T}.
\end{align}
 Also using \eqref{Sim_LW} in Lemma \ref{NI-lemma-jordon-f}, we can write $B_1$ as
\begin{align*}
B_{1} =P_{1}^{-1}(A_{1}^{T}C_{1}^{T}-L_{1}^{T}W).%\label{NI_B1-2}
\end{align*}
Substituting  for $W$ in terms of $M$  %\eqref{NI_W4}
into
%\eqref{NI_B1-2}
this expression for $B_{1}$ gives
 %Since $G(s)$ is NI,  it follows from Lemma
%\ref{NI-lemma-jordon-f} that there exist  matrices $P_1>0, P_{2}>0$
%such that  \eqref{Sim_LL1}, \eqref{Sim_LW} and \eqref{Sim_WW1} hold.
%Using \eqref{Sim_LW}, we can write $B_1$ as,
%\begin{align}
%B_{1} =P_{1}^{-1}(A_{1}^{T}C_{1}^{T}-L_{1}^{T}W).\label{NI_B1}
%\end{align}
% Substituting   \eqref{Sim_LL1} and \eqref{NI_B1} in \eqref{Sim_WW1}
%gives
%\begin{align}
%&\left( W^{T}+C_{1}P_{1}^{-1}L_{1}^{T}\right) \left(
%W+L_{1}P_{1}^{-1}C_{1}^{T}\right)
%=C_{2}B_{2}+B_{2}^{T}C_{2}^{T}+C_{3}B_{3}+B_{3}^{T}C_{3}^{T},\label{NI_WW4}
%\end{align}
%which implies that
%\begin{align}
% W =M-L_{1}P_{1}^{-1}C_{1}^{T},\label{NI_W4}
%\end{align} where
%$M^2=C_{2}B_{2}+B_{2}^{T}C_{2}^{T}+C_{3}B_{3}+B_{3}^{T}C_{3}^{T}.$ Also, substituting    \eqref{NI_W4} in \eqref{NI_B1} gives
\begin{align}\label{B1_m}
B_{1}
&=P_{1}^{-1}(A_{1}^{T}C_{1}^{T}-L_{1}^{T}(M-L_{1}P_{1}^{-1}C_{1}^{T}))
\notag \\
&=P_{1}^{-1}A_{1}^{T}C_{1}^{T}-P_{1}^{-1}L_{1}^{T}M-P_{1}^{-1}L_{1}^{T}L_{1}P_{1}^{-1}C_{1}^{T}
\notag \\
&=-A_{1}P_{1}^{-1}C_{1}^{T}-P_{1}^{-1}L_{1}^{T}M.
 \end{align}
Also, from the definition of $N$ in \eqref{N-g-lemma-1}, it follows that
\begin{align}\label{Nc2M1}
%& N%
%\begin{bmatrix}
%C_{2} & C_{3a}%
%\end{bmatrix}%
%=0, \notag\\
& N%
\begin{bmatrix}
C_{2} & C_{3a}%
\end{bmatrix}%
=0.
\end{align}
Therefore, Lemma  \ref{b2-b3b-c2-c3a} implies
\begin{align*}
& N%
\begin{bmatrix}
B_{2}^T & B_{3b}^T%
\end{bmatrix}%
=0,\notag\\
&\Rightarrow N%
(C_{2}B_{2}+B_{2}^{T}C_{2}^{T}+C_{3a}B_{3a}\notag\\&\;\;\;\;\;\;\;\;+B_{3a}^{T}C_{3a}^{T}+C_{3b}B_{3b}+B_{3b}^{T}C_{3b}^{T})N%
=0, \notag\\
&\Rightarrow N(M^TM)N=0, \text{ using \eqref{eq:mm1}}.
\end{align*}
Hence, $M N=0.$
%\begin{align*}
% M N=0.%\label{NM1}
%\end{align*}%
Therefore,
\begin{align}
 M N^{\frac{1}{2}}=0.\label{eq:NM11}
\end{align}
Substituting this into \eqref{B1_m} implies
\begin{align}
&B_{1}N^{\frac{1}{2}}
=-A_{1}P_{1}^{-1}C_{1}^{T}N^{\frac{1}{2}}-P_{1}^{-1}L_{1}^{T}MN^{\frac{1}{2}},
\notag \\
\Rightarrow & B_{1}N^{\frac{1}{2}}
=-A_{1}P_{1}^{-1}C_{1}^{T}N^{\frac{1}{2}},
\notag \\
\Rightarrow & N^{\frac{1}{2}}C_1A_{1}^{-1}B_{1}N^{\frac{1}{2}}
=-N^{\frac{1}{2}}C_1P_{1}^{-1}C_{1}^{T}N^{\frac{1}{2}}.
\label{NI_B1f}
\end{align}
Substituting  \eqref{NI_B1f} into \eqref{shour2} gives the condition
 \begin{align*}
%&\begin{bmatrix}
%H_{1}-H_{12}H_{22}^{-1}H_{21} & H_{13}-H_{12}H_{22}^{-1}H_{23} \\
%H_{31}-H_{32}H_{22}^{-1}H_{21} & H_{3}-H_{32}H_{22}^{-1}H_{23}%
%\end{bmatrix}
%>0 \notag\\
%&\Leftrightarrow
I+ N^{\frac{1}{2}}C_1A_{1}^{-1}B_{1}N^{\frac{1}{2}}-N^{\frac{1}{2%
}}C_{3b}P_{2}^{-1}C_{3b}^{T}N^{\frac{1}{2}}
>0.%\label{shour3}
\end{align*} This is equivalent to condition \eqref{Eq:Con2:th2}. Hence, in the
case when $N$ is positive semidefinite, it follows from this and \eqref{shour0}  that $T>0$  if and only if conditions
\eqref{Eq:Con1:th2} and \eqref{Eq:Con2:th2} are satisfied. This
completes the proof of Claim 1.
%Note that  using \eqref{Sim_LW2} in Lemma \ref{NI-lemma-jordon-f}
%and the fact that $B_{3b}$ is full rank, it follows that the
%symmetric matrix $P_2$ satisfies
%\begin{align}\label{p_2-proof}
%P_2=C_{3a}^{T}B_{3b}^{T}(B_{3b}B_{3b}^{T})^{-1}.
%\end{align}
%===================
%===== end T ============
%==================

Now, observe that
\begin{align}\label{eq:lep:closed:loop}
&T\breve{A}+\breve{A}^{T}T \notag\\=& \ \begin{bmatrix}
P-C^{T} \bar{D}C & -C^{T} \bar{C} \\
-\bar{C}^{T}C & \bar{P}%
\end{bmatrix}
  \times \begin{bmatrix}
A+B \bar{D}C & B \bar{C} \\
\bar{B}C & \bar{A}%
\end{bmatrix}
\notag\\
&\: \: +
\begin{bmatrix}
A+B \bar{D}C & B \bar{C} \\
\bar{B}C & \bar{A}%
\end{bmatrix}
^{T}
 \times\begin{bmatrix}
P-C^{T} \bar{D}C & -C^{T} \bar{C} \\
-\bar{C}^{T}C & \bar{P}%
\end{bmatrix},
\notag\\
=&- \begin{bmatrix}
 C^{T}\bar{D}W^{T}+L^{T}  & C^{T}\bar{W}^{T} \\
\bar{C}^{T}W^{T} &  \bar{L}^{T}
\end{bmatrix}
    \begin{bmatrix}
 W\bar{D}C+L  & W\bar{C} \\
\bar{W}C &  \bar{L}
\end{bmatrix}%
\notag\\
\leq&0.
\end{align}

Together with Claim 1, this  implies  that $\breve{A}$ has all its
eigenvalues  in the closed left-half of the complex plane if and
only if conditions \eqref{Eq:Con1:th2} and \eqref{Eq:Con2:th2} are
satisfied  in the case when $N$ is positive semidefinite; e.g., see
Lemma 3.19 in \cite{Glover1996}. Similarly, in the case when $N$ is
negative semidefinite $\breve{A}$ has all its eigenvalues  in the
closed left-half of the complex plane if and only if condition
\eqref{Eq:Con1:th2} is satisfied.

 %We now show that  $\det (\breve{A})\neq 0$. Indeed, using the
% assumption that
%$(A+B\bar{G}(0)C)$ is nonsingular, we obtain
%
%\begin{align}\label{det_eq3}
%\det (\breve{A}) &=\det (\bar{A})\det ((A+B\bar{D}C-B\bar{C},\left(
%\bar{A}\right) ^{-1}\bar{B}C), \nonumber \\
% &=\det (\bar{A})\det (A+B\bar{G}(0)C), \nonumber\\
% &=\det (\bar{A})\det (A+B\bar{G}(0)C), \nonumber \\
%&\neq0,
%\end{align}
%since $\det (\bar{A})\neq0$ by Lemma \ref{SNI-lemma}.

In order to complete the proof of the sufficiency part of the
theorem, we must show that if conditions \eqref{Eq:Con1:th2} and
\eqref{Eq:Con2:th2} are satisfied  in the case when $N$ is positive semidefinite,  then the matrix $\breve{A}$ can have no
eigenvalues on the $j\omega$ axis.  Similarly, we must show that if conditions \eqref{Eq:Con1:th2} and
\eqref{Eq:Con3:th2} are satisfied in the case that $N$ is negative
semidefinite, then the matrix $\breve{A}$ can have no
eigenvalues on the $j\omega$ axis.

Indeed, using Lemma \ref{math4}, the fact that $G(s)$ is NI and the
fact that $\bar{G}(s)$ is SNI, we conclude that
$\det(I-G(j\omega)\bar{G}(j \omega))\neq 0$  for all $\omega >0$.
This implies that $\breve{A}$ has no eigenvalues on the imaginary
axis for $\omega >0$. Thus, to complete the proof, we will show that
in the case when $N$ is positive semidefinite,  conditions
\eqref{Eq:Con1:th2} and \eqref{Eq:Con2:th2} imply that  $\det
(\breve{A})\neq 0$. Similarly, in the case when $N$ is negative
semidefinite, we will show that  conditions \eqref{Eq:Con1:th2} and
\eqref{Eq:Con3:th2} imply that  $\det (\breve{A})\neq 0$. Indeed,
\begin{align}\label{det_eq3}
\det (\breve{A}) &=\det (\bar{A})\det (A+B\bar{D}C-B\bar{C}
\bar{A} ^{-1}\bar{B}C), \nonumber \\
 &=\det (\bar{A})\det (A+B\bar{G}(0)C).
\end{align}
This implies that $\det (\breve{A})\neq 0$ if  $\det
(A+B\bar{G}(0)C)\neq 0$, since  $\det (\bar{A})\neq0$   using Lemma
\ref{SNI-lemma} and the fact that $\bar{G}(s)$ is SNI. Now, define the  matrix
\begin{align}\label{lamda:matr1}
\Lambda=&\begin{bmatrix}
B_{2}\bar{G}(0)C_{2} & B_{2}\bar{G}(0)C_{3a} \\
B_{3b}\bar{G}(0)C_{2} & B_{3b}\bar{G}(0)C_{3a}%
\end{bmatrix}\notag \\=&\begin{bmatrix}
B_{2} \\
B_{3b}%
\end{bmatrix}%
\bar{G}(0)%
\begin{bmatrix}
C_{2} & C_{3a}%
\end{bmatrix}.
\end{align}
It follows from Lemma \ref{b2-b3b-c2-c3a} that there exists a
non-singular matrix $R_d$ such that
\begin{align}\label{lamda2}
%\Lambda=&\begin{bmatrix}
%B_{2}\bar{G}(0)C_{2} & B_{2}\bar{G}(0)C_{3a} \\
%B_{3b}\bar{G}(0)C_{2} & B_{3b}\bar{G}(0)C_{3a}%
%\end{bmatrix}\notag\\
%=&\begin{bmatrix}
%B_{2} \\
%B_{3b}%
%\end{bmatrix}%
%\bar{G}(0)%
%\begin{bmatrix}
%C_{2} & C_{3a}%
%\end{bmatrix}\notag\\
\Lambda=&R_d\begin{bmatrix}
C^{T}_{2} \\
C^{T}_{3a}%
\end{bmatrix}%
\bar{G}(0)%
\begin{bmatrix}
C_{2} & C_{3a}%
\end{bmatrix}=R_d\Xi.
\end{align}
Since the matrix  $\Xi$ is assumed to be invertible, this implies
that the matrix $\Lambda$ in \eqref{lamda:matr1} is invertible.
% and $N_d=\left( \bar{G}(0)-\bar{G}(0)%
%\begin{bmatrix}
%C_{2} & C_{3a}%
%\end{bmatrix}%
%\Lambda^{-1}%
%\begin{bmatrix}
%B_{2} \\
%B_{3b}%
%\end{bmatrix}%
%\bar{G}(0)\right)$. Using Lemma \ref{b2-b3b-c2-c3a}, it follows that
%there exist a non-singular matrix $R_d$ such that
%\begin{align}\label{lamda2}
%\Lambda=&\begin{bmatrix}
%B_{2}\bar{G}(0)C_{2} & B_{2}\bar{G}(0)C_{3a} \\
%B_{3b}\bar{G}(0)C_{2} & B_{3b}\bar{G}(0)C_{3a}%
%\end{bmatrix}\notag\\
%=&\begin{bmatrix}
%B_{2} \\
%B_{3b}%
%\end{bmatrix}%
%\bar{G}(0)%
%\begin{bmatrix}
%C_{2} & C_{3a}%
%\end{bmatrix}\notag\\
%=&R_d\begin{bmatrix}
%C^{T}_{2} \\
%C^{T}_{3a}%
%\end{bmatrix}%
%\bar{G}(0)%
%\begin{bmatrix}
%C_{2} & C_{3a}%
%\end{bmatrix}=R_d\Xi.
%\end{align}
%Also,
%\begin{align}\label{N-d2}
%N_d=&\left( \bar{G}(0)-\bar{G}(0)%
%\begin{bmatrix}
%C_{2} & C_{3a}%
%\end{bmatrix}%
%\Lambda^{-1}%
%\begin{bmatrix}
%B_{2} \\
%B_{3b}%
%\end{bmatrix}%
%\bar{G}(0)\right)\notag\\
%=&\left( \bar{G}(0)-\bar{G}(0)%
%\begin{bmatrix}
%C_{2} & C_{3a}%
%\end{bmatrix}%
%\Xi^{-1}%
%\begin{bmatrix}
%C^{T}_{2} \\
%C^{T}_{3a}%
%\end{bmatrix}%
%\bar{G}(0)\right)=N.
%\end{align}

Now, substituting  \eqref{J_C_F} into \eqref{det_eq3}, it is
straightforward to verify  that
\begin{small}
%\scalebox{0.7}{
\begin{align}
&\det (A+B\bar{G}(0)C)\notag \\
 =&-\det \Lambda \det \notag\\
&\left(\begin{array}{c}
\begin{bmatrix}
A_{1} & 0 \\
0 & I%
\end{bmatrix}%
+%
\begin{bmatrix}
B_{1} \\
B_{3a}%
\end{bmatrix}
 \\ \times
\left( \bar{G}(0)-\bar{G}(0)%
\begin{bmatrix}
C_{2} & C_{3a}%
\end{bmatrix}%
\Lambda^{-1}%
\begin{bmatrix}
B_{2} \\
B_{3b}%
\end{bmatrix}%
\bar{G}(0)\right)\\ \times
\begin{bmatrix}
C_{1} & C_{3b}%
\end{bmatrix}\end{array}
\right)\notag  \\
=&-\det \Lambda \det \notag\\ &\left(\begin{array}{c}
\begin{bmatrix}
A_{1} & 0 \\
0 & I%
\end{bmatrix}%
+%
\begin{bmatrix}
B_{1} \\
B_{3a}%
\end{bmatrix}\\ \times
\left( \bar{G}(0)-\bar{G}(0)%
\begin{bmatrix}
C_{2} & C_{3a}%
\end{bmatrix}%
(R_d\Xi)^{-1}%
R_d\begin{bmatrix}
C^{T}_{2} \\
C^{T}_{3a}%
\end{bmatrix}
\bar{G}(0)\right)\\ \times
\begin{bmatrix}
C_{1} & C_{3b}%
\end{bmatrix}\end{array}
\right)\notag  \\
=&-\det \Lambda \det \left(
\begin{bmatrix}
A_{1} & 0 \\
0 & I%
\end{bmatrix}%
+%
\begin{bmatrix}
B_{1} \\
B_{3a}%
\end{bmatrix}%
N%
\begin{bmatrix}
C_{1} & C_{3b}%
\end{bmatrix}%
\right) \label{det-G1-G2-0} \\
=&-\det \Lambda \det
\begin{bmatrix}
A_{1} & 0 \\
0 & I%
\end{bmatrix}\notag\\ & \;\;\;\times
\det \left( I+%
\begin{bmatrix}
A_{1} & 0 \\
0 & I%
\end{bmatrix}%
^{-1}%
\begin{bmatrix}
B^{1} \\
B^{3a}%
\end{bmatrix}%
N%
\begin{bmatrix}
C_{1} & C_{3b}%
\end{bmatrix}%
\right) \notag \\
=&-\det \Lambda \det
\begin{bmatrix}
A_{1} & 0 \\
0 & I%
\end{bmatrix}\notag\\ & \;\;\;\times
\det \left( I+%
\begin{bmatrix}
C_{1} & C_{3b}%
\end{bmatrix}%
\begin{bmatrix}
A_{1} & 0 \\
0 & I%
\end{bmatrix}%
^{-1}%
\begin{bmatrix}
B^{1} \\
B^{3a}%
\end{bmatrix}%
N\right) \notag \\
&=-\det \Lambda \det
\begin{bmatrix}
A_{1} & 0 \\
0 & I%
\end{bmatrix}%
\det \left( I+\left( C_{1}A_{1}^{-1}B_{1}+C_{3b}B_{3a}\right)
N\right). \label{det3}
%&=-\det \Lambda \det
%\begin{bmatrix}
%A_{1} & 0 \\
%0 & I%
%\end{bmatrix}%
%\det \left[ I+\left( -G_{0}+C_{3b}B_{3a}\right) N_d\right].
\end{align}
\end{small}

In the case when $N$ is positive semidefinite,  %Using
%\eqref{lamda2}, \eqref{N-d2} and the fact that $N\geq0$ or $N\leq0$
  \eqref{lamda2} and \eqref{det3} imply  that
\begin{align}\label{det321}
&\det (A+B\bar{G}(0)C)\notag\\&=-\det R_d\det\Xi \det
\begin{bmatrix}
A_{1} & 0 \\
0 & I%
\end{bmatrix}%
\notag\\&\;\;\;\;\;\times \det \left[ I+\left( C_{1}A_{1}^{-1}B_{1}+C_{3b}B_{3a}\right) N^{\frac{1}{2}}N^{\frac{1}{2}%
}\right], \notag   \\
&=-\det R_d\det\Xi\det
\begin{bmatrix}
A_{1} & 0 \\
0 & I%
\end{bmatrix}%
\notag\\&\;\;\;\;\;\times\det \left[ I+N^{\frac{1}{2}}\left( C_{1}A_{1}^{-1}B_{1}+C_{3b}B_{3a}\right) N^{\frac{1}{2}%
}\right].
\end{align}
Now using \eqref{Nc2M1}, it follows that the columns of the matrix $N^{\frac{1}{2}}$ are contained in the set  $ \mathcal{N}\left(
\begin{bmatrix}
C_{2}^{T} \\
C_{3a}^{T}%
\end{bmatrix}\right)$. Hence, it follows from the second part  of Lemma \ref{b2-b3b-c2-c3a} that $(B_{3a}+P_2^{-1}C_{3b}^T)N^{\frac{1}{2}}=0$. This implies that
%Using Lemma \ref{b2-b3b-c2-c3a} and the fact that $N^{\frac{1}{2}}
%\in \mathcal{N}\left(
%\begin{bmatrix}
%C_{2}^{T} \\
%C_{3a}^{T}%
%\end{bmatrix}\right)$ it follows that
\begin{equation*}%\label{eq:nul32}
 N^{\frac{1}{2}}C_{3b}B_{3a} N^{\frac{1}{2}}=-N^{\frac{1}{2}%
}C_{3b}P_{2}^{-1}C^{T}_{3b}N^{\frac{1}{2}}.
\end{equation*}
Hence \eqref{det321} can be written as
\begin{align}
&\det (A+B\bar{G}(0)C)\notag \\&=-\det R_d\det\Xi\det
\begin{bmatrix}
A_{1} & 0 \\
0 & I%
\end{bmatrix}%
\notag\\&\;\;\;\;\;\times\det \left[ I+N^{\frac{1}{2}}C_{1}A_{1}^{-1}B_{1}N^{\frac{1}{2}}-N^{\frac{1}{2}%
}C_{3b}P_{2}^{-1}C^{T}_{3b}N^{\frac{1}{2}}\right],
 \notag \\
&=-\det R_d\det\Xi\det
\begin{bmatrix}
A_{1} & 0 \\
0 & I%
\end{bmatrix}%
\notag\\&\;\;\;\;\;\times \det \left[ I+N^{\frac{1}{2}}C_{1}A_{1}^{-1}B_{1}N^{\frac{1}{2}}-N^{\frac{1}{2}}%
C_{3b}P_{3b}^{-1}C^{T}_{3b}N^{\frac{1}{2}}\right].
\end{align}
Since the matrices $R_d, A_{1}, \bar{A}$ are invertible and also using
\eqref{Eq:Con1:th2}-\eqref{Eq:Con2:th2}, \eqref{det_eq3}, it follows
that $\det (\breve{A})\neq 0$ as required.

In the case when $N$ is negative  semidefinite,  we consider the  matrix
$\tilde{N}=(-N)^{\frac{1}{2}}$. Then \eqref{det3} implies  that
\begin{align}\label{det3211}
&\det (A+B\bar{G}(0)C)\notag\\%&=-\det R_d\det\Xi \det
%\begin{bmatrix}
%A_{1} & 0 \\
%0 & I%
%\end{bmatrix}%
%\notag\\&\;\;\;\;\;\times\det \left[ I-\left( C_{1}A_{1}^{-1}B_{1}+C_{3b}B_{3a}\right) \tilde{N}^2%
%\right], \notag   \\
&=-\det R_d\det\Xi\det
\begin{bmatrix}
A_{1} & 0 \\
0 & I%
\end{bmatrix}%
\notag\\&\;\;\;\;\;\times\det \left[ I-\tilde{N}\left(
C_{1}A_{1}^{-1}B_{1}+C_{3b}B_{3a}\right) \tilde{N}\right].
\end{align}
Using \eqref{Nc2M1}, it follows that the columns of the matrix $\tilde{N}$ are contained in the set  
$ \mathcal{N}\left(
\begin{bmatrix}
C_{2}^{T} \\
C_{3a}^{T}%
\end{bmatrix}\right)$. Hence, it follows from the second part  of Lemma \ref{b2-b3b-c2-c3a} that $(B_{3a}+P_2^{-1}C_{3b}^T)\tilde{N}=0$. This implies
%Using Lemma \ref{b2-b3b-c2-c3a} and the fact that $\tilde{N} \in
%\mathcal{N}\left(
%\begin{bmatrix}
%C_{2}^{T} \\
%C_{3a}^{T}%
%\end{bmatrix}\right)$ it follows that
\begin{equation*}%\label{eq:nul32}
 \tilde{N}C_{3b}B_{3a} \tilde{N}=-\tilde{N}C_{3b}P_{2}^{-1}C^{T}_{3b}\tilde{N}.
\end{equation*}
Hence,  \eqref{det3211} can be written as
\begin{align}
&-\det R_d\det\Xi\det
\begin{bmatrix}
A_{1} & 0 \\
0 & I%
\end{bmatrix}%
\notag\\&\;\;\;\;\;\times\det \left[
I-\tilde{N}C_{1}A_{1}^{-1}B_{1}\tilde{N}+\tilde{N}C_{3b}P_{2}^{-1}C^{T}_{3b}\tilde{N}\right]
, \notag \\
&=-\det R_d\det\Xi\det
\begin{bmatrix}
A_{1} & 0 \\
0 & I%
\end{bmatrix}%
\notag\\&\;\;\;\;\;\times \det \left[ I+N^{\frac{1}{2}}C_{1}A_{1}^{-1}B_{1}N^{\frac{1}{2}}-N^{\frac{1}{2}}%
C_{3b}P_{3b}^{-1}C^{T}_{3b}N^{\frac{1}{2}}\right].
\end{align}
Since the matrices $R_d, A_{1}, \bar{A}$ are invertible and also using
\eqref{Eq:Con1:th2}-\eqref{Eq:Con3:th2}, \eqref{det_eq3}, it follows
that $\det (\breve{A})\neq 0$ as required. This completes the proof of the sufficiency part of the theorem.

%Since the matrices $R_d, A_{1}$ are invertible and also using
%\eqref{Eq:Con1:th2} and \eqref{Eq:Con3:th2}, it follows that $\det
%(\breve{A})\neq 0$.

To complete the proof of the necessity part of the theorem, suppose
that the positive-feedback interconnection between the NI transfer
function matrix $G(s)$ and the SNI  transfer function matrix
$\bar{G}(s)$ is internally stable.  This implies that  the matrix
$\breve{A}$ is Hurwitz and hence has all its eigenvalues are in the
open left-half of the complex plane. This together with Claim 1 and
\eqref{eq:lep:closed:loop}  implies that conditions
\eqref{Eq:Con1:th2} and \eqref{Eq:Con3:th2}  are satisfied in the
case when $N$ is negative  semidefinite. Similarly, in the case when
$N$ is positive semidefinite, Claim 1 and \eqref{eq:lep:closed:loop}
implies that conditions \eqref{Eq:Con1:th2} and \eqref{Eq:Con2:th2}
are satisfied. This completes the proof of the theorem.
%Conditions
%\eqref{Eq:Con1:th2} and \eqref{Eq:Con2:th2} in the case when $N$ is
%positive semidefinite or  \eqref{Eq:Con1:th2} and
%\eqref{Eq:Con3:th2} in the case when $N$ is negative semidefinite
%must be satisfied. This completes the proof of the necessity part of
%the theorem. This completes the proof  of the  theorem.
 %\begin{}
% $\blacksquare$
% \end{}
\hfill $\blacksquare$

%==========================
%==========================
%==========================
%==========================
%==========================
%==========================
%Now consider the following    corollaries  %describing    special
%%cases of  Theorem \ref{Pre_main_re}. These corollaries
%which  will be used in proving the remaining   results of the paper.  %Theorem \ref{min:result:clo1} and
%%Corollaries \ref{min:result:clo3}-\ref{min:result:clo2}.
%%==========================
%%==========================
%===Corollar 33333333....11111111111======
%==========================
%==========================

%\begin{corollary}\label{Pre_main_re-G1-G2-0}
%Suppose that the matrix $\Xi$ in \eqref{segma-N} is non-singular and
%the matrix  $N$ in \eqref{N-g-lemma-1} satisfies $N\begin{bmatrix}
%C_{1} & C_{3b}
%\end{bmatrix}=0$.
%Also suppose that the transfer function matrix $G(s),$ with the
%minimal state space realization (\ref{eq:xdot})-(\ref{eq:y}), is NI
%and   the transfer function matrix $\bar{G}(s),$ with the minimal
%state space realization (\ref{SNI:eq:xdot})-(\ref{SNI:eq:y}), is
%SNI. Then the closed-loop positive-feedback interconnection  between
%$G(s)$ and $\bar{G}(s)$ is internally stable   if and only if
%conditions \eqref{Eq:Con1:th2} is satisfied.
%\end{corollary}
%of Corollary \ref{Pre_main_re-G1-G2-0}
\emph{Proof of Corollary \ref{Pre_main_re-G1-G2-0}:} The proof of
this  corollary will proceeds in an almost identical fashion to the
proof of Theorem \ref{Pre_main_re}. Indeed, we first state the
following claim:

\emph{Claim 2}: Assume that the matrix  $N$ in \eqref{N-g-lemma-1}
satisfies $N\begin{bmatrix} C_{1} & C_{3b}
\end{bmatrix}=0$, then  $T>0$ if and only if \eqref{Eq:Con1:th2} is satisfied.

This claim corresponds  to Claim 1 in Theorem \ref{Pre_main_re} when
we relax the conditions on the matrix $N$. The proof of this claim
is similar to the proof of Claim 1 in the proof of Theorem
\ref{Pre_main_re} since \eqref{Eq:shur:cond2-th2} is automatically
satisfied in the case when $N\begin{bmatrix} C_{1} & C_{3b}
\end{bmatrix}=0$.

Also,  the determinant condition in \eqref{det_eq3} will be
automatically satisfied using the fact $N\begin{bmatrix} C_{1} &C_{3b}
\end{bmatrix}=0$ in  \eqref{det-G1-G2-0}. The proof of the corollary then follows as in the proof of Theorem \ref{Pre_main_re}. \hfill $\blacksquare$

\emph{Proof of Corollary \ref{min:result:clo2-jour}: } Note that for
the system \eqref{eq:xdot}-\eqref{J_C_F} corresponding to the case
of this corollary, the  conditions
\eqref{Eq:Con1:th2}-\eqref{Eq:Con3:th2} in Theorem \ref{Pre_main_re}
 reduce to conditions \eqref{Eq:Con1:crol7}-\eqref{Eq:Con3:crol7}. Then the
proof of the corollary proceeds in an identical fashion to the proof
of Theorem \ref{Pre_main_re} for the special case being considered,
where the matrix $P$ defined in \eqref{Eq:P:NI:th2} becomes a matrix
of the form $P=\begin{bmatrix}
P_{1} & 0 &0 \\
0 & 0 &0 \\
0 & 0& P_2
\end{bmatrix}$. \hfill $\blacksquare$

\emph{Proof of Corollary \ref{min:result:clo1-jour}:} First note
that for the  system \eqref{eq:xdot}-\eqref{J_C_F} corresponding to
the case of this corollary, the  conditions
\eqref{Eq:Con1:th2}-\eqref{Eq:Con3:th2} in Theorem \ref{Pre_main_re}
reduce to conditions \eqref{Eq:Con1:crol6}-\eqref{Eq:Con3:crol6}.
Then the proof of the corollary proceeds in an identical fashion to
the proof of Theorem \ref{Pre_main_re} for the special case being
considered, where the matrix $P$ defined in \eqref{Eq:P:NI:th2}
becomes a matrix of the form $P=\begin{bmatrix}
P_{1} & 0 \\
0 & 0%
\end{bmatrix}$. \hfill $\blacksquare$

\section{Appendix B}
\label{appB}

Here, we present the proof of the main results in the paper.

We first show   that any NI system can be transformed to the block
diagonal form given  in \eqref{eq:xdot}-\eqref{J_C_F}.

\begin{lemma}\label{jordon-lemma}
Any NI system with transfer function matrix $G(s)$ and minimal state
space realization \eqref{eq:xdot}, can be transformed
to the  block diagonal form given in \eqref{J_C_F}.
\end{lemma}

\begin{proof}
Suppose that the transfer function matrix $G(s),$  with a minimal
state space realization $\begin{bmatrix}
\begin{array}{c|c}
A & B \\ \hline C & 0
\end{array}
\end{bmatrix}$ is NI. It follows from Theorem 2.1.1 in \cite{chen1998h} that we can find a non-singular  state space
transformation matrix $T$ such that  the matrix $T^{-1}AT $  is in real
Jordon  block diagonal form and the realization $\begin{bmatrix}
\begin{array}{c|c}
T^{-1}AT & T^{-1}B \\ \hline CT & 0
\end{array}
\end{bmatrix}$ is minimal. Also, we can choose this  transformation so
that   the  Jordon blocks of $T^{-1}AT $ are ordered  according to
the magnitudes of the corresponding  eigenvalues of the matrix $A$,
such that the last blocks  correspond to  the zero eigenvalues of
$A$ if they  exist. Furthermore, this transformation can be chosen
so that the Jordan blocks corresponding to the zero eigenvalues are
ordered according to increasing order of the Jordan blocks. Also, a
further transformation can be applied so that the matrix $A_3$
corresponding to the Jordan blocks of order two is of the form
\eqref{matA3}.

Now, we claim that if $G(s)$ is NI, then there are  no  Jordan
blocks corresponding to zero eigenvalues of order greater than or
equal to three. To prove this claim, suppose that there is a Jordon
block of $A$ corresponding to a zero eigenvalue of order greater
than or equal to three. This together with the minimality of the
realization  implies that $G_3=\underset{s\longrightarrow 0}{\lim
}s^{3}G(s)\neq0$ which contradicts  the NI  definition. Thus the
zero eigenvalues of $A$ will only have  Jordon blocks of order one
or two. From this,  it now follows that the matrix $T^{-1}AT $ will
be of the form \eqref{J_C_F}. This completes the proof of the lemma.
\end{proof}

The next lemma is a technical lemma, which will be used in order to
prove  our results.

\begin{lemma}\label{matrix-lemma}
For any full rank matrices $A,B,C$ and $D$ which satisfy $AB=CD$
where $A \in \mathbb{R}^{n \times r},B \in \mathbb{R}^{r \times n},C
\in \mathbb{R}^{n \times r},D \in \mathbb{R}^{r \times n}$ and
$n\geq r$,  there exists an invertible matrix $R$ such that $A=CR$
and $B=R^{-1}D.$
\end{lemma}

\begin{proof}
Since $B$ is of  full rank, and $n\geq r$, $AB=CD$ implies
\begin{align*}
 &ABB^T=CDB^T,\\
 \Rightarrow &A=CDB^T(BB^T)^{-1},\\
 \Rightarrow &A=CR
\end{align*}
where $R=DB^T(BB^T)^{-1}$.
%
% that $\mathcal{N}(A)=\mathcal{N}(C)$, and
%$\mathcal{R}(A)=\mathcal{R}(C)$. Now, let $a_{i}$ be the ith column
%of $A$. It follows that $a_{i}\in \mathcal{R}(C)$. This implies that
%there exists an $n \times 1$ vector $r_{i}$ such that
%$a_{i}=Cr_{i}$. Now define  a $n \times n$ matrix $R$ whose columns are
%the $r_{i}^{\prime s}$. This implies that $A=CR$.

To show that $R$ is nonsingular, suppose that $R$ is singular. Then
there exists a non-zero $n \times 1$ vector $x$ such that $Rx=0$.
This implies that $Ax=0$ which contradicts the fact that $A$ is a
full rank. Hence, that there exists a nonsingular matrix $R$ such
that $A=CR$. Also, since $C$  is of full rank and $n\geq r$, it
follows that $C$ has a left inverse, which implies that $RB=D;$
i.e., $B=R^{-1}D$. This complete the proof.
\end{proof}

 \emph{Proof of Theorem \ref{min:result}:}  Lemma
\ref{jordon-lemma} shows that any strictly proper  NI system can be represented  in
the  block diagonal form
 \eqref{eq:xdot}-\eqref{J_C_F}. This implies that we only need to show the
 equivalence
 of the assumptions and
the conditions \eqref{Eq:Con1:th2}-\eqref{Eq:Con3:th2} in
 Theorem \ref{Pre_main_re} and the assumptions and the  conditions \eqref{Eq:Con1:th1}-\eqref{Eq:Con3:th1} in this theorem.

  First, it is straightforward to verify  that the condition $k\neq0$ is equivalent to the condition   $G_2\neq0$. Also, it follows from \eqref{Jour-f_G0} and \eqref{JJ-m} that there exists a full rank matrix $J$ such that
 \begin{align*}
 C_{3a}B_{3b}=JJ^{T}.%\label{JJ-m2}
 \end{align*}
Also, it follows from Lemma \ref{full-rank} and  Lemma
\ref{matrix-lemma}  that there exists an invertible matrix $X$ such
that $C_{3a}=JX$ and $B_{3b}=X^{-1}J^{T}$. We let
$P_{2}=C_{3a}^{T}B_{3b}^{T}(B_{3b}B_{3b}^{T})^{-1}$ and  note that
$B_{3b}B_{3b}^{T}$ is invertible since $B_{3b}$ is of  full rank.
Then Lemma  \ref{NI-lemma-jordon-f}  implies that $P_2$ is symmetric
and also we obtain
\begin{align}\label{P3b_2}
P_{2} =X^{T}X.
\end{align}
In the case when  $N$ is positive semidefinite, the definition of
$N$ implies that $N
\begin{bmatrix} C_2&C_{3a}\end{bmatrix}=0,$ and hence
$N^{\frac{1}{2}} \begin{bmatrix} C_2&C_{3a}\end{bmatrix}=0.$
 Using \eqref{Jour-f_G0}, it follows that $G_{1}=\begin{bmatrix} C_2&C_{3a}\end{bmatrix}\begin{bmatrix}
B_2\\B_{3a}\end{bmatrix}+C_{3b}B_{3b},$ which implies
\begin{align}
& N^{\frac{1}{2}}G_{1}=N^{\frac{1}{2}}\begin{bmatrix}
C_2&C_{3a}\end{bmatrix}\begin{bmatrix}
B_2\\B_{3a}\end{bmatrix}+N^{\frac{1}{2}}C_{3b}B_{3b},\notag\\
&\Rightarrow N^{\frac{1}{2}}G_{1}=N^{\frac{1}{2}}C_{3b}B_{3b}\text{
since }N^{\frac{1}{2}}\begin{bmatrix}
C_2&C_{3a}\end{bmatrix}=0,\notag\\
&\Rightarrow N^{\frac{1}{2}}G_{1}J=N^{\frac{1}{2}}C_{3b}X^{-1}J^{T}J,\notag\\
&\Rightarrow
N^{\frac{1}{2}}G_{1}J(J^{T}J)^{-1}=N^{\frac{1}{2}}C_{3b}X^{-1}\label{C3b_2}.
\end{align}
Substituting \eqref{Jour-f_G0}, \eqref{P3b_2} and  \eqref{C3b_2}
into condition \eqref{Eq:Con2:th2}  in Theorem  \ref{Pre_main_re},
it follows that this condition can be rewritten as
\begin{align}I-&\left(
N^{\frac{1}{2}}G_{0}N^{\frac{1}{2}}+N^{\frac{1}{2}}G_{1}J(J^{T}J)^{-1}(J^{T}J)^{-T}J^{T}G_{1}^{T}N^{\frac{1}{2}}\right)
\notag\\ &>0.\label{cond2_m}
\end{align}

Similarly, in the case when  $N$ is negative semidefinite, it
follows that $\tilde{N}
\begin{bmatrix} C_2&C_{3a}\end{bmatrix}=0.$ This implies that
\begin{align} &
\tilde{N}G_{1}=\tilde{N}\begin{bmatrix}
C_2&C_{3a}\end{bmatrix}\begin{bmatrix}
B_2\\B_{3a}\end{bmatrix}+\tilde{N}C_{3b}B_{3b},\notag\\
&\Rightarrow \tilde{N}G_{1}=\tilde{N}C_{3b}B_{3b}\text{ since
}\tilde{N}\begin{bmatrix}
C_2&C_{3a}\end{bmatrix}=0,\notag\\
&\Rightarrow \tilde{N}G_{1}J=\tilde{N}C_{3b}X^{-1}J^{T}J,\notag\\
&\Rightarrow
\tilde{N}G_{1}J(J^{T}J)^{-1}=\tilde{N}C_{3b}X^{-1}\label{C3b_2-n-delata}.
\end{align}
Substituting \eqref{Jour-f_G0}, \eqref{P3b_2} and
\eqref{C3b_2-n-delata} into condition \eqref{Eq:Con3:th2}  in
Theorem  \ref{Pre_main_re}, it follows that this condition can be
rewritten as
\begin{align}&\det\left(I+\left(
\tilde{N}G_{0}\tilde{N}+\tilde{N}G_{1}J(J^{T}J)^{-1}(J^{T}J)^{-T}J^{T}G_{1}^{T}\tilde{N}\right)\right)
\notag\\ &\neq0.\label{cond2_m-ndelta}
\end{align}

%
% Substituting
%\eqref{Jour-f_G0}, \eqref{P3b_2} and  \eqref{C3b_2} into Condition
%1) in Theorem  \ref{Pre_main_re}, it follows that this condition can
%be rewritten as
%\begin{align}I-\left(
%N^{\frac{1}{2}}G_{0}N^{\frac{1}{2}}+N^{\frac{1}{2}}G_{1}J(J^{T}J)^{-1}(J^{T}J)^{-T}J^{T}G_{1}^{T}N^{\frac{1}{2}}\right)
%>0.\label{cond2_m}
%\end{align}

Now, using Lemma \ref{geees} and substituting for $G_1$ and $G_2$
from \eqref{f_G0} in the Hankel matrix defined in
\eqref{Hankel_m}, it follows that
\begin{align*}%\label{gamath12}
 \Gamma=\begin{bmatrix} G_1 & G_2
\\G_2 &0
\end{bmatrix}=\begin{bmatrix} \tilde{C}
\\ \tilde{C}\tilde{A}
\end{bmatrix}\begin{bmatrix} \tilde{B}
& \tilde{A}\tilde{B}
\end{bmatrix}
\end{align*}
where
\begin{align*}
 &\tilde{A}=\begin{bmatrix}A_2 & 0\\0&A_3
\end{bmatrix}=\begin{bmatrix}0 & 0\\0&A_3
\end{bmatrix};\:\:\ \tilde{B}=\begin{bmatrix} B_2
\\ B_3
\end{bmatrix};\;\;\ \\ &\tilde{C}=\begin{bmatrix} C_2 & C_3
\end{bmatrix}.
\end{align*}
 Using this and the SVD in \eqref{SVD-H-1}, it follows that
 \begin{align*}%\label{gamath13}
 \begin{bmatrix} \tilde{C}
\\ \tilde{C}\tilde{A}
\end{bmatrix}\begin{bmatrix} \tilde{B}
& \tilde{A}\tilde{B}
\end{bmatrix}=\begin{bmatrix} U_1
\\U_2
\end{bmatrix}V^T.
\end{align*}
 Using Lemma \ref{full-rank} and  Lemma \ref{matrix-lemma}, it follows that there exists  a nonsingular matrix $R$ such that
\begin{align}\label{uugama}
 U=\begin{bmatrix} U_1
\\U_2
\end{bmatrix}=\begin{bmatrix} \tilde{C}
\\ \tilde{C}\tilde{A}
\end{bmatrix}R=\begin{bmatrix} \hat{C}
\\ \hat{C}\hat{A}
\end{bmatrix},
\end{align}
where
\begin{align*}
 \hat{C}=\tilde{C}R,\\
\hat{A}=R^{-1}\tilde{A}R.
\end{align*}
This implies that $\hat{A}^2=R^{-1}\tilde{A}^2 R=0$ since
$\tilde{A}^2 =0$.
It follows that $U\hat{A}=\begin{bmatrix} \hat{C}\hat{A}
\\ \hat{C}\hat{A}^2
\end{bmatrix}=\begin{bmatrix} \hat{C}\hat{A}
\\ 0
\end{bmatrix}=\begin{bmatrix} U_2
\\ 0
\end{bmatrix}$, which implies
\begin{align*}
& \hat{A}=U^TU\hat{A}=U^T\begin{bmatrix} U_2
\\ 0\end{bmatrix}=U_{1}^TU_2.%\label{A-hat}
\end{align*}
Using this and \eqref{SVD-uu}, it follows that
\begin{align}\label{span:A:hat}
\mathcal{N}(\hat{A})=span\{\hat{V_2}\}.
\end{align}
 Also,
since $\tilde{A}=\begin{bmatrix}A_2 & 0\\0&A_3
\end{bmatrix}=\begin{bmatrix}0&0 & 0\\0&0&I\\0&0 & 0
\end{bmatrix}$, it follows that
\begin{align}\label{span:A:delta}
\mathcal{N}(\tilde{A})=span\{\begin{bmatrix} I&0
\\ 0&I\\0&0\end{bmatrix}\}.
\end{align}
Now observe  that we can write the matrix $\begin{bmatrix}
C_2&C_{3a}\end{bmatrix}$ as
\begin{align}
\begin{bmatrix} C_2&C_{3a}\end{bmatrix}=\tilde{C}\begin{bmatrix} I&0
\\ 0&I\\0&0\end{bmatrix}\label{c_2C_3a}.
\end{align}
%Now, consider the SVD of the matrix $\hat{A}$ as following
%\begin{align}
%\hat{A}=\hat{U} \hat{S}\hat{V}^T=\hat{U}\begin{bmatrix}S_1&0
%\\ 0&0
%\end{bmatrix}\begin{bmatrix} \hat{V_1}^T
%\\ \hat{V_2}^T
%\end{bmatrix},
%\end{align}
%where $S_1>0$, and $\hat{V}$ is orthogonal matrix.
 Also, observe that $\hat{A}x=0 $  if and only if
\begin{align*}
&  R^{-1}\tilde{A}Rx=0 \\
 \Leftrightarrow& \tilde{A}Rx=0 \\
 \Leftrightarrow& Rx\in \mathcal{N}(\tilde{A}).
\end{align*}
 Hence, $\mathcal{N}(\tilde{A})=R\mathcal{N}(\hat{A})$. Therefore it follows from \eqref{span:A:hat} and  \eqref{span:A:delta} that
 \begin{align*}R\; span\{\hat{V_2}\}=span
\{R\hat{V_2}\} =span\{\begin{bmatrix} I&0
\\ 0&I\\0&0\end{bmatrix}\}.
 \end{align*}
This implies that there exists a  nonsingular matrix $\hat{R}$ such
that
\begin{align*}
\begin{bmatrix} I&0
\\ 0&I\\0&0\end{bmatrix}=R \hat{V_2} \hat{R}.%\label{span_1}
\end{align*}
 Substituting  this into \eqref{c_2C_3a} and using \eqref{uugama} implies
 \begin{align}\begin{bmatrix}
C_2&C_{3a}\end{bmatrix}=\tilde{C}R\hat{V_2}\hat{R}=\hat{C}\hat{V_2}\hat{R}=U_{1}\hat{V_2}\hat{R}=F\hat{R}\label{C2_C3a_3},
\end{align}
where $F=U_{1}\hat{V_2}$ as in \eqref{F-SVD1}. Substituting
\eqref{C2_C3a_3}  into the matrix \eqref{N-g-lemma-1} and
\eqref{segma-N} implies that
\begin{align}\label{Nf_matrix}
 N&=\left( \bar{G}%
(0)-\bar{G}(0)%
F\hat{R}
(\hat{R}^TF^T\bar{G}(0)F\hat{R})^{-1}%
\hat{R}^TF^T%
\bar{G}(0)\right)\notag\\&=\bar{G}(0)-\bar{G}(0)F\left( F^{T}\bar{G}(0)F\right) ^{-1}F^{T}%
\bar{G}(0)\notag\\&=N_f,
\end{align} where $N_f$ is defined as into \eqref{N-f-m}.
Substituting  \eqref{Nf_matrix} in  \eqref{cond2_m} and
\eqref{cond2_m-ndelta} implies that conditions \eqref{Eq:Con2:th2}
and \eqref{Eq:Con3:th2} in Theorem \ref{Pre_main_re}  are equivalent
to conditions \eqref{Eq:Con2:th1} and \eqref{Eq:Con3:th1} in the
theorem respectively.

Also, \eqref{C2_C3a_3} implies   that
\begin{align*}\Xi=
\begin{bmatrix}
C_{2}^{T} \\
C_{3a}^{T}
\end{bmatrix}
\bar{G}(0)
\begin{bmatrix}
C_{2} & C_{3a}
\end{bmatrix}=\hat{R}^TF^T\bar{G}(0)F\hat{R}.
\end{align*}
It follows that condition \eqref{Eq:Con1:th2} in  Theorem
\ref{Pre_main_re} is equivalent to condition \eqref{Eq:Con1:th1} in
the theorem since $\hat{R}$ is invertible. This completes the proof
of the theorem. \hfill $\blacksquare$
%\begin{flushright}
%$\blacksquare$
%\end{flushright}

%=======================
%=======================
%=======================
%=proof of Corollary 1==
%=======================
%=======================

\emph{Proof  of Corollary \ref{min:result:clo3}:} In order to prove
this corollary, we show that the stability conditions and the
assumptions in Corollary \ref{min:result:clo2-jour} are equivalent
to the stability conditions and the assumptions in this corollary.
First, it is straightforward to verify that the conditions $k\neq0$
and $n_2=0$ are equivalent to the conditions  $G_2\neq0$ and
$G_1=0$. Also, using \eqref{Jour-f_G0} and the decomposition in
\eqref{JJ-m}, it follows that $C_{3a}B_{3b}=JJ^T$, and hence Lemma
\ref{matrix-lemma} implies that there exist an invertible matrix $X$
such that $C_{3a}=JX$. This implies that the matrix $N_2$ in
\eqref{N-2-g10} is equal  to the matrix $N$ in \eqref{N-g-croll-2}.
Also, since $C_{3a}=JX$ and $X$ is invertible, it follows that
condition \eqref{Eq:Con1:crol3} in Corollary \ref{min:result:clo3}
is equivalent to condition \eqref{Eq:Con1:crol7} in Corollary
\ref{min:result:clo2-jour}. Since $G_1=0$, it follows that
\begin{align}\label{nc3b-f}
  &NG_1= N(C_{3a}B_{3a}+C_{3b}B_{3b})=0, \notag\\
  &\Rightarrow NC_{3b}B_{3b}=0, \text{ since } NC_{3a}=0, \notag\\
  &\Rightarrow NC_{3b}=0, \text{ since } B_{3b} \text{ is of full rank}.
\end{align}
 This implies that $N^{\frac{1}{2}}C_{3b}=0$ in the case when $N$ is positive semidefinite.
Using the fact that $G_0=-C_1A_{1}^{-1}B_1$ from  Lemma \ref{geees},
it follows that condition \eqref{Eq:Con2:crol3} in Corollary
\ref{min:result:clo3} is equivalent to condition
\eqref{Eq:Con2:crol7} in Corollary \ref{min:result:clo2-jour}. Also,
in the case when $N$ is negative  semidefinite \eqref{nc3b-f}
implies that $\tilde{N}C_{3b}=0$. Using the fact that
$G_0=-C_1A_{1}^{-1}B_1$ from  Lemma \ref{geees}, it follows that
condition \eqref{Eq:Con3:crol3} in Corollary \ref{min:result:clo3}
is equivalent to condition \eqref{Eq:Con3:crol7} in Corollary
\ref{min:result:clo2-jour}. This completes the proof of the
corollary. \hfill $\blacksquare$

%Now, we present the proofs  of Corollaries
%\ref{min:result:clo1}-\ref{min:result:clo5}.

%=======================
%=======================
%=======================
%=proof of Corollary 2==
%=======================
%=======================

\emph{Proof  of Theorem \ref{min:result:clo3.1}:} In order to prove
this theorem, we first  show  that
$\mathcal{N}(G_2)\subseteq\mathcal{N}(G_0^T)$ implies the condition
$N\begin{bmatrix} C_{1} & C_{3b}
\end{bmatrix}=0$ in Corollary \ref{Pre_main_re-G1-G2-0}, in
the case when $G_1=0$. Indeed, suppose that
$\mathcal{N}(G_2)\subseteq\mathcal{N}(G_0^T)$. This  implies that
\begin{align*}%\label{eq:rang-g2}
 \mathcal{R}(G_2)\supseteq\mathcal{R}(G_0)
\end{align*}
where  $\mathcal{R}(\cdot)$ denotes the range space of a matrix.
Since  $G_2=JJ^T$  and $J$ is of full rank, it follows that
\begin{align*}%\label{eq:rang-g2-2}
  \mathcal{R}(JJ^T)\supseteq\mathcal{R}(G_0),
\end{align*}
which implies that there exist a matrix $Q$ such that $G_0=JQ$.
Then, we consider the matrix $N$ defined as
\begin{align*}%\label{N2-g2-croll-1}
N=\bar{G}(0)-\bar{G}(0)C_{3a}(C_{3a}^{T}\bar{G}(0)C_{3a})^{-1}C_{3a}^{T}\bar{G}(0),
 \end{align*}
which is the formula for the matrix $N$ in Corollary
\ref{Pre_main_re-G1-G2-0} in the case in which $G_1=0$. This implies
that
\begin{align*}
  &NG_0=NJQ=0, \text{ since } NJ=0,
  \end{align*}
  and hence from Lemma \ref{geees}, it follows that
  \begin{align*}%\label{eq:rang-g2-3}
  & NC_1A_{1}^{-1}B_1=0.
  \end{align*}
  Using a similar calculation as in equation   \eqref{B1_m} in the proof of Theorem \ref{Pre_main_re},  this %\eqref{eq:rang-g2-3}
 implies that
  \begin{align}\label{eq:g2-C1}
  &NC_1A_{1}^{-1}(A_{1}P_{1}^{-1}C_{1}^{T}-P_{1}^{-1}L_{1}^{T}M)N=0,\notag \\
  &\Rightarrow NC_1P_{1}^{-1}C_{1}^{T}N=0 \text{ using \eqref{eq:NM11}},\notag \\
  &\Rightarrow NC_1=0.
\end{align}
Also, since $G_1= C_{3}B_{3}=C_{3a}B_{3a}+C_{3b}B_{3b}=0$, it
follows that
\begin{align}\label{eq:g2-Cb3}
  &NC_{3a}B_{3a}+NC_{3b}B_{3b}=0,\notag \\
  &\Rightarrow NC_{3b}B_{3b}=0, \text{ since  } NC_{3a}=0, \notag \\
  &\Rightarrow NC_{3b}=0 \text{ since  } B_{3b} \text{ is of full rank. }
  \end{align}
Using \eqref{eq:g2-C1} and  \eqref{eq:g2-Cb3}, it follows that
$N\begin{bmatrix} C_{1} & C_{3b}
\end{bmatrix}=0$. This implies  the assumptions
in Corollary \ref{Pre_main_re-G1-G2-0} are satisfied  in the case
when $G_1=0$. Also, as in the proof of
 Theorem \ref{min:result}, the condition \eqref{Eq:Con1:crol3} reduces to condition \eqref{Eq:Con1:th2} in Corollary \ref{Pre_main_re-G1-G2-0}. \hfill $\blacksquare$
%\emph{Proof  of Corollary \ref{min:result:clo3.1}:} ??????
%Since $\mathcal{N}(G_2)\subseteq\mathcal{N}(G_0^T)$ and using the fact that $N_2G_2=0$ it follows that $N_2G_0=0$. This implies that conditions   \eqref{Eq:Con2:crol3} and  \eqref{Eq:Con3:crol3} in Corollary \ref{min:result:clo3} are automatically satisfied. Hence, proof  of Corollary \ref{min:result:clo3.1} follows using Corollary \ref{min:result:clo3}. \hfill $\blacksquare$
%%, it follows that the closed-loop positive-feedback interconnection between $G(s)$ and $\bar{G}(s)$  is internally stable   if and only if conditions \eqref{Eq:Con1:crol3}  is satisfied.
%%
%

%=======================
%=======================
%=======================
%=proof of Corollary 3==
%=======================
%=======================

%\emph{Proof  of Corollary \ref{min:result:clo4}:}
% Since $G_2$ is assumed to be positive definite  in this
%corollary, it follows that the matrix $J$ in \eqref{JJ-m} is
%invertible. Then, the condition \eqref{Eq:Con1:crol1} reduces to the
%condition $\bar{G}(0)<0$ and the corollary follows immediately from
%Corollary \ref{min:result:clo3.1}. \hfill $\blacksquare$

%=======================
%=======================
%=======================
%=proof of Theorem 2==
%=======================
%=======================

\emph{Proof  of Theorem \ref{min:result:clo1}:} In order to prove
this theorem, we show that the stability conditions and the
assumptions in this theorem  are equivalent to the stability
conditions and the assumptions in Corollary
\ref{min:result:clo1-jour}. First, it is straightforward to verify
that the conditions  $n_2\neq0$ and $k=0$ are equivalent to the
conditions  $G_1\neq0$ and $G_2=0$. Using Lemma \ref{matrix-lemma}
and the fact that $G_1=C_2B_2$ from Lemma \ref{geees}, it follows
that there exists an invertible matrix $R$ such that $C_2=F_1R$,
where the matrix $F_1$ is given in \eqref{C-D-1-f_G1}. This implies
that the matrix $N_1$ in \eqref{N-1-g20} is equal  to the  matrix
$N$ in \eqref{N-g-croll-1}. Also,  since $C_2=F_1R$ and $R$ is
invertible, it follows that condition \eqref{Eq:Con1:crol1} in this
theorem   is equivalent to condition \eqref{Eq:Con1:crol6} in
Corollary \ref{min:result:clo1-jour}. Finally, using the fact that
$G_0=-C_1A_{1}^{-1}B_1$ from Lemma \ref{geees},  it follows that
conditions \eqref{Eq:Con2:crol1} and \eqref{Eq:Con3:crol1} in this
theorem are equivalent to conditions \eqref{Eq:Con2:crol6} and
\eqref{Eq:Con3:crol6} in Corollary \ref{min:result:clo1-jour}
respectively. This completes the proof of the theorem. \hfill
$\blacksquare$
%\begin{flushright}
%$\blacksquare$
%\end{flushright}
%\end{proofc1}

%=======================
%=======================
%=======================
%=proof of Corollary 2222..11111111111111==
%=======================
%=======================

\emph{Proof  of Theorem \ref{min:result:clo2.1}:} In order to prove
this theorem, we first  show  that
$\mathcal{N}(G_1^T)\subseteq\mathcal{N}(G_0^T)$ implies the
condition  $N\begin{bmatrix} C_{1} & C_{3b}
\end{bmatrix}=0$ in Corollary \ref{Pre_main_re-G1-G2-0}, in the case when $G_2=0$. Indeed, suppose that $\mathcal{N}(G_1^T)\subseteq\mathcal{N}(G_0^T)$. This  implies that
\begin{align}\label{eq:rang-g1}
  \mathcal{R}(G_1)\supseteq\mathcal{R}(G_0).
\end{align}
Since  $G_1=C_2B_2$ from Lemma \ref{geees} and $B_2$ is of full rank
using Lemma \ref{full-rank}, it follows that
$\mathcal{R}(C_2)=\mathcal{R}(G_1)$. Using \eqref{eq:rang-g1}, it
follows that
\begin{align*}%\label{eq:rang-g1-2}
  \mathcal{R}(C_2)\supseteq\mathcal{R}(G_0),
\end{align*}
which implies that there exists a matrix $Q$ such that $G_0=C_2Q$.
Then, we consider the matrix $N$ defined as
\begin{align*}%\label{N2-g-croll-1}
N=\bar{G}(0)-\bar{G}(0)C_{2}(C_{2}^{T}\bar{G}(0)C_{2})^{-1}C_{2}^{T}\bar{G}(0),
 \end{align*}
which is the formula for the matrix $N$ in Corollary
\ref{Pre_main_re-G1-G2-0} for the case in which $G_2=0$. This
implies that
\begin{align*}%\label{eq:rang3-g1-3}
  &NG_0=NC_2Q=0,\text{ since }NC_2=0
  \end{align*}
  and hence from Lemma \ref{geees} it follows that
  \begin{align*}%\label{eq:rang-g1-3}
  & NC_1A_{1}^{-1}B_1=0.
  \end{align*}
  Using a similar calculation as in equation   \eqref{B1_m} in the proof of Theorem \ref{Pre_main_re}, this %\eqref{eq:rang-g1-3}
implies
  \begin{align}
  &NC_1A_{1}^{-1}(A_{1}P_{1}^{-1}C_{1}^{T}-P_{1}^{-1}L_{1}^{T}M)N=0,\notag \\
  &\Rightarrow NC_1P_{1}^{-1}C_{1}^{T}N=0, \text{ using \eqref{eq:NM11}}\notag \\
  &\Rightarrow NC_1=0\notag.
\end{align}
This implies the assumptions  in Corollary \ref{Pre_main_re-G1-G2-0}
are  satisfied  in the case when $G_2=0$. Also, as in the proof of
Theorem \ref{min:result},  condition \eqref{Eq:Con1:crol1} reduces
to condition \eqref{Eq:Con1:th2} in Corollary
\ref{Pre_main_re-G1-G2-0}. \hfill $\blacksquare$

%Since $\mathcal{N}(G_2)\subseteq\mathcal{N}(G_0^T)$ and using the fact that $N_2G_2=0$ it follows that $N_2G_0=0$. This implies that conditions   \eqref{Eq:Con2:crol3} and  \eqref{Eq:Con3:crol3} in Corollary \ref{min:result:clo3} are automatically satisfied. Hence, proof  of Corollary \ref{min:result:clo3.1} follows using Corollary \ref{min:result:clo3}. \hfill $\blacksquare$
%%, it follows that the closed-loop positive-feedback interconnection between $G(s)$ and $\bar{G}(s)$  is internally stable   if and only if conditions \eqref{Eq:Con1:crol3}  is satisfied.
%%
%

%=======================
%=======================
%=======================
%=proof of Corollary 2222.22222=
%=======================
%=======================

\emph{Proof  of Corollary \ref{min:result:clo2}:} In the case where
$G_1$ is assumed to be invertible in this corollary, it follows that
the matrix $F_1$ in \eqref{C-D-1-f_G1} is invertible. Then, the
condition \eqref{Eq:Con1:crol1} reduces to the condition
$\bar{G}(0)<0$ and the corollary follows immediately from Corollary
\ref{min:result:clo2.1}. In the case where  $G_2$ is assumed to be
positive definite in this corollary, it follows that the matrix $J$
in \eqref{JJ-m} is invertible. Then, the condition
\eqref{Eq:Con1:crol1} reduces to the condition $\bar{G}(0)<0$ and
the corollary follows immediately from Corollary
\ref{min:result:clo3.1}. \hfill $\blacksquare$

\end{document}